\documentclass[11pt]{amsart}
\usepackage{amsmath,amsthm,amsfonts,amscd,amssymb,eucal,latexsym,mathrsfs}
\usepackage[all]{xy}
\newtheorem{theorem}{Theorem}[section]

\newtheorem{lemma}[theorem]{Lemma}
\newtheorem{proposition}[theorem]{Proposition}

\theoremstyle{definition}
\newtheorem{definition}[theorem]{Definition}
\newtheorem{remark}[theorem]{Remark}

\newtheorem{example}[theorem]{Example}

\newtheorem{problem}[theorem]{Problem}

\theoremstyle{plain}

\theoremstyle{definition}

\theoremstyle{remark}

\newcommand{\spn}{{\rm span}}

\newcommand{\cL}{{\mathcal L}}

\newcommand{\cT}{{\mathcal T}}
\newcommand{\cQ}{{\mathcal Q}}
\newcommand{\cP}{{\mathcal P}}

\newcommand{\cS}{{\mathcal S}}

\newcommand{\Cb}{{\mathbb C}}
\newcommand{\Eb}{{\mathbb E}}
\newcommand{\Zb}{{\mathbb Z}}

\newcommand{\Rb}{{\mathbb R}}
\newcommand{\Nb}{{\mathbb N}}

\newcommand{\supp}{{\rm supp}}

\newcommand{\tr}{{\rm tr}}

\newcommand{\UP}{{\rm UP}}

\newcommand{\Sym}{{\rm Sym}}
\newcommand{\AP}{{\rm AP}}

\newcommand{\Hom}{{\rm Hom}}

\newcommand{\ch}{{\bar{h}}}
\newcommand{\ddet}{{\rm det}}
\newcommand{\Fix}{{\rm Fix}}
\newcommand{\mmod}{\hspace*{1.2mm}{\rm mod}\hspace*{1.2mm}}

\allowdisplaybreaks

\begin{document}

\title[Entropy and the variational principle]{Entropy and the variational principle for actions of sofic groups}

\author{David Kerr}
\author{Hanfeng Li}
\address{\hskip-\parindent
David Kerr, Department of Mathematics, Texas A{\&}M University,
College Station TX 77843-3368, U.S.A.}
\email{kerr@math.tamu.edu}

\address{\hskip-\parindent
Hanfeng Li, Department of Mathematics, SUNY at Buffalo,
Buffalo NY 14260-2900, U.S.A.}
\email{hfli@math.buffalo.edu}

\date{February 21, 2011}

\begin{abstract}
Recently Lewis Bowen introduced a notion of entropy for measure-preserving actions
of a countable sofic group on a standard probability space admitting a
generating partition with finite entropy. By applying an operator algebra perspective we develop a
more general approach to sofic entropy which produces both measure and topological dynamical
invariants, and we establish the variational principle in this context.
In the case of residually finite groups we use the variational principle to compute the topological entropy
of principal algebraic actions whose defining group ring element is invertible in the full group $C^*$-algebra.
\end{abstract}

\maketitle

\section{Introduction}

Recently Lewis Bowen introduced a collection of entropy invariants for measure-preserving actions
of a countable sofic group on a standard probability space admitting a
generating partition with finite entropy \cite{Bowen10}. The basic idea is to model
the dynamics of a measurable partition of the probability space by means of partitions of a finite space
on which the group acts in a local and approximate way according to the definition of soficity.
The cardinality of the set of all such model partitions is then used to asymptotically generate a
number along a fixed sequence of sofic approximations. This quantity is then shown to be invariant over
all generating measurable partitions with finite entropy. It might however depend on the
choice of sofic approximation sequence, yielding in general a collection of invariants.
A major application of this sofic measure entropy was the extension of the Ornstein-Weiss
entropy classification of Bernoulli shifts over countably infinite amenable groups
to a large class of nonamenable groups, including all nontorsion countable sofic groups \cite{Bowen10}.


Given Bowen's work, it is natural to ask whether there exist analogous
invariants for continuous actions of a countable sofic group on a compact metrizable space, and if
so whether they are connected to Bowen's measure entropy via a variational principle. One might also wonder
whether there exists an alternative approach to sofic measure entropy
that enables one to extend Bowen's invariants to actions that are not
generated by a partition with finite entropy. Such a general notion of sofic measure entropy
would be not only valuable from a purely measure-dynamical viewpoint but also
necessary for the formulation of a variational principle for topological systems.

The goal of this paper is to provide affirmative answers to all of these questions.
The key is to view the dynamics at the
operator algebra level and replace the combinatorics of partitions with an analysis of multiplicative or
approximately multiplicative linear maps that are approximately equivariant.
As a consequence our definitions of topological and measure entropy will not
involve the counting of partitions but rather the computation of the maximal cardinality of
$\varepsilon$-separated subsets of certain spaces of linear maps, in the spirit of Rufus Bowen's
approach to topological entropy for $\Zb$-actions \cite{Bow71}. In fact our definitions can be
translated into the language of $\varepsilon$-separation between embedded sofic approximations, which
can be viewed as systems of interlocking approximate partial orbits (see Remark~\ref{R-embedded sofic approx}),
but we will adhere throughout to the linear perspective since it is
instrumental to our development of measure entropy.

It is instructive to compare the situation of sofic measure entropy with the origins
of entropy for single measure-preserving transformations in the work of Kolmogorov and Sinai \cite{Kol58,Sin59}.
Kolmogorov showed that all dynamically generating partitions for a given transformation have the same entropy, and
used this to define the entropy of the system when such a partition exists, assigning the value $\infty$ otherwise.
Sinai then proposed the now standard definition which takes the supremum of the entropies over
all partitions. This gives reasonable values in the absence of a generating partition, in particular
for the identity transformation, and agrees with Kolmogorov's definition when a generating partition exists.
Lewis Bowen's sofic measure entropy is based, in the spirit of Kolmogorov, on the comparison of
generating partitions with finite entropy, and leaves open the problem of assigning a value
in the absence of such a partition. In this case however one cannot extend the definition by taking a
supremum as Sinai did, since Bowen's entropy can increase under taking factors, in particular
for Bernoulli actions of free groups. Thus a novel
strategy is required, and our idea is to apply the notion of dynamical generator in the broader operator-algebraic
context of finite sets of $L^\infty_\Rb$ functions and even bounded sequences of such functions. Then every action admits
a dynamical generator, and we show that the entropy as we define it
takes a common value on such generators, in accord with the approaches of
Kolmogorov and Bowen. Since we are no longer working with partitions, Bowen's combinatorial
arguments must be replaced by a completely different type of analysis that plays off
the operator and Hilbert space norms at the function level. The point in using functions is that
a continuous spectrum can witness dynamical behaviour at arbitrarily fine scales, in contrast to the fixed
scale of a partition. In fact one can in principle compute our sofic measure entropy by means of a single
function, since $L^\infty$ over a standard probability space is itself singly generated as a von Neumann algebra.
However, for the proof of the variational principle it is necessary to work with bounded sequences of functions,
since not all topological systems are finitely generated in the $C^*$-dynamical sense.

We begin in Section~\ref{S-measure} by setting up our operator-algebraic definition of entropy
for measure-preserving actions,
which at the local level applies to any bounded sequence in $L^\infty_\Rb$ over the measure space in question.
For technical simplicity we will actually work with sequences in the unit ball of $L^\infty_\Rb$, which
via scaling does not affect the scope of the definition.
Theorem~\ref{T-gen comparison} asserts that two such sequences that are dynamically generating
have the same entropy relative to a fixed sofic approximation sequence,
which enables us to define the global measure entropy of the system
without the assumption of a generating partition with finite entropy. Section~\ref{S-comparison}
is devoted to establishing the equality with Bowen's entropy in the presence of a generating partition
with finite entropy.
Extending a computation from \cite{Bowen10}
in the finite entropy setting, we show in a separate paper that, for a countable sofic group,
a Bernoulli action with infinite entropy base has infinite entropy with respect
to every sofic approximation sequence \cite{KerLi10B}.
As a consequence, such Bernoulli actions do not admit a generating countable partition with finite entropy,
which in the amenable case is well known and in the case that the acting group contains the
free group on two generators was established by Bowen in \cite{Bowen10}.

Once we have set up the measurable framework we then translate everything into topological terms,
with locality now referring to sequences in the unit ball of the $C^*$-algebra of continuous functions
over the compact space in question (Section~\ref{S-top}). The arguments in this case are much
simpler since one can
work with unital homomorphisms and
does not need to worry about controlling an $L^2$-norm under perturbations,
which is the source of considerable technical complications in the measurable setting
(cf.\ Proposition~\ref{P-bounded}).
As before, two dynamically generating sequences have the same entropy (Theorem~\ref{T-gen comparison top}),
and since dynamically generating sequences always exist by metrizability we thereby obtain
a conjugacy invariant. For a topological Bernoulli action the value of this invariant is easily
computed to be the logarithm of the cardinality of the base. We also show at the end of Section~\ref{S-top}
that the restriction of a topological Bernoulli action to a proper closed invariant subset has strictly
smaller entropy. This yields an entropy proof of Gromov's result that countable sofic groups are surjunctive
\cite{Gro99} (see also \cite{Wei00}) in line with what Gromov observed in the case of amenable groups
using classical entropy.

In order to facilitate the comparison with topological entropy in Sections~\ref{S-variational}
and \ref{S-algebraic}, we show in Section~\ref{S-Hom} how to express measure entropy in terms
of unital homomorphisms instead of linear maps which are merely approximately multiplicative.
In Section~\ref{S-variational} we establish the variational principle, which asserts
that, with respect to a fixed sofic approximation sequence, the topological entropy of a continuous action on
a compact metrizable space is equal to the supremum of the measure entropies over all invariant
Borel probability measures.

Finally in Section~\ref{S-algebraic} we give an application of the variational principle to the
study of algebraic actions of a residually finite group $G$ that complements
a recent result of Lewis Bowen \cite{Bowen10b}.
Given an element $f$ in the integral group ring $\Zb G$ which is invertible in the full group $C^*$-algebra
of $G$, we show that the topological entropy of the canonical action of $G$ on $\widehat{\Zb G / \Zb G f}$,
with respect to any sofic approximation sequence arising from finite quotients of $G$, is equal to the
logarithm of the Fuglede-Kadison determinant of $f$ as an element in the group von Neumann algebra of $G$.
In \cite{Bowen10b} Bowen established the same result for measure entropy with respect to the normalized Haar measure
under the assumption that $f$ is invertible in $\ell^1 (G)$. In the case of amenable acting groups and
classical entropy these relationships were developed in
\cite{LinSchWar90,Den06,DenSch07,Li10}.

In \cite{Bowen10a} Bowen showed that, when the acting group is amenable and there exists a
generating finite measurable partition, the sofic measure entropy
as defined in \cite{Bowen10} is equal to the classical Kolmogorov-Sinai measure entropy, independently of the sofic
approximation sequence.
In \cite{KerLi10} we will show
that, for any measure-preserving action of a countable amenable group on a standard probability space,
the sofic measure entropy defined in Section~\ref{S-measure} agrees with its classical counterpart,
independently of the sofic approximation sequence. It follows by the variational principle of Section~\ref{S-variational}
and the classical variational principle that, for a continuous action of a countable amenable group on a
compact metrizable space, the sofic topological entropy with respect to any sofic approximation sequence
is equal to the classical topological entropy, which for $\Zb$-actions was introduced in \cite{AdlKonMcA65}.
We will also give in \cite{KerLi10} a direct argument for this equality
which sheds some more light on the sofic definition.

We round out the introduction with some terminology, conventions, and notation used in the paper,
in particular regarding sofic groups and unital commutative $C^*$-algebras.
Write $\Sym (F)$ for the group of permutations of a set $F$, or simply $\Sym (d)$ when $F = \{ 1,\dots ,d \}$.
Let $G$ be a countable discrete group. We write $e$ for its identity element.
We say that $G$ is {\it sofic} if for $i\in\Nb$ there
are a sequence $\{ d_i \}_{i=1}^\infty$ of positive integers and a sequence
$\{ \sigma_i \}_{i=1}^\infty$ of maps $s\mapsto\sigma_{i,s}$ from $G$ to $\Sym (d_i )$
which is asymptotically multiplicative and free in the sense that
\[
\lim_{i\to\infty}
\frac{1}{d_i} \big| \{ k\in \{ 1,\dots ,d_i \} : \sigma_{i,st} (k) = \sigma_{i,s} \sigma_{i,t} (k) \} \big| = 1
\]
for all $s,t\in G$ and
\[
\lim_{i\to\infty}
\frac{1}{d_i} \big| \{ k\in \{ 1,\dots ,d_i \} : \sigma_{i,s} (k) \neq \sigma_{i,t} (k) \} \big| = 1
\]
for all distinct $s,t\in G$. Such a sequence $\{ \sigma_i \}_{i=1}^\infty$ for which
$\lim_{i\to\infty} d_i = \infty$ will be called a {\it sofic approximation sequence} for $G$.
We include the condition $\lim_{i\to\infty} d_i = \infty$ as it is crucial for certain
results in the paper (in particular for the variational principle), and note that
it is automatic if $G$ is infinite.
Throughout the paper the notation $\Sigma = \{ \sigma_i : G\to\Sym (d_i ) \}_{i=1}^\infty$ will be tacitly
understood to refer to a fixed sofic approximation sequence which is arbitrary unless otherwise indicated.

All function spaces will be over the complex numbers, unless the notation is tagged with the subscript $\Rb$,
in which case we mean the real subspace of real-valued functions.
The unital commutative $C^*$-algebras that will be encountered in this paper
are function spaces of the form $L^\infty (X,\mu )$ for a standard probability space $(X,\mu )$
(these are the commutative von Neumann algebras with separable predual),
$C(X)$ for a compact metrizable space $X$, and $\Cb^d$ for $d\in\Nb$, which can also be viewed as $C(X)$ where
$X=\{ 1,\dots ,d \}$. The norm on these $C^*$-algebras will be written
$\| \cdot \|_\infty$.
The adjoint in each of these cases is given by pointwise complex conjugation,
and following general $C^*$-algebra convention we will write the adjoint of an element $f$ by $f^*$.
A $^*$-subalgebra of a $C^*$-algebra is a subalgebra which is closed under taking adjoints.
A linear subspace of a $C^*$-algebra is said to be self-adjoint if it is closed under taking adjoints.
A projection in a $C^*$-algebra is an element $p$ satisfying $p^2 = p$ and $p^* = p$.
Via charateristic functions, projections in $C(X)$ correspond to clopen subsets of $X$
while projections in $L^\infty (X,\mu )$ correspond to measurable subsets of $X$ modulo
sets of measure zero.

Throughout we will be working with unital positive linear maps between unital commutative $C^*$-algebras,
or unital self-adjoint subspaces thereof.
A linear map $\varphi :V\to W$ between unital self-adjoint subspaces of unital commutative $C^*$-algebras
is said to be positive if $\varphi (f) \geq 0$ whenever $f\geq 0$ and unital if $\varphi (1) = 1$.
In the case that $\varphi$ is positive its norm $\| \varphi \| = \sup_{\| f \|\leq 1} \| \varphi (f) \|$
is equal to $\| \varphi (1) \|$. In particular $\| \varphi \| = 1$ if $\varphi$ is both unital and positive.
Given unital self-adjoint linear subspaces $V_1 \subseteq V_2$ of a unital commutative $C^*$-algebra
and a $d\in\Nb$, every unital positive linear map $\varphi : V_1 \to \Cb^d$
admits a unital linear extension $\tilde{\varphi} : V_2 \to \Cb^d$ with $\| \tilde{\varphi} \| = 1$
by applying the Hahn-Banach theorem
to each of the $d$ linear functionals obtained by composing $\varphi$ with the coordinate
projections $\Cb^d \to\Cb$. Since $\tilde{\varphi} (1) = 1$ such an extension is automatically
positive (see Section~4.3 of \cite{KadRinI83}).

A unital linear map $\varphi : A\to B$ between unital commutative $C^*$-algebras is
said to be a homomorphism if $\varphi (fg) = \varphi (f)\varphi (g)$ for all $f,g\in A$.
By Gelfand theory every unital commutative $C^*$-algebra is of the form $C(K)$ for some compact
Hausdorff space $K$ which is unique up to homeomorphism
(in the case of $L^\infty (X,\mu )$ this space is extremely disconnected), and every unital
homomorphism $\varphi : C(K_1 )\to C(K_2 )$ where $K_1$ and $K_2$ are compact Hausdorff spaces
is given by composition with a continuous map from $K_2$ to $K_1$. In particular, unital homomorphisms
are positive. See \cite{KadRinI83} for more background on $C^*$-algebras.

For a $d\in\Nb$ we will invariably use $\zeta$ to denote the uniform probability measure on $\{ 1,\dots ,d \}$,
which will be regarded as a state (i.e., a unital positive linear functional) on
$\Cb^d \cong C(\{ 1,\dots ,d \} )$ whenever appropriate.
Given a map $\sigma : G\to\Sym (d)$, we will also use $\sigma$ to denote
the induced action on $\Cb^d \cong C(\{ 1,\dots ,d \} )$, i.e., for $f\in\Cb^d$ and $s\in G$ we will write
$\sigma_s (f)$ to mean $f\circ\sigma_s^{-1}$.

Given a state $\mu$ on a unital commutative
$C^*$-algebra, we will write $\| \cdot \|_2$ for the associated $L^2$-norm $f\mapsto\mu (f^* f)^{1/2}$,
with $\mu$ being understood from the context. In the case of $L^\infty (X,\mu)$ this will always be the
$L^2$-norm with respect to $\mu$, and for $\Cb^d$ it will always be the $L^2$-norm with respect to $\zeta$, i.e.,
$f\mapsto (d^{-1} \sum_{k=1}^d |f(k)|^2 )^{1/2}$.

Actions of a group $G$ on a space $X$ will invariably be denoted by $\alpha$,
although the actual use of this letter will be reserved for the induced action on the appropriate space of
functions over $X$. For the action on $X$ we will simply use the concatenation $(s,x)\mapsto sx$. Thus $\alpha_s (f)$
for $s\in G$ will mean the function $x\mapsto f(s^{-1} x)$.
\medskip

\noindent{\it Acknowledgements.}
The first author was partially supported by NSF
grant DMS-0900938. He is grateful to Lewis Bowen for several seminal discussions. Part of this work was carried
out during a visit of the first author to SUNY at Buffalo in February 2010 and he thanks the
analysis group there for its hospitality. The second author was partially supported by NSF grant DMS-0701414.
We thank the referee for some helpful comments which include the simple proof of Theorem~\ref{T-approximate det}
and the fact that Lemma~\ref{L-fixed point} is valid in the stated generality.

\section{Measure entropy}\label{S-measure}


In this section we will define our notion of entropy
for measure-preserving actions of a countable sofic group, as inspired by Bowen's entropy from \cite{Bowen10}.
We will show in Section~\ref{S-comparison} that the two definitions of measure
entropy agree in the presence of a generating countable measurable partition with finite entropy.


Throughout this section and the next $G$ will be a countable sofic group,
$(X,\mu )$ a standard probability space, and $\alpha$ an action of $G$
by measure-preserving transformations on $X$. As explained in the introduction,
$\alpha$ will actually denote the induced action of $G$ on $L^\infty (X,\mu )$ by automorphisms, so that
for $f\in L^\infty (X,\mu )$ and $s\in g$ the function $\alpha_s (f)$ is given by $x \mapsto f(s^{-1} x)$.

By taking characteristic functions, a measurable partition of $X$ corresponds, modulo sets of measure zero,
to a partition of unity in $L^\infty (X,\mu )$ consisting of projections.
We will abuse notation by using the same symbol to denote both.

The von Neumann subalgebras of $L^\infty (X,\mu )$ are, by Kaplansky's density theorem \cite[Thm.\ 5.3.5]{KadRinI83},
precisely the unital $^*$-subalgebras which are closed in the $L^2$ norm. These correspond,
modulo measure algebra isomorphism, to the measurable factors of $X$
via composition of functions. So the $G$-invariant von Neumann subalgebras
of $L^\infty (X,\mu )$ correspond, modulo measure algebra $G$-isomorphism,
to the dynamical factors of $X$ with respect to the action of $G$.
A set $\Omega\subseteq L^\infty (X,\mu )$ is said to be {\it dynamically generating}
if it is not contained in any proper $G$-invariant von Neumann subalgebra of $L^\infty (X,\mu )$.
When $\Omega$ is a partition of unity consisting of projections this is equivalent to the usual notion of a
generating partition.


Our first goal will be to define
the entropy $h_{\Sigma ,\mu} (\cS )$ of a sequence $\cS$ of elements in the unit ball of $L^{\infty}_{\Rb}(X, \mu)$.
We could similarly define the entropy of an arbitrary subset of the unit ball of $L^{\infty}_{\Rb}(X, \mu)$,
but for the purpose of reducing the number of parameters in the definitions we will use the sequential formalism
(see however the discussion after Definition~\ref{D-global}).
We will show in Theorem~\ref{T-gen comparison} that
$h_{\Sigma ,\mu} (\cS )$ depends only on the $G$-invariant von Neumann subalgebra of $L^\infty (X,\mu )$
generated by $\cS$, so that we can define the global entropy $h_{\Sigma ,\mu} (X,G)$
as the common value of $h_{\Sigma ,\mu} (\cS )$ over all dynamically generating
sequences $\cS$ in the unit ball of $L^{\infty}_{\Rb}(X, \mu)$.

Note that, since $(X,\mu )$ is assumed to be a standard probability space, there always exists a generating
finite partition of unity in $L^{\infty} (X, \mu)$. Indeed we can identify $(X,\mu )$ with a subset of
$[0,1]$ consisting of a subinterval with Lebesgue measure and countably many atoms and take our partition
of unity to consist of the functions $x\mapsto x$ and $x\mapsto 1-x$. Thus for the purpose of defining
global measure entropy we could instead simply work with finite partitions
of unity in $L^\infty_\Rb (X,\mu )$. However, the use of sequences
is necessary in order to establish the variational principle (Theorem~\ref{T-variational}) due to the
fact that continuous actions on compact metrizable spaces need not be finitely generated at the function level.

Let $\sigma$ be a map from $G$ to $\Sym (d)$ for some $d\in\Nb$.
Let $\cS=\{p_n\}_{n\in \Nb}$ be a sequence of elements in the unit ball of $L^{\infty}_{\Rb}(X, \mu)$
(with respect to the $L^{\infty}$-norm).
Let $F$ be a nonempty finite subset of $G$
and $m\in\Nb$. We write $\cS_{F,m}$ for the set of all
products of the form $\alpha_{s_1} (f_1 ) \cdots \alpha_{s_j} (f_j )$ where $1\le j\le m$ and
$f_1 , \dots ,f_j \in \{p_1, \dots, p_m\}$ and $s_1 , \dots ,s_j \in F$.
On the set of unital positive linear maps from some self-adjoint unital linear subspace of $L^\infty (X,\mu )$
containing $\spn (\cS )$ to $\Cb^d$ we define the pseudometric
\begin{align*}
\rho_\cS (\varphi , \psi ) &= \sum_{n=1}^\infty \frac{1}{2^n}\| \varphi (p_n) - \psi (p_n) \|_2 .
\end{align*}
In the following definition we consider the collection of unital positive maps from $L^\infty (X,\mu )$ to $\Cb^d$
which, in a local sense, are approximately mutiplicative,
approximately pull the uniform probability measure $\zeta$ back to $\mu$, and are approximately equivariant.

\begin{definition}\label{D-UP}
Let $m\in\Nb$ and $\delta > 0$.
Define $\UP_\mu (\cS , F,m,\delta ,\sigma )$ to be the set of all unital positive linear maps
$\varphi : L^\infty (X,\mu ) \to \Cb^d$ such that
\begin{enumerate}
\item[(i)] $\| \varphi (\alpha_{s_1} (f_1 ) \cdots \alpha_{s_j} (f_j ))
- \varphi (\alpha_{s_1} (f_1 ))\cdots\varphi (\alpha_{s_j} (f_j )) \|_2 < \delta$
for all $1\le j\le m$, $f_1 , \dots ,f_j \in \{p_1, \dots, p_m\}$ and $s_1 , \dots ,s_j \in F$,

\item[(ii)] $| \zeta\circ\varphi (f) - \mu (f) | < \delta$ for all $f\in \cS_{F, m}$,

\item[(iii)] $\| \varphi\circ\alpha_s (f) - \sigma_s \circ\varphi (f) \|_2 < \delta$ for all $s\in F$ and
$f\in \{p_1, \dots, p_m\}$.
\end{enumerate}
\end{definition}


For a pseudometric space $(Y,\rho )$ and an $\varepsilon\geq 0$ we write $N_\varepsilon (Y, \rho )$ for the
maximal cardinality of finite $\varepsilon$-separated subset of $Y$ respect to $\rho$.
In the case $\varepsilon = 0$ the number $N_0 (Y, \rho )$
is simply cardinality modulo the relation of zero distance.

Throughout this section, as elsewhere, $\Sigma = \{ \sigma_i : G \to \Sym (d_i ) \}_{i=1}^\infty$
is a fixed sofic approximation sequence.


Note that $\UP_\mu (\cS ,F,m, \delta ,\sigma)\supseteq \UP_\mu (\cS ,F',m', \delta' ,\sigma)$
and hence $N_{\varepsilon}(\UP_\mu (\cS ,F,m, \delta ,\sigma), \rho_\cS)\geq N_{\varepsilon'}(\UP_\mu (\cS ,F',m', \delta' ,\sigma), \rho_\cS)$
whenever $F\subseteq F'$, $m\leq m'$, $\delta\geq\delta'$, and $\varepsilon\leq \varepsilon'$.

\begin{definition}\label{D-entropy}
Let $\cS$ be a sequence in the unit ball of $L^\infty (X,\mu )$, $\varepsilon > 0$, $F$ a nonempty finite subset of $G$,
$m\in \Nb$, and $\delta > 0$. We define
\begin{align*}
h_{\Sigma ,\mu}^\varepsilon (\cS ,F,m,\delta ) &=
\limsup_{i\to\infty} \frac{1}{d_i} \log N_\varepsilon (\UP_\mu (\cS ,F,m,\delta ,\sigma_i ),\rho_{\cS} ) ,\\
h_{\Sigma ,\mu}^\varepsilon (\cS ,F,m) &= \inf_{\delta > 0} h_{\Sigma ,\mu}^\varepsilon (\cS ,F,m,\delta ) ,\\
h_{\Sigma ,\mu}^\varepsilon (\cS ,F) &= \inf_{m\in\Nb} h_{\Sigma ,\mu}^\varepsilon (\cS ,F,m) ,\\
h_{\Sigma ,\mu}^\varepsilon (\cS ) &= \inf_{F} h_{\Sigma ,\mu}^\varepsilon (\cS ,F) ,\\
h_{\Sigma ,\mu} (\cS ) &= \sup_{\varepsilon > 0} h_{\Sigma ,\mu}^\varepsilon (\cS )
\end{align*}
where the infimum in the second last line is over all nonempty finite subsets of $G$.
If $\UP_\mu (\cS ,F,\delta ,\sigma_i )$ is empty for all sufficiently large $i$,
we set $h_{\Sigma ,\mu}^\varepsilon (\cS , F,\delta ) = -\infty$.
\end{definition}


\begin{remark}\label{R-pou}
If we add $1$ to $\cS$ by setting $p'_1=1$ and $p'_{j+1}=p_j$ for all $j\in \Nb$, then for any nonempty finite subset
$F$ of $G$, any $m\in \Nb$, any $\delta>0$ and any map $\sigma$ from $G$ to $\Sym(d)$ for some $d\in \Nb$,  we have $\UP_{\mu}(\cS, F, m,\delta, \sigma)\supseteq \UP_{\mu}(\cS', F, m+1,\delta, \sigma)\supseteq\UP_{\mu}(\cS, F, m+1,\delta, \sigma)$. Thus $h_{\Sigma, \mu}(\cS)=h_{\Sigma, \mu}(\cS')$.
\end{remark}

%

Notice that the quantity $N_\varepsilon (\UP_\mu (\cS ,F,m,\delta ,\sigma_i ),\rho_{\cS} )$ in
Definition~\ref{D-entropy} is a purely local one, in the sense that the maps in
$\UP_\mu (\cS ,F,m,\delta ,\sigma_i )$ could have been merely defined on the finite-dimensional unital self-adjoint
linear subspace of $L^\infty (X,\mu )$ which gives meaning to the conditions in its definition.
Indeed any such map on this subspace can be extended to a unital positive map on all of $L^\infty (X,\mu )$
by the Hahn-Banach theorem, as discussed in the introduction.
This locality is crucial in the proof of the variational principle
in Section~\ref{S-variational}. On the other hand, in order to carry out the perturbation argument
showing that $h_{\Sigma, \mu}(\cS)$ depends only on the $G$-invariant von Neumann subalgebra
generated by $\cS$ (Theorem~\ref{T-gen comparison}) one also needs some $L^2$-norm control
on unital positive maps beyond the finite-dimensional subspace on which the
computation of $N_\varepsilon (\UP_\mu (\cS ,F,m,\delta ,\sigma_i ),\rho_{\cS} )$ depends.
To this end we next demonstrate that $h_{\Sigma, \mu}(\cS)$ can be calculated using
unital positive maps which are uniformly bounded with respect to the $L^2$-norm.
Note that if $\cS$ consists of projections then this can be accomplished much more easily by
simply composing with conditional expectations onto finite-dimensional $^*$-subalgebras.

\begin{definition} \label{D-bounded}
Let $\cS$ be a sequence in the unit ball of $L^{\infty}(X, \mu)$, $\lambda>1$, $F$ a nonempty finite subset
of $G$, $m\in \Nb$, $\delta>0$, and $\sigma$ a map from $G$ to $\Sym(d)$ for some $d\in \Nb$.
Define $\UP_{\mu, \lambda}(\cS, F, m, \delta, \sigma)$ to be the subset of
$\UP_{\mu}(\cS, F, m, \delta, \sigma)$ consisting of $\varphi$ satisfying
$\|\varphi(f)\|_2\le \lambda \|f\|_2$ for all $f\in L^{\infty}(X, \mu)$.
\end{definition}

\begin{proposition} \label{P-bounded}
Let $\cS = \{ p_n \}_{n=1}^\infty$ be a sequence in the unit ball of $L^{\infty}(X, \mu)$ and $\lambda>1$.
Then
\[ h_{\Sigma, \mu}(\cS)=\sup_{\varepsilon>0}\inf_F\inf_{m\in \Nb}\inf_{\delta>0}\limsup_{i\to \infty}\frac{1}{d_i}\log N_{\varepsilon}(\UP_{\mu, \lambda}(\cS, F, m, \delta, \sigma_i), \rho_\cS),\]
where $F$ ranges over all nonempty finite subsets of $G$.
\end{proposition}

\begin{proof}
Replacing $L^{\infty}(X, \mu)$ by the $G$-invariant von Neumann subalgebra generated by $\cS$ if necessary,
we may assume that $\cS$ is dynamically generating.
Since $\UP_{\mu}(\cS, F, m, \delta, \sigma)\supseteq \UP_{\mu, \lambda}(\cS, F, m, \delta, \sigma)$,
the left side of the displayed equality is clearly bounded below by the right side.

To prove the inverse inequality, by Remark~\ref{R-pou} we may assume that $p_1=1$.
It suffices to show that, for any $\varepsilon>0$, one has
\begin{align}\label{E-bound}
h_{\Sigma, \mu}^{\varepsilon}(\cS)\le \inf_F\inf_{m\in \Nb}\inf_{\delta>0}\limsup_{i\to \infty}\frac{1}{d_i}\log N_{\varepsilon/2}(\UP_{\mu, \lambda}(\cS, F, m, \delta, \sigma_i), \rho_\cS).
\end{align}
Set $\lambda_1=\min(2, \lambda^{1/3})$. Let $F$ be a finite subset of $G$ containing $e$, $m\in \Nb$ with $2^{-(m-1)}<\varepsilon/8$, and $0<\delta<\varepsilon/4$.

Take a finite partition of unity $\cQ$ in $L^{\infty}(X, \mu)$ consisting of projections such
that $\|f-\Eb(f|\cQ)\|_{\infty}<(18m)^{-1}\delta$ for every $f\in \cS_{F, m}$, where
$\Eb(\cdot |\cQ)$ denotes the conditional expectation from $L^{\infty}(X, \mu)$ to $\spn(\cQ)$.

Take $0<\eta<(4|\cQ|)^{-1}$ be a small number which we will determine in a moment.
Since $\cS$ is dynamically generating and $p_1=1$, by Kaplansky's density theorem \cite[Thm.\ 5.3.5]{KadRinI83}
there are a finite set $E\subseteq G$ containing $F$ and an integer
$\ell\ge m$ such that for each $q\in\cQ$ there exists
some $\tilde{q}\in \spn(\cS_{E, \ell})$ satisfying $\|\tilde{q}\|_{\infty}\le 1$ and $\|q-\tilde{q}\|_2<\eta$.
Set $q'=\tilde{q}\overline{\tilde{q}}$. Then $q'\in \spn(\cS_{E, 2\ell})$, $q'\ge 0$, $\|q'\|_{\infty}\le 1$,
and
\[
\|q-q'\|_2=\|q\overline{q}-\tilde{q}\overline{\tilde{q}}\|_2\le \|q(\overline{q}-\overline{\tilde{q}})\|_2+\|(q-\tilde{q})\overline{\tilde{q}}\|_2\le 2\|q-\tilde{q}\|_2<2\eta.
\]
Denote by $\theta$ the linear map $\spn(\cQ)\rightarrow L^{\infty}(X, \mu)$ sending $q$ to $q'$.
Then $\theta$ is positive.
When $\eta$ is small enough, we have
$\|\theta(f)-f\|_2\le (18m)^{-1}\delta\|f\|_2$ and $\|\theta(f)\|_2\le \lambda_1\|f\|_2$
for all $f\in \spn(\cQ)$.

Take $0<\eta'<\delta/3$ such that if $\varphi$ is a linear map from $\spn(\cS_{E, 2\ell})$ to some Hilbert space
satisfying $|\left<f_1, f_2\right>-\left<\varphi(f_1), \varphi(f_2)\right>|<4\eta'$ for all $f_1, f_2\in \cS_{E, 2\ell}$,
then $\|\varphi(f)\|_2\le \lambda_1\|f\|_2$ for all $f\in \spn(\cS_{E, 2\ell})$.

Given a map $\sigma : G\to\Sym(d)$ for some $d\in \Nb$,
we will construct a map $\Gamma: \UP_{\mu}(\cS, E, 4\ell, \eta', \sigma)\rightarrow \UP_{\mu, \lambda}(\cS, F, m, \delta, \sigma)$ such that $\rho_\cS(\Gamma(\varphi), \varphi)<\varepsilon/4$ for every $\varphi\in \UP_{\mu}(\cS, E, 4\ell, \eta', \sigma)$.
Then for any $\varphi, \psi\in \UP_{\mu}(\cS, E, 4\ell, \eta', \sigma)$ one has
\begin{align*}
\rho_\cS(\varphi, \psi)&\le \rho_\cS(\varphi, \Gamma(\varphi))+\rho_{\cS}(\Gamma(\varphi), \Gamma(\psi))+\rho_\cS(\Gamma(\psi), \psi)\\
&< \frac{\varepsilon}{2}+\rho_{\cS}(\Gamma(\varphi), \Gamma(\psi)).
\end{align*}
Thus for any $\varepsilon$-separated subset $L$ of $\UP_{\mu}(\cS, E, 4\ell, \eta', \sigma)$ with respect to $\rho_\cS$, $\Gamma(L)$ is $\varepsilon/2$-separated. Therefore $N_{\varepsilon}(\UP_{\mu}(\cS, E, 4\ell, \eta', \sigma), \rho_\cS)\le N_{\varepsilon/2}(\UP_{\mu, \lambda}(\cS, F, m, \delta, \sigma), \rho_\cS)$, and hence $h_{\Sigma, \mu}^\varepsilon(\cS, E, 4\ell, \eta')\le  \limsup_{i\to \infty}\frac{1}{d_i}\log N_{\varepsilon/2}(\UP_{\mu, \lambda}(\cS, F, m, \delta, \sigma_i), \rho_\cS)$.
Since $F$ can be chosen to contain an arbitrary finite subset of $G$, $m$ can be arbitrarily large,
and $\delta$ can be arbitrarily small, this implies \eqref{E-bound}.

Let $\varphi\in \UP_{\mu}(\cS, E, 4\ell, \eta', \sigma)$.
For any $1\le j\le 2\ell$ and $(h_1, s_1), (h_2, s_2)\in \{p_1, \dots, p_{2\ell}\}^j\times E^j$, since $4\ell\ge 2j$,
we have
\begin{align*}
\lefteqn{\bigg|\bigg<\prod_{k=1}^{j}\alpha_{s_{1, k}}(h_{1, k}),
\prod_{k=1}^{j}\alpha_{s_{2, k}}(h_{2, k})\bigg>-\bigg<\varphi\bigg(\prod_{k=1}^{j}\alpha_{s_{1, k}}(h_{1, k})\bigg),
\varphi\bigg(\prod_{k=1}^{j}\alpha_{s_{2, k}}(h_{2, k})\bigg)\bigg>\bigg|} \hspace*{5mm}\\
\hspace*{5mm} &= \bigg|\mu\bigg(\prod_{k=1}^{j}\alpha_{s_{1, k}}(h_{1, k})\alpha_{s_{2, k}}(h_{2, k})\bigg)-
\zeta\bigg(\varphi\bigg(\prod_{k=1}^{j}\alpha_{s_{1, k}}(h_{1, k})\bigg)\varphi\bigg(\prod_{k=1}^{j}
\alpha_{s_{2, k}}(h_{2, k})\bigg)\bigg)\bigg|\\
&\le \bigg|\mu\bigg(\prod_{k=1}^{j}\alpha_{s_{1, k}}(h_{1, k})\alpha_{s_{2, k}}(h_{2, k})\bigg)-
\zeta\bigg(\varphi\bigg(\prod_{k=1}^{j}\alpha_{s_{1, k}}(h_{1, k})\alpha_{s_{2, k}}(h_{2, k})\bigg)\bigg)\bigg|\\
&\hspace*{10mm} \ +\bigg|\zeta\bigg(\varphi\bigg(\prod_{k=1}^{j}\alpha_{s_{1, k}}(h_{1, k})
\alpha_{s_{2, k}}(h_{2, k})\bigg)\bigg) \\
&\hspace*{50mm} \ -\zeta\bigg(\varphi\bigg(\prod_{k=1}^{j}\alpha_{s_{1, k}}(h_{1, k})\bigg)\varphi\bigg(\prod_{k=1}^{j}
\alpha_{s_{2, k}}(h_{2, k})\bigg)\bigg)\bigg|\\
&< \eta'+\bigg\|\varphi\bigg(\prod_{k=1}^{j}\alpha_{s_{1, k}}(h_{1, k})\alpha_{s_{2, k}}(h_{2, k})\bigg)
-\varphi\bigg(\prod_{k=1}^{j}\alpha_{s_{1, k}}(h_{1, k})\bigg)\varphi\bigg(\prod_{k=1}^{j}
\alpha_{s_{2, k}}(h_{2, k})\bigg)\bigg\|_2\\
&\le \eta'+\bigg\|\varphi\bigg(\prod_{k=1}^{j}\alpha_{s_{1, k}}(h_{1, k})
\alpha_{s_{2, k}}(h_{2, k})\bigg)-\prod_{k=1}^{j}\varphi(\alpha_{s_{1, k}}(h_{1, k}))
\varphi(\alpha_{s_{2, k}}(h_{2, k}))\bigg\|_2\\
&\hspace*{10mm} \ + \bigg\|\bigg(\prod_{k=1}^{j}\varphi(\alpha_{s_{1, k}}(h_{1, k}))\bigg)\bigg(\prod_{k=1}^{j}
\varphi(\alpha_{s_{2, k}}(h_{2, k}))\bigg)\\
&\hspace*{50mm} \ -\varphi\bigg(\prod_{k=1}^{j}\alpha_{s_{1, k}}(h_{1, k})\bigg)\varphi\bigg(\prod_{k=1}^{j}
\alpha_{s_{2, k}}(h_{2, k})\bigg)\bigg\|_2\\
&< 2\eta'+\bigg\|\prod_{k=1}^{j}\varphi(\alpha_{s_{1, k}}(h_{1, k}))-
\varphi\bigg(\prod_{k=1}^{j}\alpha_{s_{1, k}}(h_{1, k})\bigg)\bigg\|_2\\
&\hspace*{10mm} \ + \bigg\|\prod_{k=1}^{j}\varphi(\alpha_{s_{2, k}}(h_{2, k}))-\varphi\bigg(\prod_{k=1}^{j}
\alpha_{s_{2, k}}(h_{2, k})\bigg)\bigg\|_2\\
&< 4\eta'.
\end{align*}
By the choice of $\eta'$, we conclude that $\|\varphi(f)\|_2\le \lambda_1\|f\|_2$ for all $f\in \spn(\cS_{E, 2\ell})$.
Thus $\|\varphi\circ \theta(f)\|_2\le \lambda^2_1\|f\|_2$ for all $f\in \spn(\cQ)$.

As $\theta$ and $\varphi$ are positive, $\sum_{q\in \cQ}\varphi\circ \theta(q)\ge 0$.
Since $\theta(\cQ)\subseteq \spn(\cS_{E, 2\ell})$ and $1=p_1\in \spn(\cS_{E, 2\ell})$, we have
\begin{align*}
\bigg\|\sum_{q\in \cQ}\varphi\circ \theta(q)-1\bigg\|_2&= \bigg\|\varphi\bigg(\sum_{q\in \cQ}\theta(q)-1\bigg)\bigg\|_2
\le \lambda_1\bigg\|\sum_{q\in \cQ}\theta(q)-1\bigg\|_2\\
&= \lambda_1\bigg\|\sum_{q\in \cQ}(\theta(q)-q)\bigg\|_2
\le \lambda_1 \sum_{q\in \cQ}\|\theta(q)-q\|_2\\
&< 2\lambda_1|\cQ|\eta\le 4|\cQ|\eta.
\end{align*}
Then there exists a subset $J\subseteq \{1, \dots, d\}$ with $|J|\ge d(1-4|\cQ|\eta)$
such that $ |\sum_{q\in \cQ}\varphi\circ \theta(q)(a)-1|<(4|\cQ|\eta)^{1/2}$ for all $a\in J$.
Then $\sum_{q\in \cQ}\varphi\circ \theta(q)(a)>1-(4|\cQ|\eta)^{1/2}>0$ for all $a\in J$.
Take a unital positive linear map $\tilde{\varphi}:\spn(\cQ)\rightarrow \Cb^d$
such that $\tilde{\varphi}(q)=(\sum_{q_1\in \cQ}\varphi\circ \theta(q_1))^{-1}\varphi\circ \theta(q)$ on $J$ for every
$q\in \cQ$.
Now we define $\Gamma(\varphi):L^{\infty}(X, \mu)\rightarrow \Cb^d$ to be $\tilde{\varphi}\circ \Eb(\cdot|\cQ)$.
Clearly $\Gamma(\varphi)$ is a unital positive linear map.

Denote by $P_J$ the orthogonal projection $\Cb^d\rightarrow \Cb^J$.
For any $q\in \cQ$, we have
\begin{align*}
\|\varphi\circ \theta(q)-\tilde{\varphi}(q)\|_2
&\le \|(1-P_J)(\varphi\circ \theta(q)-\tilde{\varphi}(q))\|_2+
\|P_J(\varphi\circ \theta(q)-\tilde{\varphi}(q))\|_2\\
&\le \|\varphi\circ \theta(q)-\tilde{\varphi}(q)\|_{\infty}\bigg(\frac{d-|J|}{d}\bigg)^{1/2} +\frac{(4|\cQ|\eta)^{1/2}}{1-(4|\cQ|\eta)^{1/2}} \|\varphi\circ \theta(q)\|_2\\
&\le (\|\varphi\circ \theta(q)\|_{\infty}+\|\tilde{\varphi}(q)\|_{\infty})(4|\cQ|\eta)^{1/2}+\frac{(4|\cQ|\eta)^{1/2}}{1-(4|\cQ|\eta)^{1/2}} \|\varphi\circ \theta(q)\|_{\infty}\\
&\le 2(4|\cQ|\eta)^{1/2}+\frac{(4|\cQ|\eta)^{1/2}}{1-(4|\cQ|\eta)^{1/2}}.
\end{align*}
When $\eta$ is small enough, we obtain $\|\varphi\circ \theta(g)-\tilde{\varphi}(g)\|_2\le (\lambda-\lambda_1^2)\|g\|_2$ for
all $g\in \spn(\cQ)$. Then $\|\tilde{\varphi}(g)\|_2\le \|\varphi\circ \theta(g)-\tilde{\varphi}(g)\|_2+\|\varphi\circ\theta(g)\|_2\le \lambda \|g\|_2$ for all $g\in \spn(\cQ)$, and hence $\|\Gamma(\varphi)(f)\|_2\le \lambda\|\Eb(f|\cQ)\|_2\le \lambda \|f\|_2$ for all $f\in L^{\infty}(X, \mu)$.

Let $f\in \cS_{F, m}$. Set $f'=\Eb(f|\cQ)$.
Then $f'=\sum_{p\in \cQ}\frac{\mu(fq)}{\mu(q)}q$, and hence
\begin{align*}
\|\varphi(\theta(f'))-\tilde{\varphi}(f')\|_2
&=\bigg\|\sum_{q\in \cQ}\frac{\mu(fq)}{\mu(q)}(\varphi\circ \theta(q)-\tilde{\varphi}(q))\bigg\|_2\\
&\le \sum_{q\in \cQ}\frac{\mu(fq)}{\mu(q)}\|\varphi\circ \theta(q)-\tilde{\varphi}(q)\|_2\\
&\le (2(4|\cQ|\eta)^{1/2}+\frac{(4|\cQ|\eta)^{1/2}}{1-(4|\cQ|\eta)^{1/2}})\sum_{q\in \cQ}\frac{\mu(fq)}{\mu(q)}.
\end{align*}
When $\eta$ is small enough, we get
\[\|\varphi(\theta(f'))-\tilde{\varphi}(f')\|_2<\frac{\delta}{9m}.\]
Since $f, \theta(f')\in \spn(\cS_{E, 2\ell})$, we have
\begin{align} \label{E-bound2}
\|\varphi(f)-\Gamma(\varphi)(f)\|_2
&\le \|\varphi(f)-\varphi(\theta(f'))\|_2+\|\varphi(\theta(f'))-\Gamma(\varphi)(f)\|_2\\
\nonumber &\le \lambda_1\|f-\theta(f')\|_2+\|\varphi(\theta(f'))-\tilde{\varphi}(f')\|_2\\
\nonumber &< \lambda_1\|f-f'\|_2+\lambda_1\|f'-\theta(f')\|_2+\frac{\delta}{9m}\\
\nonumber &\le \frac{\lambda_1\delta}{18m}+\frac{\lambda_1\delta\|f'\|_2}{18m}+\frac{\delta}{9m}\\
\nonumber &\le \frac{\delta}{9m}+\frac{\delta}{9m}+\frac{\delta}{9m}= \frac{\delta}{3m}.
\end{align}

For all $1\le j\le m$ and $(h, s)\in \{p_1, \dots, p_m\}^j\times F^j$ we have, since
$E\supseteq F$ and $\ell\ge m$,
\begin{align*}
\lefteqn{\bigg\|\Gamma(\varphi)\bigg(\prod_{k=1}^j\alpha_{s_k}(h_k)\bigg)-
\prod_{k=1}^j\Gamma(\varphi)(\alpha_{s_k}(h_k))\bigg\|_2} \hspace*{25mm}\\
\hspace*{25mm} &\le
\bigg\|\Gamma(\varphi)\bigg(\prod_{k=1}^j\alpha_{s_k}(h_k)\bigg)-
\varphi\bigg(\prod_{k=1}^j\alpha_{s_k}(h_k)\bigg)\bigg\|_2\\
&\hspace*{15mm} \ +
\bigg\|\varphi\bigg(\prod_{k=1}^j\alpha_{s_k}(h_k)\bigg)-\prod_{k=1}^j\varphi(\alpha_{s_k}(h_k))\bigg\|_2\\
&\hspace*{15mm} \ +
\bigg\|\prod_{k=1}^j\varphi(\alpha_{s_k}(h_k))-\prod_{k=1}^j\Gamma(\varphi)(\alpha_{s_k}(h_k))\bigg\|_2\\
&\overset{\eqref{E-bound2}}< \frac{\delta}{3m}+\eta'+
\sum_{k=1}^j\|\varphi(\alpha_{s_k}(h_k))-\Gamma(\varphi)(\alpha_{s_k}(h_k))\|_2\\
&\overset{\eqref{E-bound2}}< \frac{\delta}{3m}+\eta'+\frac{\delta}{3}<\delta.
\end{align*}
Also, for all $f\in \cS_{F, m}$ we have
\begin{align*}
|\zeta\circ \Gamma(\varphi)(f)-\mu(f)|
&\le
|\zeta\circ \Gamma(\varphi)(f)-\zeta\circ \varphi(f)|+|\zeta\circ \varphi(f)-\mu(f)|\\
&<
\|\Gamma(\varphi)(f)-\varphi(f)\|_2+\eta'\\
&\overset{\eqref{E-bound2}}< \frac{\delta}{3m}+\eta'<\delta.\\
\end{align*}
Furthermore, for all $s\in F$ and $f\in \{p_1, \dots, p_m\}$ we have, since $e\in F$ and $F\subseteq E$,
\begin{align*}
\|\Gamma(\varphi)\circ \alpha_s(f)-\sigma_s\circ \Gamma(\varphi)(f)\|_2
&\le
\|\Gamma(\varphi)\circ \alpha_s(f)-\varphi\circ \alpha_s(f)\|_2 \\
&\hspace*{15mm} \ +\|\varphi\circ \alpha_s(f)-\sigma_s\circ \varphi(f)\|_2\\
&\hspace*{15mm} \ +\|\sigma_s\circ \varphi(f)-\sigma_s\circ \Gamma(\varphi)(f)\|_2\\
&\overset{\eqref{E-bound2}}< \frac{\delta}{3m}+\eta'+\frac{\delta}{3m}<\delta.
\end{align*}
Therefore $\Gamma(\varphi)\in \UP_{\mu, \lambda}(\cS, F, m, \delta, \sigma)$.

Finally, since $e\in F$, we have
\begin{align*}
\rho_\cS(\varphi, \Gamma(\varphi))&=\sum_{j=1}^{\infty}\frac{1}{2^j}\|\varphi(p_j)-\Gamma(\varphi)(p_j)\|_2\\
&\le \sum_{j=1}^{m}\|\varphi(p_j)-\Gamma(\varphi)(p_j)\|_2+\frac{1}{2^{m-1}}\\
&\overset{\eqref{E-bound2}}< \frac{\delta}{3}+\frac{1}{2^{m-1}}<\frac{\varepsilon}{4},
\end{align*}
as desired.
\end{proof}

We now show that all dynamically generating sequences in the unit ball of $L^\infty_\Rb (X,\mu )$
have the same entropy. This is the counterpart of Theorem~2.1 in \cite{Bowen10},
of which it provides another proof in conjunction with Proposition~\ref{P-infinite prime} in the next section.

\begin{theorem}\label{T-gen comparison}
Let $\cS=\{p_n\}^{\infty}_{n=1}$ and $\cT=\{q_n\}^{\infty}_{n=1}$ be dynamically generating sequences in the unit ball of $L^\infty (X,\mu )$.
Then $h_{\Sigma ,\mu} (\cT ) = h_{\Sigma ,\mu} (\cS )$.
\end{theorem}

\begin{proof}
It suffices by symmetry to prove that $h_{\Sigma, \mu} (\cT ) \leq h_{\Sigma, \mu} (\cS )$.
By Remark~\ref{R-pou}, we may assume that $p_1=q_1=1$.

Let $\varepsilon>0$. Take $R\in \Nb$ with $2^{-(R-1)}<\varepsilon/3$.
Since $\cS$ is dynamically generating and $p_1=1$, by Kaplansky's density theorem \cite[Thm.\ 5.3.5]{KadRinI83}
there is a nonempty finite set $E\subseteq G$ and an
$\ell\in\Nb$ such that for each $q\in\{q_1, \dots, q_R\}$ there exist $d_{q,g} \in\Cb$ for $g\in \cS_{E, \ell}$ such that
the function
\[
q' = \sum_{g\in \cS_{E, \ell}} d_{q, g}g
\]
satisfies $\| q - q' \|_2 < (12R)^{-1}\varepsilon$.
Set $M=\max_{1\le j\le R}\max_{g\in \cS_{E, \ell}} | d_{q_j,g} |$
and $\varepsilon'=\varepsilon/(2^{\ell+4}M R \ell^{\ell+1} |E|^\ell )$.
We will show that
\begin{align} \label{E-gen comparison}
\lefteqn{\inf_F\inf_{m\in \Nb}\inf_{\delta>0}\limsup_{i\to \infty}\frac{1}{d_i}\log N_{\varepsilon}
(\UP_{\mu, 2}(\cT, F, m, \delta, \sigma_i), \rho_\cT)} \hspace*{10mm}\\
\nonumber &\hspace*{10mm} \le \inf_F\inf_{m\in \Nb}\inf_{\delta>0}\limsup_{i\to \infty}\frac{1}{d_i}
\log N_{\varepsilon'}(\UP_{\mu, 2}(\cS, F, m, \delta, \sigma_i), \rho_\cS),
\end{align}
where $F$ ranges over all nonempty finite subsets of $G$.
Since $\varepsilon$ is an arbitrary
positive number, by Proposition~\ref{P-bounded} this will imply $ h_{\Sigma, \mu} (\cT ) \leq h_{\Sigma, \mu} (\cS )$.

Let $F$ be a nonempty finite subset of $G$ containing $e$ and $E$, $m$ a positive integer with $m\ge \ell$,
and $\delta\in (0,\varepsilon' ]$.

As $\cT$ is dynamically generating and $q_1=1$, by Kaplansky's density theorem \cite[Thm.\ 5.3.5]{KadRinI83}
there are a nonempty finite set $D\subseteq G$ and an $n\in\Nb$ such that
for each $f\in\cS_{F, m}$ there exist
$c_{f, g} \in\Cb$ for $g\in \cT_{D, n}$ such that  the function
\[
f' = \sum_{g\in \cT_{D, n}} c_{f,g}g
\]
satisfies $\|f'\|_{\infty}\le 1$ and $\| f - f' \|_2 < \delta/(6m)$.
Set
$M_1=\max_{f\in \cS_{F, m}}\max_{g\in \cT_{D, n}} | c_{f, g} |$.

Take
a $\delta' > 0$ such that
$\max((m+1)n^{mn}|D|^{mn}M_1^m\delta',(2+3n)n^n|D|^nM_1\delta')
<\delta/3$.
We will show that
\begin{align} \label{E-gen comparison1}
\lefteqn{\limsup_{i\to \infty}\frac{1}{d_i}\log N_{\varepsilon}(\UP_{\mu, 2}(\cT, FD, mn, \delta'), \rho_\cT)}
\hspace*{15mm} \\
\nonumber &\hspace*{15mm} \leq \limsup_{i\to \infty}\frac{1}{d_i}
\log N_{\varepsilon'}(\UP_{\mu, 2}(\cS, F, m, \delta), \rho_\cS).
\end{align}
Since $F$ can be chosen to contain an arbitrary finite subset of $G$, $m$ can be arbitrarily large,
and $\delta$ can be arbitrarily small, this implies \eqref{E-gen comparison}.

Let $\sigma$ be a map from $G$ to $\Sym (d)$ for some $d\in\Nb$, which we assume to be
a good enough sofic approximation for our purposes below.
Let $\varphi\in\UP_{\mu, 2} (\cT ,FD,mn, \delta', \sigma )$.
We will show that $\varphi \in\UP_{\mu, 2} (\cS ,F,m, \delta , \sigma )$.

Let $(f,v)\in \{p_1, \dots, p_m\}^m\times F^m$.  Using that
$\|\varphi(\alpha_{v_k}(f_k))\|_{\infty}\le \|\alpha_{v_k}(f_k)\|_{\infty}\le 1$, $f_k\in \cS_{F, m}$,
and $\|\varphi(\alpha_{v_k}(f'_k))\|_{\infty}\le \|\alpha_{v_k}(f'_k)\|_{\infty}\le 1$
for each $k=1,\dots ,m$, and that $\varphi$ has norm at most $2$ with respect to the $L^2$-norms,
we have
\begin{align*}
\bigg\|\varphi\bigg(\prod_{k=1}^m\alpha_{v_k}(f_k)\bigg)-\varphi\bigg(\prod_{k=1}^m\alpha_{v_k}(f'_k)\bigg)\bigg\|_2
&\le 2\bigg\|\prod_{k=1}^m\alpha_{v_k}(f_k)-\prod_{k=1}^m\alpha_{v_k}(f'_k)\bigg\|_2\\
&\le 2\sum_{k=1}^m\|\alpha_{v_k}(f_k)-\alpha_{v_k}(f'_k)\|_2 \\
&=2\sum_{k=1}^m\|f_k-f'_k\|_2<\frac{\delta}{3}
\end{align*}
and
\begin{align*}
\bigg\|\prod_{k=1}^m\varphi(\alpha_{v_k}(f'_k))-\prod_{k=1}^m\varphi(\alpha_{v_k}(f_k))\bigg\|_2
&\le \sum_{k=1}^m\|\varphi(\alpha_{v_k}(f'_k))-\varphi(\alpha_{v_k}(f_k))\|_2\\
&\le 2\sum_{k=1}^m\|\alpha_{v_k}(f'_k)-\alpha_{v_k}(f_k)\|_2 \\
&=2\sum_{k=1}^m\|f'_k-f_k\|_2<\frac{\delta}{3}.
\end{align*}
For any $(h_1, s_1), \dots, (h_m, s_m)\in \{q_1, \dots, q_n\}^n\times D^n$ we have
\begin{align*}
\lefteqn{\bigg\|\varphi\bigg(\prod_{k=1}^m\alpha_{v_k}\bigg(\prod_{j=1}^n
\alpha_{s_{k, j}}(h_{k,j})\bigg)\bigg)-\prod_{k=1}^m\varphi\bigg(\alpha_{v_k}\bigg(\prod_{j=1}^n
\alpha_{s_{k, j}}(h_{k,j})\bigg)\bigg)\bigg\|_2} \hspace*{15mm} \\
\hspace*{15mm} &\leq \bigg\|\varphi\bigg(\prod_{k=1}^m\prod_{j=1}^n\alpha_{v_ks_{k, j}}(h_{k,j})\bigg)-
\prod_{k=1}^m\prod_{j=1}^n\varphi(\alpha_{v_ks_{k, j}}(h_{k,j}))\bigg\|_2\\
&\hspace*{20mm} \ +\bigg\|\prod_{k=1}^m\prod_{j=1}^n
\varphi(\alpha_{v_ks_{k, j}}(h_{k,j}))-\prod_{k=1}^m\varphi\bigg(\prod_{j=1}^n
\alpha_{v_ks_{k, j}}(h_{k,j})\bigg)\bigg\|_2\\
&\leq \delta'+\sum_{k=1}^m\bigg\|\prod_{j=1}^n
\varphi(\alpha_{v_ks_{k, j}}(h_{k,j}))-\varphi\bigg(\prod_{j=1}^n\alpha_{v_ks_{k, j}}(h_{k,j})\bigg)\bigg\|_2\\
&\leq (m+1)\delta',
\end{align*}
and hence
\begin{align*}
\lefteqn{\bigg\|\varphi\bigg(\prod_{k=1}^m\alpha_{v_k}(f'_k)\bigg)-\prod_{k=1}^m\varphi(\alpha_{v_k}(f'_k))\bigg\|_2}
\hspace*{10mm}\\
\hspace*{10mm} &\le \sum_{g_1\in \cT_{D, n}}\dots \sum_{g_m\in \cT_{D, n}}
\bigg(\prod_{k=1}^m|c_{f_k, g_k}|\bigg)\bigg\|\varphi\bigg(\prod_{k=1}^m\alpha_{v_k}(g_k)\bigg)
-\prod_{k=1}^m\varphi(\alpha_{v_k}(g_k))\bigg\|_2\\
&\leq M_1^m \sum_{g_1\in \cT_{D, n}}\dots \sum_{g_m\in \cT_{D, n}}
\bigg\|\varphi\bigg(\prod_{k=1}^m\alpha_{v_k}(g_k)\bigg)
-\prod_{k=1}^m\varphi(\alpha_{v_k}(g_k))\bigg\|_2\\
&\leq (m+1)n^{mn}|D|^{mn}M_1^m\delta'
\leq \frac{\delta}{3}.
\end{align*}
Therefore
\begin{align*}
\bigg\|\varphi\bigg(\prod_{k=1}^m\alpha_{v_k}(f_k)\bigg)-\prod_{k=1}^m\varphi(\alpha_{v_k}(f_k))\bigg\|_2
&\leq
\bigg\|\varphi\bigg(\prod_{k=1}^m\alpha_{v_k}(f_k)\bigg)-\varphi\bigg(\prod_{k=1}^m\alpha_{v_k}(f'_k)\bigg)\bigg\|_2\\
&\hspace*{5mm}\ +\bigg\|\varphi\bigg(\prod_{k=1}^m\alpha_{v_k}(f'_k)\bigg)-
\prod_{k=1}^m\varphi(\alpha_{v_k}(f'_k))\bigg\|_2 \\
&\hspace*{5mm}\ +\bigg\|\prod_{k=1}^m\varphi(\alpha_{v_k}(f'_k))-\prod_{k=1}^m\varphi(\alpha_{v_k}(f_k))\bigg\|_2\\
&< \frac{\delta}{3}+\frac{\delta}{3}+\frac{\delta}{3}=\delta.
\end{align*}

Given an $f\in \cS_{F, m}$, since $q_1=1$ and $e\in F$, we have
\begin{align*}
|\zeta\circ \varphi(f')-\mu(f')| &\le \sum_{g\in \cT_{D, n}}|c_{f, g}|\cdot |\zeta\circ \varphi(g)-
\mu(g)|\\
&\le M_1 \sum_{g\in \cT_{D, n}}|\zeta\circ \varphi(g)-\mu(g)|\\
&\leq n^n|D|^nM_1\delta'<\frac{\delta}{2},
\end{align*}
and thus, using that $\varphi$ has norm at most $2$ with respect to the $L^2$-norms,
\begin{align*}
|\zeta\circ \varphi(f)-\mu(f)|
&\le
|\zeta \circ \varphi(f)-\zeta\circ \varphi(f')|+|\zeta\circ \varphi(f')-\mu(f')|+|\mu(f')-\mu(f)|\\
&<
\|\varphi(f)-\varphi(f')\|_2+\frac{\delta}{2}+\|f-f'\|_2\\
&\le
3\|f-f'\|_2+\frac{\delta}{2}< \delta.
\end{align*}

Let $t\in F$.  For $(h,s)\in \{q_1, \dots, q_n\}^n \times D^n$ we have,
using the almost multiplicativity of $\varphi$ and our assumption that $F$ contains $e$,
\begin{align*}
\lefteqn{\bigg\| \varphi\circ\alpha_t \bigg(\prod_{k=1}^n \alpha_{s_k} (h_k )\bigg)
- \sigma_t \circ\varphi \bigg(\prod_{k=1}^n \alpha_{s_k} (h_k )\bigg) \bigg\|_2}\hspace*{20mm} \\
\hspace*{20mm} &\le \bigg\| \varphi\bigg(\prod_{k=1}^n \alpha_{ts_k} (h_k )\bigg)-
\prod_{k=1}^n \varphi(\alpha_{ts_k} (h_k ))\bigg\|_2\\
&\hspace*{15mm}\ +\bigg\| \prod_{k=1}^n \varphi\circ\alpha_t (\alpha_{s_k} (h_k ))
- \prod_{k=1}^n \sigma_t \circ\varphi (\alpha_{s_k} (h_k )) \bigg\|_2 \\
&\hspace*{15mm}\ + \bigg\|\sigma_t \bigg(\prod_{k=1}^n \varphi (\alpha_{s_k} (h_k ))\bigg)-
\sigma_t \circ\varphi \bigg(\prod_{k=1}^n \alpha_{s_k} (h_k )\bigg) \bigg\|_2 \\
&\leq \delta'+\sum_{k=1}^n \| \varphi \circ\alpha_t (\alpha_{s_k} (h_k ))
- \sigma_t \circ\varphi (\alpha_{s_k} (h_k )) \|_2 +\delta'\\
&\leq 2\delta'+\sum_{k=1}^n \big( \| \varphi\circ\alpha_{ts_k} (h_{s_k} )
- \sigma_{ts_k} \circ\varphi (h_{s_k} ) \|_2 \\
&\hspace*{25mm}\ + \| (\sigma_{ts_k} - \sigma_t \circ\sigma_{s_k} )(\varphi (h_{s_k} )) \|_2 \\
&\hspace*{25mm}\ +
\| \sigma_t (\sigma_{s_k} \circ\varphi (h_{s_k} ) - \varphi\circ\alpha_{s_k} (h_{s_k} )) \|_2 \big) \\
&\leq (2+3n)\delta'
\end{align*}
assuming that $\sigma$ is a good enough sofic approximation.
Thus given a $p\in\{p_1, \dots, p_m\}$, since $p\in \cS_{F, m}$ and $q_1=1$ we have
\begin{align*}
\| \varphi\circ\alpha_t (p' ) - \sigma_t \circ\varphi (p' ) \|_2
&\leq \sum_{g\in \cT_{D, n}} |c_{p, g}| \cdot \|\varphi\circ\alpha_t(g)-\sigma_t \circ\varphi(g)\|_2\\
&\le M_1\sum_{g\in \cT_{D, n}} \|\varphi\circ\alpha_t(g)-\sigma_t \circ\varphi(g)\|_2\\
&\leq n^n|D|^nM_1(2+3n)\delta'<\frac{\delta}{3},
\end{align*}
and hence, using that $\varphi$ has norm at most $2$ with respect to the $L^2$-norms,
\begin{align*}
\| \varphi\circ\alpha_t (p) - \sigma_t \circ\varphi (p) \|_2
&\leq  \| \varphi\circ\alpha_t (p-p' ) \|_2 + \| \varphi\circ\alpha_t (p') - \sigma_t \circ\varphi (p') \|_2 \\
&\hspace*{15mm} \ + \| \sigma_t \circ\varphi (p'- p) \|_2 \\
&< 4\| p-p' \|_2 + \frac{\delta}{3} \\
&< \frac{2\delta}{3} + \frac{\delta}{3} = \delta .
\end{align*}
Therefore
$\varphi\in\UP_{\mu, 2}(\cS ,F, m, \delta , \sigma )$, and so
$\UP_{\mu, 2}(\cT ,FD,mn, \delta' , \sigma )\subseteq\UP_{\mu, 2} (\cS ,F, m, \delta , \sigma )$.

Let $\varphi$ and $\psi$ be elements of $\UP_{\mu, 2}(\cT ,FD,mn, \delta' , \sigma )$ such that
$\rho_\cS (\varphi , \psi ) \le \varepsilon'$.
Then for $(h,s)\in \{p_1, \dots, p_\ell\}^\ell \times F^\ell$ we have
\begin{align*}
\lefteqn{\bigg\| \varphi \bigg( \prod_{k=1}^\ell \alpha_{s_k} (h_k )\bigg)
- \psi \bigg( \prod_{k=1}^\ell \alpha_{s_k} (h_k )\bigg) \bigg\|_2}\hspace*{10mm} \\
\hspace*{10mm} &\leq \bigg\|\varphi \bigg( \prod_{k=1}^\ell \alpha_{s_k} (h_k )\bigg)-
\prod_{k=1}^\ell \varphi (\alpha_{s_k} (h_{s_k} ))\bigg\|_2 \\
&\hspace*{20mm} \ + \bigg\| \prod_{k=1}^\ell \varphi (\alpha_{s_k} (h_{s_k} ))
- \prod_{k=1}^\ell \psi (\alpha_{s_k} (h_{s_k} )) \bigg\|_2  \\
&\hspace*{20mm} \ +\bigg\|\prod_{k=1}^\ell \psi (\alpha_{s_k} (h_{s_k} ))-
\psi \bigg( \prod_{k=1}^\ell \alpha_{s_k} (h_k )\bigg) \bigg\|_2\\
&\leq \delta +\sum_{k=1}^\ell \| \varphi (\alpha_{s_k} (h_{s_k} )) - \psi (\alpha_{s_k} (h_{s_k} )) \|_2 +\delta\\
&\leq 2\delta+\sum_{k=1}^\ell \big( \| \varphi (\alpha_{s_k} (h_{s_k} )) - \sigma_{s_k} (\varphi (h_{s_k} )) \|_2
+ \| \sigma_{s_k} (\varphi (h_{s_k} ) - \psi (h_{s_k} )) \|_2 \\
&\hspace*{40mm}\ + \| \sigma_{s_k} (\psi (h_{s_k} )) - \psi (\alpha_{s_k} (h_{s_k} )) \|_2 \big) \\
&\leq 2\delta+(2\delta + 2^\ell\varepsilon' )\ell \leq 2^{\ell+2}\varepsilon'\ell,
\end{align*}
so that for $q\in\{q_1, \dots, q_R\}$, using the fact that $F\supseteq E$ and $p_1=1$,
\begin{align*}
\| \varphi (q' ) - \psi (q' ) \|_2
&\leq
\sum_{g\in \cS_{E, \ell}}|d_{q, g}|\cdot \|\varphi(g)-\psi(g)\|_2\\
&\leq M \sum_{g\in \cS_{E, \ell}}\|\varphi(g)-\psi(g)\|_2\\
&\leq M\ell^\ell |E|^\ell 2^{\ell+2}\varepsilon' \ell \\
&= \frac{\varepsilon}{4R} .
\end{align*}
Since $\varphi$ and $\psi$ have norms at most $2$ with respect to the $L^2$-norms, we thus obtain
\begin{align*}
\rho_\cT (\varphi , \psi ) &= \sum_{j=1}^{\infty} \frac{1}{2^j}\| \varphi (q_j) - \psi (q_j) \|_2 \\
&\le \sum_{j=1}^{R} \| \varphi (q_j) - \psi (q_j) \|_2 +2^{-(R-1)}\\
&\leq \sum_{j=1}^{R} \big( \| \varphi (q_j - q'_j ) \|_2
+ \| \varphi (q'_j ) - \psi (q'_j ) \|_2 + \| \psi (q'_j - q_j) \|_2 \big) +2^{-(R-1)}\\
&\leq 4\sum_{j=1}^{R} \| q_j - q'_j \|_2 + \frac{\varepsilon}{4} +2^{-(R-1)}\\
&< \frac{\varepsilon}{3} + \frac{\varepsilon}{4}+\frac{\varepsilon}{3}< \varepsilon .
\end{align*}
Thus any subset of $\UP_{\mu, 2}(\cT ,FD, mn, \delta' ,\sigma )$ which is $\varepsilon$-separated
with respect to $\rho_\cT$ is $\varepsilon'$-separated with respect to $\rho_\cS$,
and so
\begin{align*}
N_\varepsilon (\UP_{\mu, 2}(\cT ,FD, mn, \delta' ,\sigma ),\rho_\cT )\leq N_{\varepsilon'} (\UP_{\mu, 2}(\cS ,F, m, \delta , \sigma ),\rho_\cS ).
\end{align*}
Consequently \eqref{E-gen comparison1} holds,
as desired.
\end{proof}

In view of Theorem~\ref{T-gen comparison} we can define the measure entropy
of our system with respect to $\Sigma$ as follows.

\begin{definition}\label{D-global}
The measure entropy $h_{\Sigma ,\mu} (X,G)$ of the system $(X,\mu ,G)$ with respect to $\Sigma$ is defined as
the common value of $h_{\Sigma ,\mu} (\cS )$ over all dynamically generating
sequences $\cS$ in the unit ball of $L^\infty_{\Rb} (X,\mu )$.
\end{definition}

It follows from Theorem~\ref{T-gen comparison}, or even directly from Definition~\ref{D-entropy},
that $h_{\Sigma ,\mu} (\cS )$ depends only on the image of $\cS$ as a function on $\Nb$.
We can thus define the entropy $h_{\Sigma ,\mu} (\cP )$ of a countable subset $\cP$ of the unit ball of
$L^\infty_\Rb (X,\mu )$ as the common value of $h_{\Sigma ,\mu} (\cS )$ over all sequences $\cS$
whose image as a function on $\Nb$ is equal to $\cP$.
For a finite partition of unity $\cP\subseteq L^\infty (X,\mu )$,
we do not need the sequential formalism to define $h_{\Sigma ,\mu} (\cP )$ and can proceed more
simply as follows. For a nonempty finite set $F\subseteq G$ and $m\in\Nb$, we write $\cP_{F,m}$ for the set of
all products of the form $\alpha_{s_1} (p_1 )\cdots \alpha_{s_j} (p_j )$ where $1\leq j\leq m$,
$p_1 , \dots p_j \in\cP$, and $s_1 , \dots ,s_j \in F$. We write $\cP_F$ for the set
of all products of the form $\prod_{s\in F} \alpha_s (p_s )$ for $p\in \cP^F$.
For a $d\in\Nb$ we define on the set of unital positive maps
from some unital self-adjoint linear subspace of $L^\infty (X,\mu )$ containing $\spn (\cP )$ to $\Cb^d$ the pseudometric
\begin{align*}
\rho_{\cP} (\varphi , \psi ) &= \max_{p\in\cP} \| \varphi (p) - \psi (p) \|_2 .
\end{align*}

\begin{definition}\label{D-finite}
Let $\sigma$ be a map from $G$ to $\Sym (d)$ for some $d\in\Nb$.
Let $F$ be a nonempty finite subset of $G$ and $\delta > 0$.
Let $\cP$ be a finite partition of unity in $L^\infty (X,\mu )$.
Define $\UP_\mu (\cP , F,m,\delta ,\sigma )$ to be the set of all unital positive linear maps
$\varphi : L^\infty (X,\mu ) \to \Cb^d$ such that
\begin{enumerate}
\item[(i)] $\| \varphi (\alpha_{s_1} (f_1 ) \cdots \alpha_{s_m} (f_m ))
- \varphi (\alpha_{s_1} (f_1 ))\cdots\varphi (\alpha_{s_m} (f_m )) \|_2 < \delta$
for all $f_1 , \dots ,f_m \in \cP$ and $s_1 , \dots ,s_m \in F$,

\item[(ii)] $| \zeta\circ\varphi (f) - \mu (f) | < \delta$ for all $f\in\cP_{F,m}$,

\item[(iii)] $\| \varphi\circ\alpha_s (f) - \sigma_s \circ\varphi (f) \|_2 < \delta$ for all $f\in\cP$ and $s\in F$.
\end{enumerate}
In the case of a finite partition of unity $\cP\subseteq L^\infty (X,\mu )$ consisting of projections we define
$\Hom_\mu (\cP , F,\delta ,\sigma )$ to be the set of all unital homomorphisms
$\varphi : \spn (\cP_F ) \to \Cb^d$ such that
\begin{enumerate}
\item[(i)] $| \zeta\circ\varphi (f) - \mu (f) | < \delta$ for all $f\in\cP_F$,

\item[(ii)] $\| \varphi\circ\alpha_s (f) - \sigma_s \circ\varphi (f) \|_2 < \delta$ for all $f\in\cP$ and $s\in F$.
\end{enumerate}
We define $h_\Sigma^\varepsilon (\cP ,F,m,\delta )$,
$h_\Sigma^\varepsilon (\cP ,F,m)$, $h_\Sigma^\varepsilon (\cP ,F)$, $h_\Sigma^\varepsilon (\cP )$, and
$h_\Sigma (\cP )$ by formally substituting $\cP$ for $\cS$ in Definition~\ref{D-entropy}.
\end{definition}

It is readily verified that $h_\Sigma (\cP )$ as defined above is equal to $h_\Sigma (\cS )$ for any sequence $\cS$
whose image as a function on $\Nb$ is equal to $\cP$, and so the notation $h_\Sigma (\cP )$ is unambiguous.

\section{Comparison with Bowen's measure entropy}\label{S-comparison}

Here we show that the measure entropy in Section~\ref{S-measure} agrees with that defined by Bowen in \cite{Bowen10}
when there exists a generating measurable partition with finite entropy.
Recall that the entropy $H_{\mu}(\cP)$ of a measurable partition $\cP$ of $X$ is defined as
$-\sum_{p\in\cP} \mu (p) \log \mu (p)$.

We write $\AP (\cP ,F,\delta ,\sigma )$ for the set of approximating ordered partitions as in \cite{Bowen10},
and $h_{\Sigma ,\mu}' (\cP , F, \delta )$, $h_{\Sigma ,\mu}' (\cP , F)$, and $h_{\Sigma ,\mu}' (\cP )$,
for the entropy quantities in \cite{Bowen10}. Bowen proved that the entropy $h_{\Sigma ,\mu}' (\cP )$
takes a common value over all generating measurable partitions $\cP$ of $X$ with $H_{\mu}(\cP) < +\infty$.
The entropy of the system with respect to $\Sigma$, which we will denote here by $h_{\Sigma ,\mu}' (X,G)$,
is defined as this common value
in the case that there exists a generating measurable partition $\cP$ of $X$ with $H_{\mu}(\cP) < +\infty$.
Other notation is carried over from the previous section. In particular, for a finite partition of
unity $\cP$ consisting of projections and
a nonempty finite set $F\subseteq G$, $\cP_F$ denotes the set of all products of the form
$\prod_{s\in F} \alpha_s (p_s )$ where $p_s \in\cP$ for each $s\in F$.


\begin{lemma}\label{L-prime}
Let $\cP$ be a finite measurable partition of $X$ and $F$ a finite subset of $G$ containing $e$.
Then
\[
h_{\Sigma ,\mu}' (\cP ,F) = \inf_{\delta > 0} \limsup_{i\to\infty} \frac{1}{d_i} N_0 (\Hom_\mu (\cP , F,\delta ,\sigma_i ),\rho_{\cP})  .
\]
\end{lemma}

\begin{proof}
Let us first show that
\[
h_{\Sigma ,\mu}' (\cP ,F)
\geq \inf_{\delta > 0} \limsup_{i\to\infty} \frac{1}{d_i} N_0 (\Hom_\mu (\cP , F,\delta ,\sigma_i ),\rho_{\cP}) .
\]
Let $\sigma$ be a map from $G$ to $\Sym (d)$ for some $d\in\Nb$.
Let $\delta > 0$ and $\varphi\in\Hom_\mu (\cP , F,\delta ,\sigma )$. Write
$F = \{ s_1 , \dots ,s_\ell \}$. Then for every $r\in\cP^F$ we have
\begin{align*}
\lefteqn{\bigg| \zeta \bigg( \prod_{k=1}^\ell \sigma_{s_k} \circ\varphi (r_{s_k} )
- \prod_{k=1}^\ell \varphi\circ\alpha_{s_k} (r_{s_k} )\bigg) \bigg|}\hspace*{20mm} \\
\hspace*{20mm} &\leq \bigg\| \prod_{k=1}^\ell \sigma_{s_k} \circ\varphi (r_{s_k} )
- \prod_{k=1}^\ell \varphi\circ\alpha_{s_k} (r_{s_k} ) \bigg\|_2 \\
&\leq \sum_{k=1}^\ell \| \varphi\circ\alpha_{s_1} (r_{s_1} ) \cdots\varphi\circ\alpha_{s_{k-1}} (r_{s_{k-1}} )
(\sigma_{s_k} \circ\varphi (r_{s_k} ) - \varphi\circ\alpha_{s_k} (r_{s_k} )) \\
&\hspace*{40mm} \ \times\sigma_{s_{k+1}} \circ\varphi (r_{s_{k+1}} ) \cdots \sigma_{s_l} \circ\varphi (r_{s_\ell} ) \|_2 \\
&\leq \sum_{k=1}^\ell \| \sigma_{s_k} \circ\varphi (r_{s_k} ) - \varphi\circ\alpha_{s_k} (r_{s_k} ) \|_2 \\
&< |F| \delta
\end{align*}
and hence
\begin{align*}
\lefteqn{\sum_{r\in\cP^F} \bigg| \zeta \bigg( \prod_{k=1}^\ell \sigma_{s_k} (\varphi (r_{s_k} ))\bigg) -
\mu \bigg( \prod_{k=1}^\ell \alpha_{s_k} (r_{s_k} )\bigg) \bigg|}
\hspace*{20mm} \\
\hspace*{20mm} &\leq \sum_{r\in\cP^F} \bigg( \bigg| \zeta \bigg( \prod_{k=1}^\ell \sigma_s \circ\varphi (r_{s_k} )
- \prod_{k=1}^\ell \varphi\circ\alpha_{s_k} (r_{s_k} )\bigg) \bigg| \\
&\hspace*{30mm} \ + \bigg| \zeta\circ\varphi \bigg( \prod_{k=1}^\ell \alpha_{s_k} (r_{s_k} )\bigg)
- \mu \bigg( \prod_{k=1}^\ell \alpha_{s_k} (r_{s_k} )\bigg) \bigg| \bigg) \\
&< |\cP |^{|F|} (|F| + 1)\delta ,
\end{align*}
so that the partition $\varphi (\cP )$, ordered so as to reflect a fixed ordering of $\cP$, lies in
$\AP (\cP ,F,|\cP |^{|F|} (|F| + 1)\delta ,\sigma )$. Since for any $\varphi ,\psi\in\Hom_\mu (\cP , F,\delta ,\sigma )$
with $\rho_{\cP} (\varphi , \psi ) > 0$ the partitions $\varphi (\cP )$ and $\psi (\cP )$ are distinct, it follows that
\[
N_0 (\Hom_\mu (\cP , F,\delta ,\sigma ), \rho_\cP ) \leq | \AP (\cP ,F,|\cP |^{|F|} (|F| + 1)\delta ,\sigma ) |
\]
and hence
\[
\limsup_{i\to\infty} \frac{1}{d_i} N_0 (\Hom_\mu (\cP , F,\delta ,\sigma_i ),\rho_\cP )
\leq h_{\Sigma ,\mu}' (\cP ,F,|\cP |^{|F|} (|F| + 1)\delta ) .
\]
Taking infima over all $\delta > 0$ then yields the desired inequality.

For the reverse inequality, let $\delta > 0$ and write $\cP = \{ p_1 , \dots ,p_n \}$.
Let $\sigma$ be a map from $G$ to $\Sym (d)$ for some $d\in\Nb$ which is good enough sofic approximation
for our purposes below.
Let $\delta'$ be a positive number less than $\delta/3$ which will be further
specified below as a function of $\delta$.
Let $\cQ = \{ q_1 , \dots , q_n \} \in\AP (\cP ,F,\delta' ,\sigma )$.
Define a unital homomorphism $\varphi : \spn (\cP_F )\to\Cb^d$ as follows. First we set
\[
\varphi \bigg( \prod_{k=1}^\ell \alpha_{s_k} (p_{\gamma (k)} )\bigg) =
\prod_{k=1}^\ell \sigma_{s_k} (q_{\gamma (k)} ) \]
for all $\gamma\in \{ 1, \dots ,n\}^{\{ 1,\dots ,\ell \}}$ such that
$\prod_{k=1}^\ell \alpha_{s_k} (p_{\gamma (k)} ) \neq 0$. Write $W$ for the set of all
$\gamma\in \{ 1, \dots ,n\}^{\{ 1,\dots ,\ell \}}$ such that
$\prod_{k=1}^\ell \alpha_{s_k} (p_{\gamma (k)} ) = 0$ but $\prod_{k=1}^\ell \sigma_{s_k} (q_{\gamma (k)} ) \neq 0$.
Set $r = \sum_{\gamma\in W} \prod_{k=1}^\ell \sigma_{s_k} (q_{\gamma (k)} )$.
It is easy to see that by shrinking $\delta'$ if necessary
we can arrange that $\| r \|_2 < (12\ell )^{-1} \delta$. In the case that $W\neq\emptyset$
we take a $\gamma_0 \in \{ 1, \dots ,n\}^{\{ 1,\dots ,\ell \}} \setminus W$ and redefine $\varphi$ on
$\prod_{k=1}^\ell \alpha_{s_k} (p_{\gamma_0 (k)})$ to be $r + \prod_{k=1}^\ell \sigma_{s_k} (q_{\gamma (k)} )$.
This produces the desired $\varphi$.

Now let $s\in F\setminus\{ e \}$ and $1\leq i\leq n$. By relabeling the elements of $F$ we may assume
that $s_1 = s$ and $s_\ell = e$. Then
\begin{align*}
\| \varphi (p_i ) - q_i \|_2
&\leq \bigg\| \sum_{\gamma\in\{ 1, \dots ,n\}^{\{ 1,\dots ,\ell - 1 \}}}
\bigg[ \varphi \bigg( \alpha_e(p_i) \prod_{k=1}^{\ell-1} \alpha_{s_k} (p_{\gamma (k)} ) \bigg) -
\sigma_e (q_i )\prod_{k=1}^{\ell-1} \sigma_{s_k} (q_{\gamma (k)} ) \bigg] \bigg\|_2 \\
&\hspace*{40mm} \ + \| \sigma_e (q_i ) - q_i \|_2 \\
&\leq \| r \|_2 + \frac{\delta}{6 \ell}
\end{align*}
assuming that $\sigma$ is a good enough sofic approximation to ensure that $\sigma_e$ is sufficiently
close to the identity permutation, and hence
\begin{align*}
\lefteqn{\| \varphi\circ\alpha_s (p_i ) - \sigma_s \circ\varphi (p_i ) \|_2} \hspace*{6mm} \\
&\leq \| \varphi\circ\alpha_s (p_i ) - \sigma_s (q_i ) \|_2 + \| \sigma_s (q_i - \varphi (p_i )) \|_2 \\
\hspace*{5mm} &\leq \bigg\| \sum_{\gamma\in\{ 1, \dots ,n\}^{\{ 2,\dots ,\ell \}}}
\bigg[ \varphi \bigg( \alpha_s (p_i ) \prod_{k=2}^\ell \alpha_{s_k} (p_{\gamma (k)} ) \bigg) -
\sigma_s (q_i ) \prod_{k=2}^\ell \sigma_{s_k} (q_{\gamma (k)} )\bigg] \bigg\|_2 + \| r \|_2 + \frac{\delta}{6 \ell}\\
&\leq 2\| r \|_2 + \frac{\delta}{6 \ell} < \frac{\delta}{3\ell} .
\end{align*}
Assuming that $\sigma$ is a good enough sofic approximation, we also have
\begin{align*}
\| \varphi\circ\alpha_e (p_i ) - \sigma_e \circ\varphi (p_i ) \|_2<\frac{\delta}{3\ell}.
\end{align*}
Moreover, for every $\gamma \in \{1, \dots, n\}^{\{1, \dots, \ell \}}$,
\begin{align*}
\lefteqn{\bigg| \zeta\circ\varphi \bigg( \prod_{k=1}^\ell \alpha_{s_k} (p_{\gamma (k)} )\bigg) -
\mu \bigg( \prod_{k=1}^\ell \alpha_{s_k} (p_{\gamma (k)} ) \bigg) \bigg|}\hspace*{10mm} \\
\hspace*{15mm} &\leq \bigg| \zeta \bigg( \prod_{k=1}^\ell \varphi \circ\alpha_{s_k} (p_{\gamma (k)} )  -
\prod_{k=1}^\ell \sigma_{s_k} \circ\varphi (p_{\gamma (k)} )\bigg) \bigg| \\
&\hspace*{20mm} \ + \bigg| \zeta \bigg( \prod_{k=1}^\ell \sigma_{s_k} (\varphi (p_{\gamma (k)} ))-\prod_{k=1}^\ell\sigma_{s_k}(q_{\gamma (k)} )\bigg) \bigg| \\
&\hspace*{20mm} \ + \bigg| \zeta \bigg( \prod_{k=1}^\ell \sigma_{s_k} (q_{\gamma (k)} ) \bigg)
- \mu \bigg( \prod_{k=1}^\ell \alpha_{s_k} (p_{\gamma (k)} ) \bigg) \bigg| \\
&\leq \sum_{k=1}^\ell \| \varphi\circ\alpha_{s_k} (p_{\gamma (k)} ) - \sigma_{s_k} \circ\varphi (p_{\gamma (k)} ) \|_2
+ \sum_{k=1}^\ell \| \varphi (p_{\gamma (k)} ) - q_{\gamma (k)} ) \|_2 + \delta' \\
&< \frac{\delta}{3} + \ell (\| r \|_2+\frac{\delta}{6\ell}) + \frac{\delta}{3} < \delta .
\end{align*}
Thus $\varphi\in\Hom_\mu (\cP , F,\delta ,\sigma )$.

We define a map $\Gamma : \AP (\cP ,F,\delta' ,\sigma )\to\Hom_\mu (\cP , F,\delta ,\sigma )$
by declaring $\Gamma (\cQ )$ to be the element $\varphi$ we constructed above. Given a
$\varphi\in\Hom_\mu (\cP , F,\delta ,\sigma )$, we wish to obtain an upper bound on
the number of partitions in $\AP (\cP ,F,\delta' ,\sigma )$ whose image under $\Gamma$
agrees with $\varphi$ on $\cP$.
Suppose that $\cQ = \{ q_1 , \dots , q_n \}$ and $\cQ' = \{ q_1' , \dots ,q_n' \}$ are two such partitions.
Then for each $i=1,\dots ,n$ we have
\[
\| q_i - q_i' \|_2  \leq \| q_i - \varphi (p_i ) \|_2 + \| \varphi (p_i ) - q_i' \|_2 \leq 2(\| r \|_2+\frac{\delta}{6\ell})
< \delta
\]
so that $q_i$ and $q_i'$ differ at at most $d\delta^2$ coordinates.
It follows that the number of partitions in
$\AP (\cP ,F,\delta' ,\sigma )$ whose image under $\Gamma$ agrees with $\varphi$ on $\cP$
is at most the $n$th power of $\binom{d}{d\delta^2} 2^{d\delta^2}$.
By Stirling's approximation this number
is bounded above by $a\exp(\kappa d)$ for some $a,\kappa > 0$ not depending on $d$ with $\kappa\to 0$ as
$\delta\to 0$. Consequently
\[ |\AP (\cP ,F,\delta' ,\sigma )| \leq a\exp(\kappa d) N_0 (\Hom_\mu (\cP , F,\delta ,\sigma ),\rho_{\cP} ) \]
and thus
\[
h_{\Sigma ,\mu}' (\cP ,F,\delta' ) \leq
\limsup_{i\to\infty} \frac{1}{d_i} N_0 (\Hom_\mu (\cP , F,\delta ,\sigma_i ),\rho_{\cP} ) + \kappa .
\]
Taking an infimum over all $\delta > 0$ then yields
\[
h_{\Sigma ,\mu}' (\cP ,F)
\leq \inf_{\delta > 0} \limsup_{i\to\infty} \frac{1}{d_i} N_0 (\Hom_\mu (\cP , F,\delta ,\sigma_i ),\rho_{\cP}) ,
\]
completing the proof.
\end{proof}

%
%
%

Let $\cP$ be a countable measurable partition of $X$ with $H_{\mu}(\cP)<+\infty$.
We fix an enumeration $p_1 ,p_2 , \dots$ of the elements of $\cP$ and thereby regard
$\cP$ as a sequence in the unit ball of $L^{\infty}_{\Rb}(X, \mu)$. In the case that $\cP$ is finite
we take the tail of this enumeration to be constantly zero after we have exhausted the elements of $\cP$.
For each $n\in \Nb$, denote by $\cP_n$ the finite partition of $X$ consisting of $p_1, \dots, p_{n-1}$, and $\bigcup_{k=n}^{\infty}p_k$.
Then $\cP_1\le \cP_2\le \dots$ and $\bigvee_{n\in \Nb}\cP_n=\cP$. Thus $\{\cP_n\}_{n=1}^\infty$ is a chain of $\cP$
in the sense of \cite[Defn.\ 13]{Bowen10}.

\begin{lemma} \label{L-infinite proj}
Let $\cP = \{ p_n \}_{n=1}^\infty$ be a countable measurable partition of $X$ with $H_{\mu}(\cP)<+\infty$.
For every $\kappa > 0$ there is an $\varepsilon > 0$ such that
\[
\limsup_{n\to \infty}\inf_{\delta'>0}\limsup_{i\to\infty} \frac{1}{d_i} \log N_0 (\Hom_\mu (\cP_n , F, \delta' ,\sigma_i ),\rho_{\cP_n})
\leq h_{\Sigma ,\mu}^\varepsilon (\cP ,F,m, \delta ) + \kappa
\]
for all finite set $F\subseteq G$ containing $e$, $m\in \Nb$, and $\delta > 0$.
\end{lemma}

\begin{proof}
%
Set $\xi(t)=-t\log t$ for all $0\le t\le 1$.
Since $H_{\mu}(\cP)<+\infty$, we can find
an $\ell\in \Nb$ such that $\sum_{k=\ell+1}^{\infty}\xi(\mu(p_k))+\xi(1-\sum_{k=\ell+1}^{\infty}\mu(p_k))<\kappa/4$.
Let $\varepsilon$ be a positive number to be determined in a moment.
Let $F$ be a finite subset of $G$ containing $e$, $m\in \Nb$,
and $\delta > 0$.


Let $n\in \Nb$ be such that $n>\max(m, \ell)$.
Note that $\spn((\cP_n)_F)\supseteq \spn(\cP_{F, m})$ and $\{p_1, \dots, p_{\max(m, \ell)}\}\subseteq \cP_n$.
Let $\delta' \in (0,\delta ]$ be a small positive number depending on $n$ which we will determine in a moment.
Let $\sigma$ be a map from $G$ to $\Sym (d)$ for some $d\in\Nb$. Note that
for each $\varphi\in \Hom_\mu (\cP_n ,F,\delta'/n^{|F|} ,\sigma )$ the map
$\Gamma(\varphi):=\varphi\circ \Eb(\cdot|\spn((\cP_n)_F))$ is in $\UP_{\mu}(\cP, F, m, \delta, \sigma)$, where
$\Eb(\cdot|\spn((\cP_n)_F))$ denotes the conditional expectation from $L^{\infty}(X, \mu)$ to $\spn((\cP_n)_F)$.
Thus we have a map $\Gamma:  \Hom_\mu (\cP_n ,F,\delta'/n^{|F|} ,\sigma )\rightarrow \UP_{\mu}(\cP, F, m, \delta, \sigma)$
sending $\varphi$ to $\Gamma(\varphi)$.

If $\varphi$ and $\psi$ are elements of $\Hom_\mu (\cP_n ,F,\delta'/n^{|F|} ,\sigma )$ satisfying
$\rho_{\cP} (\Gamma(\varphi) , \Gamma(\psi) ) < \varepsilon$,
then for each $j=1,\dots ,\ell$ we have $\| \varphi (p_j) - \psi (p_j) \|_2 < 2^\ell\varepsilon$ so that the projections
$\varphi (p_j)$ and $\psi (p_j)$ differ at at most $4^\ell\varepsilon^2 d$ places.
Set $c_j=\mu(p_j)$ for $\ell+1\le j\le n-1$ and $c_n=\mu(\bigcup_{k=n}^{\infty}p_k)$.
Note that for every $\varphi\in \Hom_\mu (\cP_n ,F,\delta'/n^{|F|} ,\sigma )$ one has
$|\zeta \circ \varphi(p_k)-c_k|<\delta'$ for all $\ell+1\le k\le n-1$ and
$|\zeta \circ \varphi(\bigcup_{k=n}^{\infty}p_k)-c_n|<\delta'$.
Then the $(\rho_\cP, \varepsilon)$-neighbourhood
of $\Gamma(\varphi)$ for any element $\varphi$ of $\Hom_\mu (\cP_n ,F,\delta'/n^{|F|} ,\sigma )$ contains the images of at most
$M_1M_2$
elements modulo the relation of zero $\rho_{\cP_n}$-distance,
where
\[
M_1=\bigg(\binom{d}{4^\ell\varepsilon^2 d} 2^{4^\ell\varepsilon^2 d}\bigg)^\ell,
\]
and
\[
M_2=\sum_{j_{\ell+1}, \dots, j_n}\binom{d}{j_{\ell+1}}\binom{d-j_{\ell+1}}{j_{\ell+2}}\cdots
\binom{d-\sum_{k=\ell+1}^{n-1}j_k}{j_n}
\]
with the sum ranging over all nonnegative integers $j_{\ell+1}, \dots, j_n$ such that $|j_k/d-c_k|<\delta'$ for all $\ell+1\le k\le n$
and $\sum_{k=\ell+1}^nj_k\le d$.
By Stirling's approximation, when $\varepsilon$ is small enough depending only on $\kappa$ and $\ell$,
one has $M_1\le  a_1\exp(\kappa d/2)$
for some $a_1> 0$ independent of $d$.
Also, 
by Stirling's approximation, for above $j_{\ell+1}, \dots, j_n$ one has
\begin{align*}
\lefteqn{\binom{d}{j_{\ell+1}}\binom{d-j_{\ell+1}}{j_{\ell+2}}\cdots \binom{d-\sum_{k=\ell+1}^{n-1}j_k}{j_n}} \hspace*{30mm}\\
\hspace*{30mm} &\le a_2 \exp\bigg(\bigg(\sum_{k=\ell+1}^n\xi(j_k/d)+\xi\bigg(1-\sum_{k=\ell+1}^nj_k/d\bigg)+\kappa/8\bigg)d\bigg)
\end{align*}
for some $a_2>0$ independent of $d$ and $j_{\ell+1}, \dots, j_n$.
Since the function $\xi$ is continuous and $\xi(t_1+t_2)\le \xi(t_1)+\xi(t_2)$ for all $t_1, t_2 \geq 0$
with $t_1+t_2\le 1$, we have
\begin{align*}
\sum_{k=\ell+1}^n\xi(c_k)+\xi\bigg(1-\sum_{k=\ell+1}^nc_k\bigg)\le \sum_{k=\ell+1}^{\infty}\xi(\mu(p_k))+
\xi\bigg(1-\sum_{k=\ell+1}^{\infty}\mu(p_k)\bigg)<\kappa/4.
\end{align*}
When $\delta'$ is small enough, one has $$\sum_{k=\ell+1}^n\xi(t_k)+\xi\bigg(1-\sum_{k=\ell+1}^nt_k\bigg)<\sum_{k=\ell+1}^n\xi(c_k)+\xi\bigg(1-\sum_{k=\ell+1}^nc_k\bigg)+\kappa/8$$ whenever
$|t_k-c_k|<\delta'$ for all $\ell+1\le k\le n$. Then
\begin{align*}
M_2&\le a_2(2\delta' d)^{n-\ell}\exp\bigg(\bigg(\sum_{k=\ell+1}^n\xi(c_k)+\xi\bigg(1-\sum_{k=\ell+1}^nc_k\bigg)
+\kappa/4\bigg)d\bigg)\\
&\le a_2(2\delta' d)^{n-\ell}\exp(\kappa d/2).
\end{align*}
Consequently
we obtain
\begin{align*}
N_0 (\Hom_\mu (\cP_n , F,\delta'/n^{|F|} ,\sigma ),\rho_{\cP_n})
\leq a_1a_2(2\delta' d)^{n-\ell} \exp(\kappa d) N_\varepsilon (\UP_\mu (\cP , F, m, \delta ,\sigma ),\rho_{\cP} ) ,
\end{align*}
from which the lemma follows.
\end{proof}

\begin{lemma}\label{L-mult perturbation}
Let $A$ be a unital commutative $C^*$-algebra, $\Omega$ a nonempty finite subset of $A$, and $\varepsilon > 0$.
Then there is a $\delta > 0$ such that whenever $d\in\Nb$ and
$\varphi : A\to\Cb^d$ is a unital positive linear map satisfying
$\| \varphi (f^* f) - \varphi (f)^*\varphi (f) \|_2 < \delta$
for all $f\in\Omega$ there exists a unital homomorphism $\tilde{\varphi} : A\to\Cb^d$
such that $\| \tilde{\varphi} (f) - \varphi (f)\|_2 < \varepsilon$ for all $f\in\Omega$.
\end{lemma}

\begin{proof}
First observe that, for every $\eta\in (0,1)$
and every unital positive linear map $\varphi : A\to\Cb^d \cong C( \{ 1,\dots ,d \} )$ satisfying
$\| \varphi (f^* f) - \varphi (f)^* \varphi (f) \|_2 < \eta$
for all $f\in\Omega$, there exists $J\subseteq \{ 1,\dots ,d \}$ with $|J| \geq (1-|\Omega|\eta )d$ such that
$|\varphi (f^* f)(a) - \varphi (f)^* \varphi (f)(a)| < \sqrt{\eta}$ for all $f\in\Omega$ and $a\in J$.
If $\eta$ is small enough
then, denoting by $P_I$ the canonical projection $\Cb^d \to \Cb^I$ for a set $I\subseteq \{ 1,\dots ,d\}$,
any unital positive linear map $\tilde{\varphi} : A\to\Cb^d$ such that $P_J \circ\tilde{\varphi} = P_J\circ \varphi$ and
$P_{\{1,\dots ,d\} \setminus J} \circ\tilde{\varphi}$ is a unital homomorphism will satisfy
$\| \tilde{\varphi} (f) - \varphi (f)\|_2 < \varepsilon /2$ for all $f\in\Omega$. Thus if we redefine
$\tilde{\varphi}$ so that for every $a\in J$ the state $f\mapsto \tilde{\varphi} (f)(a)$ on $C(X)$ is multiplicative
and $| \tilde{\varphi} (f)(a) - \varphi (f)(a) | < \varepsilon /2$ for all $f\in\Omega$, we will have
$\| \tilde{\varphi} (f) - \varphi (f)\|_2 < \varepsilon$ for all $f\in\Omega$, as desired. This reduces the
problem to proving the lemma statement for states, i.e., the case $d=1$.

Say $A=C(X)$ for some compact Hausdorff space $X$.
Suppose that for some $\delta > 0$ we have a state $\varphi : C(X)\to\Cb$ satisfying
$| \varphi (f^* f) - |\varphi (f)|^2 | < \delta$ for all $f\in\Omega$. The state $\varphi$
corresponds to a regular Borel probability measure $\mu$ on $X$, and the approximate multiplicativity
condition is easily seen to imply the existence of an $\eta > 0$ with $\eta\to 0$ as $\delta\to 0$ such that
for each $f\in C(X)$ there exists a set $A_f \subseteq\Cb$ of diameter at most $\eta$
for which $\mu (f^{-1} (A_f )) \geq 1-\eta /|\Omega |$, in which case
$\mu ( \bigcap_{f\in\Omega} f^{-1} (A_f )) \geq 1 - \eta$. Thus if $\eta$ is small enough, which can ensure
by assuming $\delta$ to be sufficiently small, any multiplicative state $\tilde{\varphi} : C(X)\to\Cb$
defined by evaluation at some point in $\bigcap_{f\in\Omega} f^{-1} (A_f )$ will satisfy
$| \tilde{\varphi} (f) - \varphi (f) | < \varepsilon$ for all $f\in\Omega$, as desired.
\end{proof}

\begin{lemma} \label{L-infinite proj2}
Let $\cP = \{ p_n \}_{n=1}^\infty$ be a countable measurable partition of $X$.
Let $F$ be a finite subset of $G$ containing $e$ and let $\varepsilon>0$. Then
\[
\liminf_{n\to \infty}\inf_{\delta>0}\limsup_{i\to\infty} \frac{1}{d_i} \log N_{\varepsilon/2} (\Hom_\mu (\cP_n , F, \delta ,\sigma_i ),\rho_{\cP_n})
\geq h_{\Sigma ,\mu}^\varepsilon (\cP ,F).
\]
\end{lemma}

\begin{proof}
Let $n\in \Nb$ be such that $2^{-(n-2)}<\varepsilon/4$.
Let $\delta>0$. We will show
\begin{eqnarray} \label{E-proj Hom}
\limsup_{i\to\infty} \frac{1}{d_i} \log N_{\varepsilon/2} (\Hom_\mu (\cP_n , F, \delta ,\sigma_i ),\rho_{\cP_n})
\geq h_{\Sigma ,\mu}^\varepsilon (\cP ,F).
\end{eqnarray}

Set $m=\max(|F|, n-1)$.
Note that $(\cP_n)_F\subseteq \spn(1\cup \cP_{F, m})$.
Let $\sigma$ be a map from $G$ to $\Sym (d)$ for some $d\in\Nb$.
Given an $\eta > 0$,
by Lemma~\ref{L-mult perturbation} there is a $\delta' > 0$ not depending on $d$ and $\sigma$
such that for every $\varphi\in\UP_\mu (\cP ,F, m^2,\delta' ,\sigma )$
there is a unital homomorphism $\tilde{\varphi} : \spn ((\cP_n)_F ) \to\Cb^d$ for which
$\| \tilde{\varphi} (f) - \varphi (f) \|_2 < \min (\eta , \varepsilon /(8(n-1)))$ for all $f\in(\cP_n)_F$.
By taking $\eta$ and $\delta'$ small enough this will imply that
$\tilde{\varphi} \in \Hom_\mu (\cP_n ,F,\delta ,\sigma )$.
Define a map $\Gamma : \UP_\mu (\cP ,F, m^2, \delta' ,\sigma )\to\Hom_\mu (\cP ,F,\delta ,\sigma )$
by $\Gamma (\varphi ) = \tilde{\varphi}$.

For all
$\varphi , \psi\in \UP_\mu (\cP ,F, m^2, \delta' ,\sigma )$ we have
\begin{align*}
\rho_\cP(\varphi, \psi)&=\sum_{j=1}^{\infty} \frac{1}{2^j} \| \varphi (p_j) - \psi (p_j) \|_2 \\
&\leq \sum_{j=1}^{n-1}  \frac{1}{2^j}\| \varphi (p_j) - \psi (p_j) \|_2 +\frac{1}{2^{n-2}}\\
&\leq \sum_{j=1}^{n-1} \frac{1}{2^j}\big(\| \varphi (p_j) - \tilde{\varphi} (p_j) \|_2 + \| \tilde{\varphi} (p_j) - \tilde{\psi} (p_j) \|_2
+ \| \tilde{\psi} (p_j) - \psi (p_j) \|_2\big) +\frac{\varepsilon}{4}\\
&<\frac{\varepsilon}{2} + \rho_{\cP_n}(\Gamma(\varphi), \Gamma(\psi)).
\end{align*}
Thus for any $(\rho_{\cP}, \varepsilon)$-separated subset $L$ of $\UP_\mu (\cP ,F, m^2, \delta' ,\sigma )$,
the set $\Gamma(L)$ is $(\rho_{\cP_n}, \varepsilon/2)$-separated.
Consequently
\[
N_{\varepsilon /2} (\Hom_\mu (\cP_n ,F,\delta ,\sigma ),\rho_{\cP_n} )
\ge N_\varepsilon (\UP_\mu (\cP ,F, m^2, \delta' ,\sigma ),\rho_\cP ).
\]
Therefore \eqref{E-proj Hom} holds.
\end{proof}

\begin{proposition}\label{P-infinite prime}
Let $\cP = \{ p_n \}_{n=1}^\infty$ a countable measurable partition of $X$ with $H_{\mu}(\cP)<+\infty$.
Then
\[
h'_{\Sigma, \mu}(\cP)=h_{\Sigma, \mu}(\cP).
\]
\end{proposition}

\begin{proof}
By Lemmas~\ref{L-infinite proj} and \ref{L-prime}, we have
\[
\inf_F\limsup_{n\to \infty}h'_{\Sigma, \mu}(\cP_n, F)\le h_{\Sigma, \mu}(\cP),
\]
where $F$ ranges over the nonempty finite subsets of $G$.
By Lemmas~\ref{L-infinite proj2} and \ref{L-prime}, for any finite subset $F$ of $G$ containing $e$ we have
\[
\liminf_{n\to \infty}h'_{\Sigma, \mu}(\cP_n, F)\ge h_{\Sigma, \mu}(\cP, F).
\]
Since
\[
h'_{\Sigma, \mu}(\cP)=\inf_F\lim_{n\to \infty}h'_{\Sigma, \mu}(\cP_n, F)
\]
where $F$ ranges over the nonempty finite subsets of $G$ \cite[Prop.\ 6.2]{Bowen10},
we obtain $h'_{\Sigma, \mu}(\cP)=h_{\Sigma, \mu}(\cP)$.
\end{proof}

In view of the definitions of $h_{\Sigma ,\mu} (X,G)$ and $h_{\Sigma ,\mu}' (X,G)$,
we obtain the following from the above local result.

\begin{theorem}\label{T-prime global}
Suppose that there is a generating measurable partition $\cP$ of $X$ with $H_{\mu}(\cP)<+\infty$.
Then
\[ h_{\Sigma ,\mu} (X,G) =  h_{\Sigma ,\mu}' (X,G) . \]
\end{theorem}

\begin{remark}\label{R-hom}
It follows from Lemmas~\ref{L-infinite proj} and \ref{L-infinite proj2}
that for a countably measurable partition $\cP = \{ p_n \}_{n=1}^\infty$ of $X$ with $H_{\mu}(\cP)<+\infty$
we can compute $h_{\Sigma ,\mu} (\cP )$ by counting unital homomorphisms, i.e.,
\[
h_{\Sigma ,\mu} (\cP ) = \inf_{F} \limsup_{n\to\infty} \inf_{\delta > 0}
\limsup_{i\to\infty} \frac{1}{d_i} \log N_0 (\Hom_\mu (\cP_n ,F,\delta ,\sigma_i ),\rho_{\cP_n} )
\]
where $F$ ranges over all nonempty finite subsets of $G$. In particular, when $\cP$ is a finite
measurable partition of $X$ we have
\[
h_{\Sigma ,\mu} (\cP ) = \inf_{F} \inf_{\delta > 0}
\limsup_{i\to\infty} \frac{1}{d_i} \log N_0 (\Hom_\mu (\cP ,F,\delta ,\sigma_i ),\rho_{\cP} )
\]
where $F$ ranges over all nonempty finite subsets of $G$.
\end{remark}

\section{Topological entropy}\label{S-top}

Throughout this section $X$ is a compact metrizable space
and $\alpha$ a continuous action of a countable sofic group $G$ on $X$.

By the Gelfand theory mentioned in the introduction, the unital $C^*$-subalgebras of $C(X)$
(i.e., the unital $^*$-subalgebras which are closed in the supremum norm)
correspond to the continuous quotients of $X$ via composition of functions.
The $G$-invariant unital $C^*$-subalgebras of $C(X)$ thus correspond to the dynamical factors of $X$.
A subset of $C(X)$ is said to be {\it dynamically generating}
if it is not contained in any proper $G$-invariant unital $C^*$-subalgebra of $C(X)$.

As in the measurable case, we will begin by defining the entropy $h_\Sigma (\cS)$
of a sequence $\cS=\{p_n\}_{n\in \Nb}$ in the unit ball of $C_{\Rb}(X)$. Given
a nonempty finite set $F\subseteq G$ and an $m\in \Nb$ we write $\cS_{F,m}$ for the set of all
products of the form $\alpha_{s_1} (f_1 ) \cdots \alpha_{s_j} (f_j )$ where $1\le j\le m$ and
$f_1 , \dots ,f_j \in \{p_1, \dots, p_m\}$ and $s_1 , \dots ,s_j \in F$.
For a given $d\in\Nb$ we define on the set of unital positive linear maps
from some unital self-adjoint linear subspace of $C(X)$ containing $\spn (\cS )$ to $\Cb^d$ the pseudometric
\begin{align*}
\rho_{\cS} (\varphi , \psi ) &= \sum_{n=1}^{\infty} \frac{1}{2^n}\| \varphi (p_n) - \psi (p_n) \|_2 .
\end{align*}


\begin{definition}\label{D-UP Hom top}
Let $\sigma$ be a map from $G$ to $\Sym (d)$ for some $d\in\Nb$.
Let $\cS=\{p_n\}_{n\in \Nb}$ be a sequence in the unit ball of $C_{\Rb}(X)$. Let $F$ be a nonempty finite subset
of $G$, $m\in\Nb$, and $\delta > 0$.
%
Define $\Hom (\cS , F,\delta ,\sigma )$ to be
the set of all unital homomorphisms $\varphi : C(X) \to \Cb^d$ such that
\[
\sum_{n=1}^{\infty} \frac{1}{2^n}\| \varphi\circ\alpha_s (p_n) - \sigma_s \circ\varphi (p_n) \|_2 < \delta
\]
for all $s\in F$.
\end{definition}

As before $N_\varepsilon (\cdot ,\rho )$ denotes the maximal cardinality of a finite $\varepsilon$-separated subset
with respect to the pseudometric $\rho$.
As in the case of measure entropy, for a sequence $\cS$ in the unit ball of $C_\Rb (X)$ we have
$N_\varepsilon (\Hom (\cS ,F,\delta ,\sigma ), \rho_{\cS} )
\geq N_{\varepsilon'} (\Hom (\cS ,F',\delta' ,\sigma), \rho_{\cS} )$
whenever $F\subseteq F'$, $\delta\geq\delta'$, and $\varepsilon\leq \varepsilon'$.

As usual $\Sigma = \{ \sigma_i : G \to \Sym (d_i ) \}_{i=1}^\infty$ is a fixed sofic approximation sequence.

\begin{definition}\label{D-entropy top}
Let $\cS$ be a sequence in the unit ball of $C_\Rb (X)$, $\varepsilon > 0$, $F$ a nonempty finite subset of $G$,
and $\delta > 0$. Define
\begin{align*}
h_\Sigma^\varepsilon (\cS ,F,\delta )
&= \limsup_{i\to\infty} \frac{1}{d_i} \log N_\varepsilon (\Hom (\cS ,F,\delta ,\sigma_i ), \rho_{\cS} ) ,\\
h_\Sigma^\varepsilon (\cS ,F) &= \inf_{\delta > 0} h_\Sigma^\varepsilon (\cS ,F,\delta ) ,\\
h_\Sigma^\varepsilon (\cS ) &= \inf_{F} h_\Sigma^\varepsilon (\cS ,F) ,\\
h_\Sigma (\cS ) &= \sup_{\varepsilon > 0} h_\Sigma^\varepsilon (\cS )
\end{align*}
where the infimum in the second last line is over all nonempty finite subsets of $G$.
If $\Hom (\cS ,F,\delta ,\sigma_i )$ is empty for all sufficiently large $i$, we set
$h_\Sigma^\varepsilon (\cS ,F,\delta ) = -\infty$.
\end{definition}

\begin{remark}\label{R-pou top}
As in the measurable case (Remark~\ref{R-pou}),
if we add $1$ to $\cS$ by setting $p'_1=1$ and $p'_{j+1}=p_j$ for all $j\in \Nb$, then
$h_\Sigma(\cS)=h_\Sigma(\cS')$.
\end{remark}

\begin{remark}\label{R-embedded sofic approx}
One can reformulate our definition of topological entropy at the space level as follows.
A unital homomorphism from $C(X)$ to $\Cb^d$ is given by a set of point evaluations indexed by
$\{ 1,\dots ,d \}$, and hence corresponds to a map from $\{ 1,\dots ,d \}$ to $X$.
Thus in the definition we are measuring the maximal cardinality of an $\varepsilon$-separated subset of the
set of maps $\{ 1,\dots ,d \} \to X$ which are approximately equivariant
with respect to the sofic approximation of $G$ on $\{ 1,\dots ,d \}$, where distance between these
maps is measured in an $\ell^2$ sense relative to a fixed continuous pseudometric $\rho$ on $X$
which is dynamically generating in the sense that for any distinct $x,y\in X$ one has $\rho (sx,sy) > 0$
for some $s\in G$.
This viewpoint also applies in the measure-theoretic context: in the unital positive linear
map framework of Section~\ref{S-measure} one is effectively dealing with approximately equivariant copies
of a sofic approximation inside the space of probability measures,
while in the next section we will show how to formulate measure entropy via homomorphisms
and hence by tracking points as in the topological case. Approximately equivariant maps from $\{ 1,\dots ,d \}$ to $X$
can be regarded as systems of interlocking approximate partial orbits, and in case of amenable $G$ they approximately
decompose into partial orbits over F{\o}lner sets \cite{KerLi10}.

We also remark that one could equivalently measure the distance between approximately equivariant maps
from  $\{ 1,\dots ,d \}$ to $X$ in an
$\ell^\infty$ sense, as Proposition~\ref{P-infinity norm} shows, but since sofic approximations
are statistical anyway it is more consistent to think
entirely in $\ell^2$ terms (or some other similar type of weak approximation) unless forced to do otherwise.
See also Section~4 of \cite{Li10} for the equivalence of these kinds of approximations
for the purpose of expressing classical dynamical entropy in the amenable case.
\end{remark}

%




\begin{theorem}\label{T-gen comparison top}
Let $\cS = \{ p_n \}_{n=1}^\infty$ and $\cT = \{ q_n \}_{n=1}^\infty$ be dynamically generating sequences
in the unit ball of  $C_{\Rb}(X)$. Then $h_\Sigma (\cT ) = h_\Sigma (\cS )$.
\end{theorem}

\begin{proof}
It suffices by symmetry to prove that $h_\Sigma (\cT ) \leq h_\Sigma (\cS )$. By Remark~\ref{R-pou top}
we may assume that $p_1=q_1=1$.

%
Let $\varepsilon>0$. Take an $R\in \Nb$ with $2^{-(R-1)}<\varepsilon/3$.
Since $\cS$ is dynamically generating and $p_1=1$, there are a nonempty finite set $E\subseteq G$ and an
$\ell\in\Nb$ such that for each $q\in \{ q_1 , \dots ,q_R \}$ there exist $d_{q,g} \in\Cb$
for $g\in \cS_{E, \ell}$ such that the function
\[
q' = \sum_{g\in \cS_{E, \ell}} d_{q,g}g
\]
satisfies $\| q - q' \|_{\infty} < (6R)^{-1}\varepsilon$.
Set $M = \max_{1\le j\le R}\max_{g\in \cS_{E, \ell}} | d_{q_j,g} |$ and
$\varepsilon'=\varepsilon/(2^{\ell+3} \ell^{\ell+1} |E|^\ell MR)$.
We will show that $h^{\varepsilon}_\Sigma (\cT ) \leq h^{\varepsilon'}_\Sigma (\cS )$.
Since $\varepsilon$ is an arbitrary positive number, this implies that $h_\Sigma (\cT ) \leq h_\Sigma (\cS )$.

Let $F$ be a finite subset of $G$ containing $e$ and $E$, and let $0<\delta\leq\varepsilon'/2$.
Take an $m\in \Nb$ with $2^{-(m-1)}<\delta/3$.

As $\cT$ is dynamically generating and $q_1=1$, there are a nonempty finite set $D\subseteq G$ and an $n\in\Nb$
such that for each $p\in\{p_1, \dots, p_m\}$ there exist
$c_{p,g} \in\Cb$ for $g\in \cT_{D, n}$ such that the function
\[
p' = \sum_{g\in \cT_{D, n}} c_{p,g} g
\]
satisfies $\| p - p' \|_{\infty} < (6m)^{-1} \delta$.
Set $M_1=\max_{1\le j\le m}\max_{g\in \cT_{D, n}} |c_{p_j, g}|$.

Take a $\delta' > 0$ such that
$3n \cdot 2^nn^n|D|^nM_1\delta' < \delta/(3m)$.
We will show that $h^{\varepsilon}_\Sigma (\cT, FD, \delta') \leq h^{\varepsilon'}_\Sigma (\cS, F, \delta)$.
Since $F$ can be chosen so as to contain an arbitrary finite subset of $G$ and $\delta$ can be taken arbitrarily small,
this implies that $h^{\varepsilon}_\Sigma (\cT ) \leq h^{\varepsilon'}_\Sigma (\cS )$.

Let $\sigma$ be a map from $G$ to $\Sym (d)$ for some $d\in\Nb$ which we assume
to be a good enough sofic approximation to guarantee an estimate in the following paragraph, as will be indicated.
Let $\varphi\in\Hom (\cT ,FD,\delta', \sigma )$.
We will show that $\varphi \in\Hom (\cS ,F,\delta , \sigma )$.

Let $t\in F$. For $(h,s)\in \{q_1, \dots, q_n\}^n \times D^n$ we have,
using the multiplicativity of $\varphi$ and our assumption that $F$ contains $e$,
\begin{align*}
\lefteqn{\bigg\| \varphi\circ\alpha_t \bigg(\prod_{k=1}^n \alpha_{s_k} (h_k )\bigg)
- \sigma_t \circ\varphi \bigg(\prod_{k=1}^n \alpha_{s_k} (h_k )\bigg) \bigg\|_2} \hspace*{25mm} \\
\hspace*{25mm} &= \bigg\| \prod_{k=1}^n \varphi\circ\alpha_t (\alpha_{s_k} (h_k ))
- \prod_{k=1}^n \sigma_t \circ\varphi (\alpha_{s_k} (h_k )) \bigg\|_2 \\
&\leq \sum_{k=1}^n \| \varphi \circ\alpha_t (\alpha_{s_k} (h_k ))
- \sigma_t \circ\varphi (\alpha_{s_k} (h_k )) \|_2 \\
&\leq \sum_{k=1}^n \big( \| \varphi\circ\alpha_{ts_k} (h_{s_k} )
- \sigma_{ts_k} \circ\varphi (h_{s_k} ) \|_2 \\
&\hspace*{30mm} \ + \| (\sigma_{ts_k} - \sigma_t \circ\sigma_{s_k} )(\varphi (h_{s_k} )) \|_2 \\
&\hspace*{30mm} \ + \| \sigma_t (\sigma_{s_k} \circ\varphi (h_{s_k} ) - \varphi\circ\alpha_{s_k} (h_{s_k} )) \|_2 \big) \\
&\leq 3n\cdot 2^n\delta'
\end{align*}
assuming that $\sigma$ is a good enough sofic approximation.
Thus given a $p\in\{p_1, \dots, p_m\}$, since
$q_1=1$ we have
\begin{align*}
\lefteqn{\| \varphi\circ\alpha_t (p' ) - \sigma_t \circ\varphi (p' ) \|_2}\hspace*{15mm} \\
\hspace*{15mm} &\leq \sum_{g\in \cT_{D, n}} | c_{p,g} |\cdot \|\varphi\circ \alpha_t(g)-\sigma_t\circ \varphi(g)\|_2\\
&\le M_1\sum_{g\in \cT_{D, n}} \|\varphi\circ \alpha_t(g)-\sigma_t\circ \varphi(g)\|_2\\
&\le n^n|D|^nM_1 3n \cdot 2^n\delta'
< \frac{\delta}{3m},
\end{align*}
whence
\begin{align*}
\lefteqn{\sum_{j=1}^{\infty}\frac{1}{2^j}\| \varphi\circ\alpha_t (p_j) - \sigma_t \circ\varphi (p_j) \|_2}
\hspace*{25mm} \\
\hspace*{30mm} &\leq \sum_{j=1}^m\frac{1}{2^j}\| \varphi\circ\alpha_t (p_j) - \sigma_t \circ\varphi (p_j) \|_2
+\frac{1}{2^{m-1}}\\
&\leq \sum_{j=1}^m \big(
\| \varphi\circ\alpha_t (p_j-p'_j ) \|_2 + \| \varphi\circ\alpha_t (p'_j ) - \sigma_t \circ\varphi (p'_j ) \|_2 \\
&\hspace*{25mm} \ + \| \sigma_t \circ\varphi (p'_j - p_j) \|_2 \big) +\frac{1}{2^{m-1}}\\
&< 2
\sum_{j=1}^m
\| p_j-p'_j \|_{\infty} + \frac{\delta}{3} +\frac{1}{2^{m-1}}\\
&< \frac{\delta}{3} + \frac{\delta}{3} +\frac{\delta}{3}= \delta .
\end{align*}
Therefore
$\varphi\in\Hom (\cS ,F, \delta , \sigma )$, and so
$\Hom (\cT ,FD,\delta' , \sigma )\subseteq\Hom (\cS ,F, \delta , \sigma )$.

Let $\varphi$ and $\psi$ be elements of $\Hom (\cT ,FD, \delta' ,\sigma )$ such that
$\rho_\cS (\varphi , \psi ) \le \varepsilon'$.
Then for $(h,s)\in \{p_1, \dots, p_\ell\}^\ell \times F^\ell$ we have
\begin{align*}
\lefteqn{\bigg\| \varphi \bigg( \prod_{k=1}^\ell \alpha_{s_k} (h_k )\bigg)
- \psi \bigg( \prod_{k=1}^\ell \alpha_{s_k} (h_k )\bigg) \bigg\|_2} \hspace*{20mm} \\
\hspace*{20mm} &= \bigg\| \prod_{k=1}^\ell \varphi (\alpha_{s_k} (h_{s_k} ))
- \prod_{k=1}^\ell \psi (\alpha_{s_k} (h_{s_k} )) \bigg\|_2 \\
&\leq \sum_{k=1}^\ell \| \varphi (\alpha_{s_k} (h_{s_k} )) - \psi (\alpha_{s_k} (h_{s_k} )) \|_2 \\
&\leq \sum_{k=1}^\ell \big( \| \varphi (\alpha_{s_k} (h_{s_k} )) - \sigma_{s_k} (\varphi (h_{s_k} )) \|_2
+ \| \sigma_{s_k} (\varphi (h_{s_k} ) - \psi (h_{s_k} )) \|_2 \\
&\hspace*{30mm} \ + \| \sigma_{s_k} (\psi (h_{s_k} )) - \psi (\alpha_{s_k} (h_{s_k} )) \|_2 \big) \\
&\leq (2\cdot 2^\ell\delta + 2^\ell\varepsilon' )\ell \leq 2^{\ell+1}\varepsilon'\ell,
\end{align*}
so that, for $q\in\{q_1, \dots, q_R\}$, since $F\supseteq E$ and $p_1=1$,
\begin{align*}
\| \varphi (q' ) - \psi (q' ) \|_2
&\leq \sum_{g\in\cS_{E, \ell}} | d_{q,g} |\cdot \|\varphi(g)-\psi(g)\|_2\\
&\leq M \sum_{g\in\cS_{E, \ell}} \|\varphi(g)-\psi(g)\|_2\\
&\leq M\ell^\ell |E|^\ell \cdot 2^{\ell+1}\varepsilon' \ell \\
&= \frac{\varepsilon}{4R},
\end{align*}
and hence
\begin{align*}
\rho_\cT (\varphi , \psi ) &= \sum_{j=1}^{\infty} \frac{1}{2^j}\| \varphi (q_j) - \psi (q_j) \|_2 \\
&\leq \sum_{j=1}^R \| \varphi (q_j) - \psi (q_j) \|_2 +\frac{1}{2^{R-1}}\\
&\leq \sum_{j=1}^R \big( \| \varphi (q_j - q'_j ) \|_2
+ \| \varphi (q'_j ) - \psi (q'_j ) \|_2 + \| \psi (q'_j - q_j) \|_2 \big) +\frac{1}{2^{R-1}}\\
&\leq 2\sum_{j=1}^R \| q_j - q'_j \|_{\infty} + \frac{\varepsilon}{4} +\frac{1}{2^{R-1}}\\
&< \frac{\varepsilon}{3} + \frac{\varepsilon}{4} +\frac{\varepsilon}{3}< \varepsilon .
\end{align*}
Thus any $\varepsilon$-separated subset of $\Hom (\cT ,FD, \delta' ,\sigma )$
with respect to $\rho_\cT$ is  $\varepsilon'$-separated with respect to $\rho_\cS$,
and so
\begin{align*}
N_\varepsilon (\Hom (\cT ,FD, \delta', \sigma ),\rho_\cT )\leq N_{\varepsilon'} (\Hom (\cS ,F, \delta, \sigma ),\rho_\cS ).
\end{align*}
Consequently
$h^{\varepsilon}_\Sigma (\cT, FD, \delta') \leq h^{\varepsilon'}_\Sigma (\cS, F, \delta)$, as desired.
\end{proof}

Note that the above theorem can also be established, less directly, by combining Theorem~\ref{T-gen comparison} with
the local formula established in the proof of the variational principle in Section~\ref{S-variational}.

Since we are assuming $X$ to be a compact metrizable space,
there always exists a sequence in the unit ball of $C_{\Rb}(X)$ that generates $C(X)$ as a unital $C^*$-algebra.
In view of Theorem~\ref{T-gen comparison top} we can thus define the topological entropy
of our system with respect to $\Sigma$ as follows.

\begin{definition}\label{D-global topological}
The topological entropy $h_\Sigma (X,G)$ of the system $(X,G)$ with respect to $\Sigma$ is defined as
the common value of $h_\Sigma (\cS )$ over all dynamically generating sequences $\cS$ in the unit ball of $C_{\Rb}(X)$.
\end{definition}

%

Since $h_\Sigma (\cS )$ depends only on the image of $\cS$ by Definition~\ref{D-entropy top},
we can define $h_\Sigma (\cP )$
for a countable subset $\cP$ of the unit ball of $C_\Rb (X)$ as the common value of $h_\Sigma (\cS )$
over all sequences $\cS$ whose image is equal to $\cP$.

Suppose now that $\cP$ is a finite partition of unity in $C(X)$. Then, as in the measurable case,
we can proceed more simply as follows.
For a $d\in\Nb$ we define on the set of unital positive linear maps
from some unital self-adjoint linear subspace of $C(X)$ containing $\spn (\cP )$ to $\Cb^d$ the pseudometric
\begin{align*}
\rho_{\cP} (\varphi , \psi ) &= \max_{p\in\cP} \| \varphi (p) - \psi (p) \|_2 .
\end{align*}

\begin{definition}\label{D-finite top}
Let $\sigma$ be a map from $G$ to $\Sym (d)$ for some $d\in\Nb$.
Let $\cP$ be a finite partition of unity in $C(X)$, $F$ a nonempty finite subset of $G$,
and $\delta > 0$. Define $\Hom (\cP , F,\delta ,\sigma )$ to be
the set of all unital homomorphisms $\varphi : C(X) \to \Cb^d$ such that
\[
\| \varphi\circ\alpha_s (p) - \sigma_s \circ\varphi (p) \|_2 < \delta
\]
for all $p\in\cP$ and $s\in F$.
We then define $h_\Sigma^\varepsilon (\cP ,F,\delta )$,
$h_\Sigma^\varepsilon (\cP ,F)$, $h_\Sigma^\varepsilon (\cP )$, and
$h_\Sigma (\cP )$ by formally substituting $\cP$ for $\cS$ in Definition~\ref{D-entropy top}.
\end{definition}

It is easily seen that $h_\Sigma (\cP )$ as defined above is equal to $h_\Sigma (\cS )$ for any sequence $\cS$
whose image as a function on $\Nb$ is equal to $\cP$, and so the notation $h_\Sigma (\cP )$ is unambiguous.

We next observe that, for the purpose of defining $h_\Sigma (\cP )$,
as well as the prior sequential and measure versions of it,
it is possible to substitute the $\infty$-norm for the $2$-norm in the definition of $\rho_\cP$.
This will be used to estimate $h_\Sigma (\cP )$ in the proof of Lemma~\ref{L-algebraic lower bound},
which is the motivation for explaining this substitution here in the present topological context.
So for a given $d\in\Nb$ we define on the set of unital positive linear maps
from some unital self-adjoint linear subspace of $C(X)$ containing $\spn (\cP )$ to $\Cb^d$ the pseudometric
\begin{align*}
\rho_{\cP ,\infty} (\varphi , \psi ) &= \max_{p\in\cP} \| \varphi (p) - \psi (p) \|_\infty .
\end{align*}
and record the following.

\begin{proposition}\label{P-infinity norm}
Let $\cP$ be a finite partition of unity in $C(X)$. Then
\[
h_\Sigma (\cP ) = \sup_{\varepsilon > 0} \inf_{F} \inf_{\delta > 0}
\limsup_{i\to\infty} \frac{1}{d_i} \log N_\varepsilon (\Hom (\cP ,F,\delta ,\sigma_i ), \rho_{\cP ,\infty} )
\]
where $F$ ranges over the nonempty finite subsets of $G$.
\end{proposition}

\begin{proof}
Since $\| \cdot \|_\infty$ dominates $\| \cdot \|_2$ in $\Cb^d$, the right side of the equality dominates the left side.

For the reverse inequality, observe that, given a $\varphi\in\Hom (\cP ,F,\delta ,\sigma )$
for some nonempty finite set $F\subseteq G$, $\delta > 0$, and $\sigma : G\to\Sym (d)$,
every element of $\Hom (\cP ,F,\delta ,\sigma )$ in the $(\rho_{\cP} ,\varepsilon )$-neighbourhood
of $\varphi$ agrees with $\varphi$ on $\cP$ to within $\sqrt{\varepsilon}$ on a subset of $\{ 1,\dots ,d \}$
of cardinality at least $(1 - | \cP |\varepsilon )d$. Since $\psi (p)(a) \in [0,1]$ for all
$\psi\in\Hom (\cP ,F,\delta ,\sigma )$, $p\in\cP$, and $a\in \{ 1,\dots ,d\}$,
it follows that the maximal cardinality of a $(\rho_{\cP ,\infty} ,2\sqrt{\varepsilon})$-separated
subset of the $(\rho_{\cP} ,\varepsilon )$-neighbourhood of $\varphi$ is at most
$\sum_{k=0}^{\lfloor | \cP |\varepsilon d \rfloor} \binom{d}{k} \varepsilon^{-| \cP | k/2}$,
and by Stirling's approximation this number is bounded above by
$a\exp(\beta d) \varepsilon^{-| \cP |^2\varepsilon d/2}$
for some $a,\beta > 0$ not depending on $d$ with $\beta\to 0$ as $\varepsilon\to 0$. Consequently
\[
N_{2\sqrt{\varepsilon}} (\Hom (\cP ,F,\delta ,\sigma ),\rho_{\cP , \infty} )
\leq a\exp(\beta d) \varepsilon^{-| \cP |^2 \varepsilon d/2}
N_\varepsilon (\Hom (\cP ,F,\delta ,\sigma ),\rho_\cP ) .
\]
and hence
\begin{align*}
\limsup_{i\to\infty} \frac{1}{d_i} \log N_{2\sqrt{\varepsilon}} (\Hom (\cP ,F,\delta ,\sigma_i ),\rho_{\cP , \infty} )
&\leq h_\Sigma^\varepsilon (\cP ,F,\delta ) + \beta - | \cP |^2 \varepsilon \log \sqrt{\varepsilon}
\end{align*}
Since $\beta - | \cP |^2 \varepsilon \log \sqrt{\varepsilon} \to 0$ as $\varepsilon\to 0$,
we obtain the desired inequality.
\end{proof}

In the case that $\cP$ is a partition of unity in $C(X)$ consisting of projections, we can also express
$h_\Sigma (\cP )$ by dispensing with the $\varepsilon$
and simply counting unital homomorphisms, as we record below in Proposition~\ref{P-card top} (cf.\ Remark~\ref{R-hom}).
First we state the following topological version of Lemma~\ref{L-infinite proj}, which can be established
by a similar argument.

\begin{lemma}\label{L-proj top}
For every $\kappa > 0$ and $n\in\Nb$ there is an $\varepsilon > 0$ such that every partition of unity
$\cP\subseteq C(X)$ consisting of at most $n$ projections satisfies
\[
\limsup_{i\to\infty} \frac{1}{d_i} N_0 (\Hom (\cP , F,\delta ,\sigma ),\rho_{\cP})
\leq h_\Sigma^\varepsilon (\cP ,F,\delta ) + \kappa
\]
for all nonempty finite sets $F\subseteq G$ and $\delta > 0$.
\end{lemma}

Lemma~\ref{L-proj top} readily yields the desired formula:

\begin{proposition}\label{P-card top}
Let $\cP$ be a finite partition of unity in $C(X)$ consisting of projections. Then
\[
h_\Sigma (\cP ) = \inf_{F} \inf_{\delta > 0}
\limsup_{i\to\infty} \frac{1}{d_i} N_0 (\Hom (\cP ,F,\delta ,\sigma_i ),\rho_{\cP} )
\]
where $F$ ranges over all nonempty finite subsets of $G$.
\end{proposition}

\begin{example}\label{E-Bernoulli}
Consider the Bernoulli action of $G$ on $X=\{ 1,\dots ,k \}^G$ by left translation
for some $k\in\Nb$. Then $h_\Sigma (X,G) = \log k$ for any sofic approxmation sequence $\Sigma$,
which can be seen as follows.
Set $\cP = \{ p_1 , \dots ,p_k \}$ where $p_i$ is the
characteristic function of the set of all $(x_s )_{s\in G}$ such that $x_e = i$.
Then $\cP$ is a dynamically generating partition of unity in $C(X)$ consisting of projections.
Let $\sigma$ be a map from $G$ to $\Sym (d)$ for some $d\in\Nb$. Let $F$ be a nonempty finite subset of $G$
containing $e$ and let $\delta > 0$. Note that there are $k^d$ unital
homomorphisms from $\spn (\cP ) \cong \Cb^k$ to $\Cb^d$.
Let $\varphi$ be such a homomorphism.
For every $\omega\in \{ 1,\dots ,k \}^F$ the projection $\prod_{s\in F} \alpha_s (p_{\omega (s)} )$
is nonzero and so we can set
\[
\tilde{\varphi} \bigg( \prod_{s\in F} \alpha_s (p_{\omega (s)} ) \bigg)
= \prod_{s\in F} \sigma_s (\varphi (p_{\omega (s)} ))
\]
and extend linearly to define a unital homomorphism $\tilde{\varphi}$ from the unital
$C^*$-subalgebra $\spn (\cP_F )$ of $C(X)$ into $\Cb^d$,
where $\cP_F$ denotes the set of all products of the form $\prod_{s\in F} \alpha_s (p_{\omega (s)} )$
for $\omega\in \{ 1,\dots ,k \}^F$. We furthermore extend $\tilde{\varphi}$ arbitrarily to a unital
homomorphism $C(X)\to\Cb^d$, which we again denote by $\tilde{\varphi}$ (this can be done using
the Gelfand theory of commutative $C^*$-algebras mentioned in the introduction).
It is then readily checked that
$\tilde{\varphi} \circ\alpha_s (f) = \sigma_s \circ\tilde{\varphi} (f)$ for all $f\in \spn (\cP )$.
Therefore $N_0 (\Hom (\cP ,F,\delta ,\sigma ),\rho_\cP ) = k^d$, and so we conclude in view of
Proposition~\ref{P-card top} that $h_\Sigma (X,G) = h_\Sigma (\cP ) = \log k$.
\end{example}

A problem of Gottschalk asks which countable groups $G$ are surjunctive, i.e., have the property that
for every finite nonempty set $A$ the action of $G$ on $A^G$ by left translation is surjunctive,
which means that every injective $G$-equivariant continuous map $A^G \to A^G$ is surjective \cite{Got73}.
As observed by Gromov \cite[Subsect.\ 5.M$'''$]{Gro99} (see also Section~1 of \cite{Wei00}), the surjunctivity of
amenable $G$ follows from the fact that the classical topological entropy of a proper subshift is
strictly less than that of the full shift.
Using different means, Gromov showed more generally in \cite{Gro99} that
all countable sofic groups are surjunctive (see also Section~3 of \cite{Wei00}). In fact it is in \cite{Gro99}
that the concept of a sofic group originates, with the terminology being coined by Weiss in \cite{Wei00}.
Now that we have a definition of topological entropy for actions of any countable sofic group, we can
give an entropy proof of Gromov's result like in the amenable case. In view of Example~\ref{E-Bernoulli},
it suffices to observe the following.

\begin{theorem}
Let $G$ be a countable sofic group and let $\Sigma = \{ \sigma_i : G\to\Sym (d_i ) \}_{i=1}^\infty$
be a sofic approximation sequence for $G$.
Let $A$ be a nonempty finite set and let $\alpha$ be the restriction of the left shift action of $G$ on $A^G$ to some closed
$G$-invariant proper subset $X$. Then $h_{\Sigma}(X, G)<\log |A|$.
\end{theorem}

\begin{proof}
For each $a\in A$, denote the characteristic function of $\{x\in X: x_e=a\}$ by $p_a$.
Then $\cP=\{p_a: a\in A\}$ is a dynamically generating finite partition of unity in $C(X)$.
We may assume that $p_a \neq 0$ for each $a\in A$ by discarding all elements of $A$ which do not appear in
the coordinate description of any element of $X$.

Since $X$ is a proper subset of $A^G$, there exists some nonempty finite subset $F$ of $G$ such that $X_F\neq A^F$,
where $X_F$ denotes the set of restrictions of elements of $X$ to $F$.
To establish the theorem it enough to show that
\[
\inf_{\delta>0}\limsup_{i\to \infty}\frac{1}{d_i}\log N_0(\Hom(\cP, F, \delta, \sigma_i), \rho_{\cP})\le
\log |A|+(1/|F|^2)\log \bigg(\frac{|A|^{|F|}-1}{|A|^{|F|}}\bigg) .
\]

Fix an $f\in A^F\setminus X_F$. Then $\prod_{s\in F} \alpha_s(p_{f(s)})=0$.

Let $\delta>0$ be such that $(\delta|F|)^2<1/(4|F|^2)$.
Let $\sigma$ be a map from $G$ to $\Sym(d)$ for some $d\in \Nb$.
Let $W$ be a set of elements in $\Hom(\cP, F, \delta, \sigma)$ which pairwise are nonzero distance apart under $\rho_{\cP}$.
Then the restrictions to $\Cb\cP$ of any two distinct elements of $W$ are different.
Denote by $W'$ the set of restrictions of elements in $W$ to $\Cb\cP$. Then $|W|=|W'|$.
Note that there is a natural bijection between the set of
unital homomorphisms from $\Cb\cP$ to $\Cb^d$ and the set of partitions of $\{1, \dots, d\}$ indexed by $A$,
as we are assuming that $p_a \neq 0$ for all $a\in A$.
For each partition $Q=\{q_a: a\in A\}$ of $\{1, \dots, d\}$ indexed by $A$, the corresponding unital homomorphism
$\Cb\cP\rightarrow \Cb^d$ sends $p_a$ for $a\in A$ to the characteristic function of $q_a$.

Let $\varphi\in W$, and let $\varphi'\in W'$ be the restriction of $\varphi$ to $\Cb\cP$. Then
\begin{align*}
\bigg\|\prod_{s\in F} \sigma_s(\varphi(p_{f(s)}))\bigg\|_2
&\le \bigg\|\prod_{s\in F} \varphi(\alpha_s(p_{f(s)}))\bigg\|_2+\bigg\|\prod_{s\in F} \sigma_s(\varphi(p_{f(s)}))-\prod_{s\in F}
\varphi(\alpha_s(p_{f(s)}))\bigg\|_2\\
&\le \bigg\|\varphi\bigg(\prod_{s\in F} \alpha_s(p_{f(s)})\bigg)\bigg\|_2
+\sum_{s\in F}\bigg\|\sigma_s(\varphi(p_{f(s)}))- \varphi(\alpha_s(p_{f(s)}))\bigg\|_2\\
&< \delta |F|.
\end{align*}
Let $Q=\{q_a: a\in A\}$ be the partition of $\{1, \dots, d\}$ indexed by $A$ which corresponds to $\varphi'$.
Note that $\prod_{s\in F} \sigma_s(\varphi(p_{f(s)}))$ is the characteristic function of $\bigcap_{s\in F}\sigma_s(q_{f(s)})$.
Thus $\big|\bigcap_{s\in F}\sigma_s(q_{f(s)})\big|< (\delta|F|)^2d$.

Denote by $Z$ the set of all $n\in \{1, \dots, d\}$ such that $\sigma_s^{-1}(n)\neq \sigma_t^{-1}(n)$ for all distinct $s,t\in F$.
Let $0<\tau<1/2$. When $\sigma$ is a good enough sofic approximation of $G$, we have $|Z|\ge d(1-\tau)$.

For each $n\in Z$, denote by $V_n$ the set $\{\sigma_s^{-1}(n): s\in F\}$. Then $|V_n|=|F|$.
Take a maximal subset $Z'$ of $Z$ subject to the condition that for any distinct $m, n\in Z'$ the sets $V_n$ and $V_m$ are disjoint.
Then $Z\subseteq \bigcup_{s,t\in F}\sigma_s\sigma_t^{-1}(Z')$, and hence $|Z'|\ge |Z|/|F|^2\ge (1-\tau)d/|F|^2$.

Denote by $S$ the set of all partitions $Q'=\{q'_a: a\in A\}$ of $\{1, \dots, d\}$ indexed by $A$ for which there is some
$Z''\subseteq Z'$ satisfying $|Z''| > (\delta|F|)^2d$ and $\sigma_s^{-1}(n)\in q'_{f(s)}$ for all $n\in Z''$ and $s\in F$.
For any such $Q'$ one has $\bigcap_{s\in F}\sigma_s(q'_{f(s)})\supseteq Z''$, and hence
$\big|\bigcap_{s\in F}\sigma_s(q'_{f(s)})\big| > (\delta|F|)^2d$. Therefore $Q\not\in S$.

Define the function $\xi$ on $[0,1]$ by $\xi(t)=-t\log t$.
The number $|W|$ is bounded above by the number of partitions of $\{1, \dots, d\}$ indexed by $A$ which do not belong to $S$,
which is bounded above by
\begin{align*}
\binom{|Z'|}{|Z'|-\lfloor (\delta|F|)^2d\rfloor}(|A|^{|F|}-1)^{|Z'|-\lfloor (\delta|F|)^2d\rfloor}
|A|^{d-(|Z'|-\lfloor (\delta|F|)^2d\rfloor)|F|},
\end{align*}
which in turn by Stirling's approximation is bounded above by
\begin{align*}
C\exp(|Z'|\xi(1-\delta^2|F|^2d/|Z'|)+|Z'|\xi(\delta^2|F|^2d/|Z'|))|A|^d\bigg(\frac{|A|^{|F|}-1}{|A|^{|F|}}\bigg)^{|Z'|-\delta^2 |F|^2d}
\end{align*}
for some constant $C>0$ not depending on $d$ or $|Z'|$.
Since $|Z'|\ge (1-\tau)d/|F|^2>2\delta^2|F|^2d$ and the function $\xi$ is concave, we have
\[
\xi(1-\delta^2|F|^2d/|Z'|)+\xi(\delta^2|F|^2d/|Z'|)\le \xi(1-\delta^2|F|^4/(1-\tau))+\xi(\delta^2|F|^4/(1-\tau)).
\]
It follows that
\begin{align*}
\lefteqn{\limsup_{i\to \infty}\frac{1}{d_i}\log N_0(\Hom(\cP, F, \delta, \sigma_i), \rho_{\cP})}\hspace*{30mm} \\
\hspace*{30mm} &\le \xi(1-\delta^2|F|^4/(1-\tau))+\xi(\delta^2|F|^4/(1-\tau))\\
&\hspace*{10mm} \ +\log |A|+((1-\tau)/|F|^2-\delta^2 |F|^2)\log \bigg(\frac{|A|^{|F|}-1}{|A|^{|F|}}\bigg).
\end{align*}
Letting $\tau\rightarrow 0$, we get
\begin{align*}
\lefteqn{\limsup_{i\to \infty}\frac{1}{d_i}\log N_0(\Hom(\cP, F, \delta, \sigma_i), \rho_{\cP})}\hspace*{15mm} \\
\hspace*{15mm} &\le \xi(1-\delta^2|F|^4)+\xi(\delta^2|F|^4)+\log |A|
+(1/|F|^2-\delta^2 |F|^2)\log \bigg(\frac{|A|^{|F|}-1}{|A|^{|F|}}\bigg).
\end{align*}
Then
\begin{align*}
\inf_{\delta>0}\limsup_{i\to \infty}\frac{1}{d_i}\log N_0(\Hom(\cP, F, \delta, \sigma_i), \rho_{\cP})
&\le \log |A|+(1/|F|^2)\log \bigg(\frac{|A|^{|F|}-1}{|A|^{|F|}}\bigg),
\end{align*}
as desired.
\end{proof}

We point out that for certain $G$ it can happen that for some subshift action as in the
above theorem we have $h_\Sigma (X,G) = -\infty$ for every sofic approximation sequence $\Sigma$. For this
to occur it suffices that $X$ admit no $G$-invariant Borel probability measure, as a weak$^*$ limit argument
demonstrates (see also Theorem~\ref{T-variational}), and there are topological Markov chains over
the free group $F_2$ that do not admit an invariant Borel probability measure. Consider for example the left shift
action of $F_2$ on $\{ 0,1,2 \}^{F_2}$, and then take the closed $G$-invariant subset
$X$ consisting of elements whose allowable
transitions in the directions of the two generators are described by
$0\rightleftarrows 1\rightleftarrows 2$ and $0\rightarrow 1\rightarrow 2\rightarrow 0$.
If $X$ had an invariant Borel probability measure then by the first arrow diagram the measure of the set $A_1$ of all $x\in X$
for which $x_e = 1$ would be the sum of the measure of the set $A_0$ of all $x\in X$ for which $x_e = 0$ and the measure of
the set $A_2$ of all $x\in X$ for which $x_e = 2$, but each of the sets $A_0$, $A_1$, and $A_2$ must have measure $1/3$
by the second arrow diagram, producing a contradiction.

\section{Measure entropy via homomorphisms}\label{S-Hom}

Let $\alpha$ be a continuous action of a sofic countable group $G$ on a compact metrizable space $X$.
When considering $G$-invariant Borel probability measures on $X$, as will be the case in
Sections~\ref{S-variational} and \ref{S-algebraic}, we wish to
have a way of expressing measure entropy in terms of unital homomorphisms from $C(X)$ into $\Cb^d$
for the purpose of comparison with topological entropy. This is especially convenient when the invariant measure $\mu$
in question does not have full support, in which case $C(X)$ does not naturally embed into $L^\infty (X,\mu )$.
We therefore make the following definitions
in analogy with Definitions~\ref{D-UP} and \ref{D-entropy}, and then show in Proposition~\ref{P-Hom measure}
that we recover the measure entropy as originally defined in Section~\ref{S-measure}.

Let $\cS = \{ p_n \}_{n=1}^\infty$ be a sequence in the unit ball of $C_\Rb (X)$.
The notation $\cS_{F,m}$ and $\rho_\cS$ is as introduced in Section~\ref{S-measure}.

\begin{definition}\label{D-measure top}
Suppose that $\mu$ is a Borel probability measure on $X$. Let $\sigma$ be a map from $G$ to $\Sym (d)$
for some $d\in\Nb$. Let $F$ be a nonempty finite subset of $G$, $m\in\Nb$, and $\delta > 0$.
We write $\Hom_\mu^X (\cS ,F,m,\delta ,\sigma )$ for the set of unital homomorphisms
$\varphi : C(X)\to\Cb^d$ such that
\begin{enumerate}
\item[(i)]
$| \zeta\circ\varphi (f) - \mu (f) | < \delta$ for all $f\in \cS_{F, m}$,

\item[(ii)]
$\| \varphi\circ\alpha_s (f) - \sigma_s \circ\varphi (f) \|_2 < \delta$ for all $s\in F$ and $f\in \{p_1, \dots, p_m\}$.
\end{enumerate}
\end{definition}

\begin{definition}\label{D-measure top 2}
Suppose that $\mu$ is a Borel probability measure on $X$. Let $\varepsilon > 0$. Let $F$ be a nonempty finite
subset of $G$, $m\in\Nb$, and $\delta > 0$. We set
\begin{align*}
\ch_{\Sigma ,\mu}^\varepsilon (\cS ,F,m,\delta ) &=
\limsup_{i\to\infty} \frac{1}{d_i} \log N_\varepsilon (\Hom_\mu^X (\cS ,F,m,\delta ,\sigma_i ),\rho_{\cS} ) ,\\
\ch_{\Sigma ,\mu}^\varepsilon (\cS ,F,m) &= \inf_{\delta > 0} \ch_{\Sigma ,\mu}^\varepsilon (\cS ,F,m,\delta ) ,\\
\ch_{\Sigma ,\mu}^\varepsilon (\cS ,F) &= \inf_{m\in\Nb} \ch_{\Sigma ,\mu}^\varepsilon (\cS ,F,m) ,\\
\ch_{\Sigma ,\mu}^\varepsilon (\cS ) &= \inf_{F} \ch_{\Sigma ,\mu}^\varepsilon (\cS ,F) ,
\end{align*}
where the infimum in the last line is over all nonempty finite subsets of $G$.
\end{definition}


The proof of the following lemma is similar to that of Lemma~\ref{L-infinite proj2}.

\begin{lemma}\label{L-Hom ineq 2}
Suppose that $\mu$ is a $G$-invariant Borel probability measure on $X$.
Let $\cS = \{ p_n \}_{n=1}^\infty$ be a sequence in the unit ball of $C_{\Rb}(X)$.
Let $\varepsilon > 0$. Let $F$ be a finite
subset of $G$ containing $e$, $m$ a positive integer with $2^{-(m-1)}<\varepsilon/3$,
and $\delta > 0$. Then there is a $\delta' > 0$ such that
\[
N_\varepsilon (\UP_\mu (\cS ,F,m^2,\delta' ,\sigma ),\rho_\cS )
\leq N_{\varepsilon /3} (\Hom_\mu^X (\cS ,F,m,\delta ,\sigma ),\rho_\cS )
\]
for every $\sigma$ that maps $G$ to $\Sym (d)$ for some $d\in\Nb$.
\end{lemma}

\begin{proof}
Write $B$ for the closed subset of $X$ supporting $\mu$.
Then we can view $C(B)$ as a unital $C^*$-subalgebra of $L^\infty (X,\mu )$, i.e.,
a $^*$-subalgebra which is closed in the $L^\infty$ norm.
Given an $\eta > 0$, by Lemma~\ref{L-mult perturbation} there is a $\delta' > 0$ such that
for every map $\sigma$ from $G$ to $\Sym (d)$ for some $d\in\Nb$
and
every $\varphi\in\UP_\mu (\cS ,F,m^2,\delta' ,\sigma )$
there is a unital homomorphism $\check{\varphi} : C(B) \to\Cb^d$ for which
$\max_{f\in\cS_{F,m}} \| \check{\varphi} (f|_B ) - \varphi (f) \|_2 < \min (\eta , \varepsilon /(6m))$.
By taking $\eta$ and $\delta'$ small enough
this will imply that
$\check{\varphi} \circ\lambda\in \Hom_\mu^X (\cS ,F,m,\delta ,\sigma )$ where $\lambda$ is the
restriction map $f\mapsto f|_B$ from $C(X)$ to $C(B)$.
Define a map $\Gamma : \UP_\mu (\cS ,F,m,\delta' ,\sigma )\to\Hom_\mu^X (\cS ,F,m,\delta ,\sigma )$
by $\Gamma (\varphi ) = \check{\varphi} \circ\lambda$.

For any $\varphi , \psi\in \UP_\mu (\cS ,F,m^2, \delta',\sigma )$, we have
\begin{align*}
\rho_{\cS}(\varphi, \psi) &=
\sum_{n=1}^{\infty} \frac{1}{2^n}\| \varphi (p_n) - \psi (p_n) \|_2 \\
&\leq \sum_{n=1}^m\frac{1}{2^n}\|\varphi(p_n)-\psi(p_n)\|_2+\frac{1}{2^{m-1}}\\
&\leq \sum_{n=1}^m\frac{1}{2^n} \big( \| \varphi (p_n) - \Gamma (\varphi )(p_n) \|_2
+ \| \Gamma (\varphi )(p_n) - \Gamma (\psi )(p_n) \|_2 \\
&\hspace*{30mm} \ + \| \Gamma (\psi )(p_n) - \psi (p_n) \|_2 \big) +\frac{1}{2^{m-1}}\\
&\leq \frac23 \varepsilon + \rho_{\cS}(\Gamma(\varphi), \Gamma(\psi)).
\end{align*}
Thus for every subset $L$ of $\UP_\mu (\cS ,F,m^2, \delta',\sigma )$ which is $\varepsilon$-separated
with respect to $\rho_{\cS}$, the set $\Gamma(L)$ is $(\varepsilon/3)$-separated with respect to $\rho_{\cS}$.
Consequently
\begin{align*}
N_\varepsilon (\UP_\mu (\cS ,F,m^2,\delta',\sigma ),\rho_\cS )
\leq N_{\varepsilon /3} (\Hom^X_\mu (\cS ,F,m,\delta,\sigma ),\rho_\cS ) ,
\end{align*}
yielding the lemma.
\end{proof}

\begin{proposition}\label{P-Hom measure}
Suppose that $\mu$ is a $G$-invariant Borel probability measure on $X$.
Let $\cS = \{ p_n \}_{n=1}^\infty$ be a dynamically generating sequence in the unit ball of $C_{\Rb}(X)$.
Then
\[
h_{\Sigma ,\mu} (\cS ) = \sup_{\varepsilon > 0} \ch_{\Sigma ,\mu}^\varepsilon (\cS ) .
\]
\end{proposition}

\begin{proof}
By Remark~\ref{R-pou} we may assume that $p_1=1$.
That the left side of the displayed equality is bounded above by the right side follows easily from
Lemma~\ref{L-Hom ineq 2}.

For the reverse inequality, it suffices to show that
$h_{\Sigma, \mu}^{\varepsilon/2}(\cS)\ge \ch_{\Sigma ,\mu}^\varepsilon (\cS ) $ for every $\varepsilon>0$.
Fix a compatible metric $\rho$ on $X$.
Denote by $B$ the closed subset of $X$ supporting $\mu$.
Regard $C(B)$ as a unital  $C^*$-subalgebra of $L^\infty (X,\mu )$ as in the proof of Proposition~\ref{L-Hom ineq 2}.
For each unital homomorphism $\varphi_1:C(B)\rightarrow \Cb^d$, fix an extension of $\varphi_1$
to a unital positive linear map $L^{\infty}(X, \mu)\to \Cb^d$, which we denote by $\theta(\varphi_1)$.
Such extensions exist by the Hahn-Banach theorem, as discussed in the introduction.

Let $F$ be a finite subset of $G$ containing $e$, $m$ a positive integer with $2^{-(m-1)}<\varepsilon/8$, and $\delta>0$.

For $\tau>0$ denote by $W_\tau$ the set of all $g\in C(X)$
satisfying $g>0$ on $X$, $g<\tau$ on $X\setminus B_{\tau}$, and $g< 1+\tau$ on $B_{\tau}$,
and $g>1-\tau$ on $B$,
where $B_\tau$ is the open $\tau$-neighbourhood
$\{ x\in X : \inf_{y\in B} \rho(x,y) < \tau \}$ of $B$. Note that the regularity of $\mu$ implies
that, given an $\eta>0$, if $\tau$ is small enough then for every $g\in W_{\tau}$
and every Borel probability measure $\nu$ on $X$ satisfying
$|\nu(g)-\mu(g)|<\tau$ one has $\nu(B_{\tau})>1-\eta$.

Let $\tau$ be a positive number to be determined in a moment.
Since $W_{\tau}$ is a nonempty open subset of $C(X)$, $\cS$ dynamically generates $C(X)$, and $p_1=1$,
we can find a finite set $F'\subseteq G$ containing $F$ and an $m'\in \Nb$ no less than $m$ such that
there exists a function $g$ in the intersection $\spn (\cS_{F',m'} ) \cap W_{\tau}$.
Let $\delta'$ be a positive number to be determined in a moment.
Let $\sigma$ be a map from $G$ to $\Sym (d)$ for some $d\in\Nb$.
Given a $\varphi\in\Hom_\mu^X (\cS , F' ,m' ,\delta' , \sigma)$
we construct a unital homomorphism
$\check{\varphi} : C(B)\to\Cb^{d}$ as follows. For each $a\in \{ 1,\dots ,d \}$ the homomorphism
$f\mapsto \varphi(f)(a)$ on $C(X)$ is given by evaluation at some point $x_a \in X$, and we require
that the homomorphism $f\mapsto\check{\varphi} (f)(a)$ on $C(B)$ is given by some point $y\in B$ which
minimizes the distance from $x_a$ to points of $B$ with respect to $\rho$.
Write $\lambda$ for the restriction map
$f\mapsto f|_B$ from $C(X)$ to $C(B)$.
In view of the uniform continuity of the functions in $\cS_{F, m}$
and the fact that
$| \zeta \circ\varphi (g) - \mu (g) |<\tau$ when $\delta'$ is small enough, one can readily verify
that if $\delta'$ and $\tau$ are assumed to be small enough independently of $d$, $\sigma$ and $\varphi$
then we can ensure that $\rho_\cS (\check{\varphi} \circ\lambda , \varphi ) < \varepsilon /4$ and
$\theta(\check{\varphi}) \in \UP_\mu (\cS , F ,m ,\delta , \sigma )$.

Write $\Gamma$ for the map $\varphi\mapsto \theta(\check{\varphi})$ from
$\Hom_\mu^X (\cS , F' ,m' ,\delta' , \sigma)$ to $\UP_\mu (\cS , F ,m ,\delta , \sigma)$.
For any $\varphi, \psi\in \Hom_\mu^X (\cS , F' ,m' ,\delta' , \sigma)$ one has
\begin{align*}
\rho_\cS (\varphi, \psi)&\le \rho_\cS(\varphi, \check{\varphi} \circ\lambda )+\rho_\cS (\check{\varphi} \circ\lambda , \check{\psi} \circ\lambda )
+\rho_\cS(\check{\psi} \circ\lambda, \psi)\\
&< \varepsilon/2+ \rho_\cS (\check{\varphi} \circ\lambda , \check{\psi} \circ\lambda )=\varepsilon/2+\rho_\cS(\Gamma(\varphi), \Gamma(\psi)).
\end{align*}
Thus for any subset $L$ of $\Hom_\mu^X (\cS , F' ,m' ,\delta' , \sigma)$ which is $\varepsilon$-separated
with respect to $\rho_{\cS}$, the set
$\Gamma(L)$ is $(\varepsilon/2)$-separated with respect to $\rho_{\cS}$.
It follows that
\[ N_{\varepsilon/2}(\UP_\mu (\cS , F ,m ,\delta , \sigma ), \rho_\cS)\geq N_{\varepsilon}(\Hom_\mu^X (\cS , F' ,m' ,\delta' , \sigma), \rho_\cS),\]
and hence
\[
h_{\Sigma ,\mu}^{\varepsilon /2} (\cS ,F ,m ,\delta )  \geq
\ch_{\Sigma ,\mu}^\varepsilon (\cS ,F',m',\delta' )  .
\]
Since $F$ was an arbitrary finite subset of $G$ containing $e$,
$m$ an arbitrary large positive integer, and $\delta$ an arbitrary  positive number,
we conclude that $h_{\Sigma, \mu}^{\varepsilon/2}(\cS)\ge \ch_{\Sigma ,\mu}^\varepsilon (\cS )$.
\end{proof}

In the case of a finite subset $\cP$ of the unit ball of $C_\Rb (X)$, we can avoid the sequential formalism
(cf.\ Definitions~\ref{D-finite} and \ref{D-finite top})
by considering on the set of unital positive linear maps
from some unital self-adjoint linear subspace of $C(X)$ containing $\spn (\cP )$ to $\Cb^d$
the pseudometric
\begin{align*}
\rho_{\cP} (\varphi , \psi ) &= \max_{p\in\cP} \| \varphi (p) - \psi (p) \|_2 .
\end{align*}
and making the following definitions.

\begin{definition}\label{D-Hom finite}
Let $\sigma$ be a map from $G$ to $\Sym (d)$ for some $d\in\Nb$.
Let $\cP$ be a finite partition of unity in $C(X)$. Let $F$ be a nonempty finite subset of $G$, $m\in\Nb$,
and $\delta > 0$.
Define $\Hom_\mu^X (\cP , F,m,\delta ,\sigma )$ to be the set of all unital homomorphisms
$\varphi : C(X) \to \Cb^d$ such that
\begin{enumerate}
\item[(i)]
$| \zeta\circ\varphi (f) - \mu (f) | < \delta$ for all $f\in\cP_{F,m}$,

\item[(ii)]
$\| \varphi\circ\alpha_s (f) - \sigma_s \circ\varphi (f) \|_2 < \delta$ for all $f\in\cP$ and $s\in F$,
\end{enumerate}
where $\cP_{F,m}$ as before denotes the set of all all products of the form
$\alpha_{s_1} (p_1 )\cdots \alpha_{s_j} (p_j )$ where $1\leq j\leq m$,
$p_1 , \dots p_j \in\cP$, and $s_1 , \dots ,s_j \in F$.
Then define
$\ch_{\Sigma ,\mu}^\varepsilon (\cS ,F,m,\delta )$, $\ch_{\Sigma ,\mu}^\varepsilon (\cS ,F,m)$,
$\ch_{\Sigma ,\mu}^\varepsilon (\cS ,F)$, and $\ch_{\Sigma ,\mu}^\varepsilon (\cP )$ by formally
substituting $\cS$ for $\cP$ in Definition~\ref{D-measure top 2}.
\end{definition}

One can easily check that for any sequence $\cS$ whose image is equal to $\cP$ we have
\[
\sup_{\varepsilon > 0} \ch_{\Sigma ,\mu}^\varepsilon (\cP )
= \sup_{\varepsilon > 0} \ch_{\Sigma ,\mu}^\varepsilon (\cS ) , \]
and it follows from Proposition~\ref{P-Hom measure} that this common value is equal to
$h_{\Sigma ,\mu} (\cP )$ as in Definition~\ref{D-finite}.
We will use these facts in Section~\ref{S-algebraic}.

\section{The variational principle}\label{S-variational}

Throughout this section $\alpha$ is a continuous action of a sofic countable group $G$ on a compact
metrizable space $X$.
We write $M(X)$ for the convex set of  Borel probability measures
on $X$ equipped with the weak$^*$ topology, under which it is compact. Write $M_G (X)$ for the set of
$G$-invariant  Borel probability measures on $X$, which is a closed convex subset of $M(X)$.
In the proof below we will use the formulation of measure entropy for measures in $M_G (X)$ as given in
Section~\ref{S-Hom}. See Sections~\ref{S-measure} and \ref{S-top} for other notation.

\begin{theorem}\label{T-variational}
Let $\alpha$ be a continuous action of a sofic countable group $G$ on a compact metrizable space $X$. Then
\[
h_\Sigma (X,G) = \sup_{\mu\in M_G (X)} h_{\Sigma ,\mu} (X,G) .
\]
In particular, if $h_\Sigma (X,G) \neq -\infty$ then $M_G (X)$ is nonempty.
\end{theorem}

\begin{proof}
Fix a dynamically generating sequence $\cS = \{ p_n \}_{n=1}^\infty$ in the unit ball of $C_{\Rb}(X)$ with $p_1=1$.
Let $\varepsilon > 0$.
We will prove that
$h_\Sigma^\varepsilon (\cS ) = \max_{\mu\in M_G (X)} \ch_{\Sigma ,\mu}^\varepsilon (\cS )$,
from which the theorem will follow in view of Proposition~\ref{P-Hom measure}.

Let $\mu\in M_G (X)$. Denote by $B$ the closed subset of $X$ supporting $\mu$, which is $G$-invariant.
For every nonempty finite set $F\subseteq G$, $m\in\Nb$, $\delta > 0$, and any map $\sigma$ from
$G$ to $\Sym(d)$ for some $d\in \Nb$, we have
\[ \Hom_\mu^X (\cS ,F,m,\delta ,\sigma ) \subseteq \Hom (\cS ,F,\delta +2^{-(m-1)},\sigma ), \]
and so, for every $\varepsilon > 0$,
\begin{align*}
N_{\varepsilon} (\Hom_\mu^X (\cS ,F,m,\delta ,\sigma ),\rho_\cS )
\leq N_\varepsilon (\Hom (\cS ,F,\delta+2^{-(m-1)},\sigma ),\rho_\cS ) .
\end{align*}
Consequently $h_\Sigma^\varepsilon (\cS ) \geq \sup_{\mu\in M_G (X)} \ch_{\Sigma ,\mu}^\varepsilon (\cS )$.

Now let us prove the the reverse inequality. We may assume that $h_\Sigma (X,G) \neq -\infty$.
Let $\varepsilon>0$.
Take a sequence $e\in F_1 \subseteq F_2 \subseteq\dots$ of finite subsets of $G$ whose union is
equal to $G$. Let $n\in\Nb$. We aim to produce a $\mu_n \in M(X)$ such that
$\ch_{\Sigma ,\mu_n}^{\varepsilon } (\cS ,F_n ,n, 1/n) \geq h_\Sigma^{\varepsilon} (\cS )$
and $| \mu_n (\alpha_t (f)) - \mu_n (f) | < 1/n$ for all $t\in F_n$ and $f\in\cS_{F_n ,n}$.
By weak$^*$ compactness we can find a finite set $D\subseteq M(X)$ such that
for every map $\sigma : G\to\Sym(d)$ for some $d\in \Nb$ and
every $\varphi\in\Hom (\cS , F_n , 1/n , \sigma )$
there is a $\mu_\varphi \in D$ such that
$|\mu_\varphi (\alpha_t (f)) - \zeta\circ\varphi (\alpha_t (f))| < (3n)^{-1}$
for all $t\in F_n$ and $f\in\cS_{F_n ,n}$, where as usual $\zeta$ is the uniform
probability measure on $\{ 1,\dots ,d \}$ viewed as a state on $\Cb^d$.
Let $\sigma$ be a map from $G$ to $\Sym (d)$ for some $d\in\Nb$.
Note that for all $\varphi\in\Hom (\cS , F_n^2 , (3n)^{-2}2^{-n} , \sigma )$, $s_1 , \dots , s_n \in F_n$,
$f_1 , \dots , f_n \in\{p_1, \dots, p_n\}$, and $t\in F_n$ we have, setting $f = \alpha_{s_1} (f_1 )\cdots \alpha_{s_n} (f_n )\in \cS_{F_n, n}$
and assuming that $\sigma$ is a good enough sofic approximation,
\begin{align*}
\lefteqn{| \zeta (\varphi\circ\alpha_t (f) - \sigma_t \circ\varphi (f)) |}\hspace*{10mm} \\
\hspace*{5mm} &\leq \| \varphi\circ\alpha_t (f) - \sigma_t \circ\varphi (f) \|_2 \\
&\leq \sum_{i=1}^n \| \sigma_t \circ\varphi (\alpha_{s_1} (f_1 ) \cdots \alpha_{s_{i-1}} (f_{i-1} )) \\
&\hspace*{8mm} \ \times
(\varphi\circ\alpha_t (\alpha_{s_i} (f_i )) - \sigma_t \circ\varphi (\alpha_{s_i} (f_i )))
\varphi\circ\alpha_t (\alpha_{s_{i+1}} (f_{i+1} ) \cdots \alpha_{s_n} (f_n )) \|_2 \\
&\leq \sum_{i=1}^n \| \varphi\circ\alpha_t (\alpha_{s_i} (f_i )) - \sigma_t \circ\varphi (\alpha_{s_i} (f_i )) \|_2 \\
&\leq \sum_{i=1}^n \big( \| \varphi\circ\alpha_{ts_i} (f) - \sigma_{ts_i} \circ\varphi (f) \|_2
+ \| (\sigma_{ts_i} - \sigma_t \circ\sigma_{s_i} )(\varphi (f)) \|_2 \\
&\hspace*{40mm} \ + \| \sigma_t (\sigma_{s_i} \circ\varphi (f) - \varphi\circ\alpha_{s_i} (f)) \|_2 \big) \\
&< n \bigg( \frac{1}{9n^2} + \frac{1}{9n^2} + \frac{1}{9n^2} \bigg) = \frac{1}{3n}
\end{align*}
so that
\begin{align*}
| \mu_\varphi ( \alpha_t (f)) - \mu_\varphi (f) |
&\leq | \mu_\varphi (\alpha_t (f)) - \zeta\circ\varphi (\alpha_t (f)) | +
| \zeta (\varphi\circ\alpha_t (f) - \sigma_t \circ\varphi (f)) | \\
&\hspace*{30mm} \ + | \zeta\circ\varphi (f) - \mu_\varphi (f) | \\
&< \frac{1}{3n} + \frac{1}{3n} + \frac{1}{3n} = \frac1n .
\end{align*}
Take a maximal $\varepsilon$-separated subset $L$ of $\Hom (\cS , F_n^2 , (3n)^{-2}2^{-n} , \sigma )$.
By the pigeonhole principle there exists a $\nu\in D$ such that the set
\[ W(\sigma ,\nu ) = \{ \varphi\in L : \mu_\varphi = \nu \} \]
satisfies $| W(\sigma ,\nu ) | \geq |L|/|D|$.
Note that $W(\sigma ,\nu ) \subseteq \Hom_\nu^X (\cS , F_n ,n, 1/n , \sigma )$ as $F_n \subseteq F_n^2$ and $p_1=1$.
Since $W(\sigma ,\nu )$ is $\varepsilon$-separated,
we obtain
\begin{align*}
N_{\varepsilon} (\Hom_\nu^X (\cS ,F_n ,n,1/n , \sigma ), \rho_{\cS} )
&\geq |W(\sigma ,\nu )| \\
&\geq \frac{|L|}{|D|}
= \frac{1}{|D|} N_\varepsilon (\Hom (\cS ,F_n^2 ,(3n)^{-2}2^{-n} , \sigma ), \rho_{\cS} ) .
\end{align*}

Letting $\sigma$ now run through the terms of the
sofic approximation sequence $\Sigma$, we infer by the pigeonhole principle that there exist
a $\mu_n \in D$ and a sequence $i_1 < i_2 < \dots$ in $\Nb$ with
\[
h_\Sigma^\varepsilon (\cS , F_n^2 ,(3n)^{-2}2^{-n} ) = \lim_{k\to\infty} \frac{1}{d_{i_k}}
\log N_\varepsilon (\Hom (\cS ,F_n^2 ,(3n)^{-2}2^{-n} , \sigma_{i_k} ), \rho_\cS )
\]
such that
$|W(\sigma_{i_k} ,\mu_n )| \ge |D|^{-1} N_\varepsilon (\Hom (\cS ,F_n^2 ,(3n)^{-2}2^{-n} , \sigma_{i_k} ), \rho_{\cS} )$
for all $k\in\Nb$. Then
\begin{align*}
\ch_{\Sigma ,\mu_n}^{\varepsilon } (\cS ,F_n ,n, 1/n)
&\geq \lim_{k\to\infty} \frac{1}{d_{i_k}} \log \frac{1}{|D|} N_\varepsilon (\Hom (\cS ,F_n^2 ,(3n)^{-2}2^{-n} , \sigma_{i_k} ),\rho_\cS ) \\
&= h_\Sigma^\varepsilon (\cS , F_n^2 ,(3n)^{-2}2^{-n} ) \\
&\geq h_\Sigma^\varepsilon (\cS )
\end{align*}
and $| \mu_n (\alpha_t (f)) - \mu_n (f) | < 1/n$ for all $t\in F_n$ and $f\in\cS_{F_n ,n}$.
So $\mu_n$ satisfies the required properties.

Having constructed a $\mu_n$ for each $n\in\Nb$, take a weak$^*$ limit point $\mu$ of
the sequence $\{ \mu_n \}_{n=1}^\infty$.
Given a $t\in G$ and an $f\in C(X)$ of the form
$\alpha_{s_1} (f_1 ) \cdots \alpha_{s_k} (f_k )$ where $s_1 , \dots , s_k \in G$ and $f_1 , \dots , f_k \in\cS$,
we have
\begin{align*}
| \mu (\alpha_t (f)) - \mu (f) | &\leq | \mu (\alpha_t (f)) - \mu_n (\alpha_t (f)) |
+ | \mu_n (\alpha_t (f)) - \mu_n (f) | + | \mu_n (f) - \mu (f) |
\end{align*}
and the infimum of the right-hand side over all $n\in\Nb$ is zero.
Since $\cS$ is generating and $p_1=1$, every element of $C(X)$ can be approximated arbitrarily well by
linear combinations of functions of the above form, and so we deduce that $\mu$ is $G$-invariant.

Let $F$ be a nonempty finite subset of $G$, $m\in \Nb$, and $\delta>0$. Take an integer $n$
such that $F\subseteq F_n$, $m\le n$, $\delta\ge 2/n$, and
$\max_{f\in \cS_{F, m}}|\mu_n(f)-\mu(f)|<\delta/2$.
Then, for every map $\sigma$ from $G$ to $\Sym(d)$ for some $d\in \Nb$,  every $\varphi$ in
$\Hom_{\mu_n}^X (\cS , F_n, n, 1/n, \sigma)$, and every $f\in \cS_{F, m}$ we have
\begin{align*}
|\zeta \circ \varphi(f)-\mu(f)|&\le |\zeta \circ \varphi(f)-\mu_n(f)|+|\mu_n(f)-\mu(f)|\\
&<\frac{1}{n}+\frac{\delta}{2}\le \delta,
\end{align*}
and hence $\varphi\in \Hom_{\mu}^X (\cS , F, m, \delta, \sigma)$.
Thus
\[ \Hom_{\mu_n}^X (\cS , F_n, n, 1/n, \sigma)\subseteq \Hom_{\mu}^X (\cS , F, m, \delta, \sigma)\]
and so $\ch_{\Sigma ,\mu}^{\varepsilon} (\cS , F ,m, \delta)
\geq \ch_{\Sigma ,\mu_n}^{\varepsilon} (\cS , F_n ,n,1/n)\ge h_\Sigma^{\varepsilon} (\cS )$.
Since $F$ was an arbitrary nonempty finite subset of $G$, $m$ an arbitrary positive integer, and $\delta$
an arbitrary positive number,
we obtain $\ch_{\Sigma ,\mu}^{\varepsilon} (\cS)\ge h_\Sigma^{\varepsilon} (\cS )$.
We conclude that
$h_\Sigma^\varepsilon (\cS ) \leq \sup_{\mu\in M_G (X)} \ch_{\Sigma ,\mu}^\varepsilon (\cS )$,
as desired.
\end{proof}


\section{Algebraic actions of residually finite groups}\label{S-algebraic}

Let $G$ be a countable infinite residually finite discrete group, and $\{G_n\}_{n\in \Nb}$ be a sequence of
finite index normal subgroups with $\lim_{n\to \infty}G_n=\{e\}$ in the sense that, for any $s\in G\setminus \{e\}$,
$s\not\in G_n$ when $n$ is sufficiently large. Let $\Sigma = \{ \sigma_i : G\to \Sym (G/G_i )\}$ be the
corresponding sofic approximation sequence, i.e., $\sigma_i$ is the action of left translation via the quotient map
$G\to G/G_i$.
We denote by $C^*(G)$ the universal group $C^*$-algebra of $G$, and by $\cL G$ the left group von Neumann algebra of $G$
(see Section~2.5 of \cite{BroOz08}).
The Fuglede-Kadison determinant of an invertible element $a\in\cL G$ is given by
$\ddet_{\cL G} a = \exp \tr_{\cL G} \log |a|$ where $|a| = (a^* a)^{1/2}$ and $\tr_{\cL G}$ is the canonical tracial
state on $\cL G$ (see Section~2.2 of \cite{Li10} for more details and references).

For an element $f$ in the integral group ring $\Zb G$, the $\Zb G$-module structure of
$\Zb G/\Zb G f$ corresponds to an action of $G$ on $\Zb G/\Zb G f$.
This induces an action $\alpha_f$ of $G$ on the Pontryagin dual $X_f:=\widehat{ \Zb G/\Zb G f}$ via continuous
automorphisms. We may write
\[
X_f=\{h\in (\Rb/\Zb)^G: fh=0\},
\]
and under this identification $\alpha_f$ is the restriction of the right shift action of $G$ on $(\Rb/\Zb)^G$ to $X_f$ (see Section~3 of \cite{Li10}).

In the case that $f\in\Zb G$ is invertible in $\ell^1 (G)$, Bowen showed in \cite{Bowen10b} that
the sofic measure entropy with respect to $\Sigma$ and the normalized Haar measure on $X_f$
is equal to $\log\ddet_{\cL G}f$. The goal of this section is to establish the topological counterpart
of Bowen's result, stated below as Theorem~\ref{T-algebraic}.
In addition we only assume the invertibility of $f$ in $C^*(G)$. In general this is strictly weaker
than the invertibility of $f$ in $\ell^1 (G)$, for instance when $G$ contains a copy of the free group
on two generators, as discussed in the appendix of \cite{Li10}. Note that when $G$ is amenable
the full and reduced group $C^*$-algebras coincide, in which case $\cL G$ is the weak operator closure
of $C^* (G)$, so that the invertibility of $f$ in $C^*(G)$
is the same as the invertibility of $f$ in $\cL G$ (cf.\ \cite{Li10}).
The invertibility of $f$ in $\ell^1 (G)$ implies the existence of a finite generating measurable partition,
a fact which is used in \cite{Bowen10b}. It is not clear though whether this is the case
if $f$ is merely assumed to be invertible in $C^*(G)$, and so it is essential that we use our more general
definition of measure entropy here.

\begin{theorem} \label{T-algebraic}
Let $f\in \Zb G$ be invertible in $C^*(G)$. Then
\[ h_{\Sigma}(X_f, G)=\log \ddet_{\cL G}f.\]
\end{theorem}

Denote by $\pi$ the homomorphism $C^*(G)\rightarrow \cL G$.
For each $n\in \Nb$  denote by
$\pi_n$ the homomorphism $C^*(G)\rightarrow \cL(G/G_n)$.
The following lemma was proved by Deninger and Schmidt \cite[Lemma 6.2]{DenSch07} in the case $f\in \ell^1(G)$.

\begin{lemma} \label{L-approximate trace}
For every $f\in C^*(G)$ one has
\begin{align*}
\tr_{\cL G} \pi(f)=\lim_{n\to \infty}\tr_{\cL (G/G_n)}\pi_n(f) .
\end{align*}
\end{lemma}
\begin{proof}
Consider first the case $f\in \Cb G$. Say, $f=\sum_{s\in G} f_s s$ for $f_s\in \Cb$.
Then
\begin{align*}
\tr_{\cL G}\pi(f)=f_{e} \hspace*{3mm}\text{and}\hspace*{3mm} \tr_{\cL (G/G_n)}\pi_n(f)=\sum_{s\in G_n}f_{s}.
\end{align*}
When $n$ is sufficiently large, $G_n\cap \supp(f)\subseteq \{e\}$ and hence $\tr_{\cL (G/G_n)}\pi_n(f)=f_{e}=\tr_{\cL G}\pi(f)$.

Now consider general $f\in C^*(G)$. Let $\varepsilon>0$. Take a $g\in \Cb G$ with $\|f-g\|<\varepsilon$.
Since both $\tr_{\cL G}\circ \pi$ and $\tr_{\cL (G/G_n)}\circ \pi_n$ are states on $C^*(G)$, we have
\begin{gather*}
|\tr_{\cL G}\pi(f)-\tr_{\cL G}\pi(g)|\le \|f-g\|<\varepsilon, \text{ and}\\
|\tr_{\cL (G/G_n)}\pi_n(f)-\tr_{\cL (G/G_n)}\pi_n(g)|\le \|f-g\|<\varepsilon.
\end{gather*}
Therefore, when $n$ is sufficiently large one has
\begin{align*}
|\tr_{\cL G}\pi(f)-\tr_{\cL (G/G_n)}\pi_n(f)|
&\le |\tr_{\cL G}\pi(f)-\tr_{\cL G}\pi(g)| \\
&\hspace*{10mm} \ + |\tr_{\cL (G/G_n)}\pi_n(f)-\tr_{\cL (G/G_n)}\pi_n(g)| \\
&< 2\varepsilon.
\end{align*}
\end{proof}

The following theorem was proved by Deninger and Schmidt \cite[Thm.\ 6.1]{DenSch07}
in the case of invertible $f\in \ell^1(G)$.

\begin{theorem} \label{T-approximate det}
For every invertible $f\in C^*(G)$ one has
\begin{gather*}
\ddet_{\cL G}\pi(f)=\lim_{n\to \infty}\ddet_{\cL (G/G_n)}\pi_n(f).
\end{gather*}
\end{theorem}

\begin{proof}
Given that $0$ is not in the spectrum of $|f|$, by the functional calculus we have
$\kappa (\log |f|) = \log |\kappa (f)|$ for every unital $^*$-homomorphism
$\kappa$ from $C^* (G)$ to another $C^*$-algebra. We thus obtain the result
by applying Lemma~\ref{L-approximate trace} to $\log |f|$.
\end{proof}

For each $n\in \Nb$ denote by $\Fix_{G_n}(X_f)$ the set of points in $X_f$ fixed by $G_n$.
Note that $X_f$ is a compact group and $\Fix_{G_n}(X_f)$ is a subgroup of $X_f$.
Since $\pi_n(f)\in \Zb (G/G_n)$, we can define $X_{\pi_n(f)}$ similarly. Namely,
\begin{align*}
X_{\pi_n(f)}=\big\{h\in (\Rb/\Zb)^{G/G_n}: \pi_n(f)h=0\big\}.
\end{align*}
Note that $X_{\pi_n(f)}$ is a compact group.

\begin{lemma} \label{L-fixed point}
Let $f\in \Zb G$ and $n\in\Nb$. Then there is a natural group isomorphism
$\Phi_n: X_{\pi_n(f)}\rightarrow \Fix_{G_n}(X_f)$ given by $(\Phi_n(h))_s=h_{sG_n}$
for all $h\in X_{\pi_n(f)}$ and $s\in G$.
\end{lemma}

\begin{proof}
It is clear that the formula $(\Psi_n(h))_s=h_{sG_n}$
for $h\in (\Rb/\Zb)^{G/G_n}$ and $s\in G$ defines a group isomorphism $\Psi_n$ from $(\Rb/\Zb)^{G/G_n}$
to the set of $G_n$-fixed points in $(\Rb /\Zb )^G$.
Taking a set $R_n \subseteq G$ of coset representatives for $G/G_n$
and writing $\rho_n : G\to G/G_n$ for the quotient map,
we have, for every $h\in (\Rb/\Zb)^{G/G_n}$ and $s\in G$,
\begin{align*}
(\pi_n (f)h)_{\rho_n (s)} = \sum_{r\in R_n} \bigg( \sum_{t\in G_n} f_{rt} \bigg) h_{\rho_n (r^{-1} s)}
= \sum_{r\in R_n} \sum_{t\in G_n} f_{rt} h_{\rho_n (t^{-1} r^{-1} s)}
= (f\Psi_n (h))_s ,
\end{align*}
so that $f\Psi_n(h) = 0$ if and only if $\pi_n (f)h = 0$. Consequently
we obtain the desired isomorphism $\Phi_n$ by restricting $\Psi_n$.
\end{proof}


Take a finite partition of unity $\cP$ in $C(\Rb/\Zb)$ which generates $C(\Rb/\Zb)$ as a unital $C^*$-algebra.
Via the coordinate map $X_f\rightarrow \Rb /\Zb$ which evaluates at $e$,
we will think of $\cP$ as a partition of unity in $C(X_f)$. Clearly $\cP$ dynamically generates $C(X_f)$,
and so $h_\Sigma (X_f ,G) = h_\Sigma (\cP )$ and $h_{\Sigma ,\mu} (X_f ,G) = h_{\Sigma ,\mu} (\cP )$
for every $G$-invariant Borel probability measure $\mu$ on $X_f$ (see Definitions~\ref{D-finite} and \ref{D-finite top}).
Consider the compatible metric $\rho$ on $\Rb/\Zb$ defined by
$\rho(x, y)=\max_{p\in \cP}|p(x)-p(y)|$ for $x, y\in \Rb/\Zb$. Again, via the coordinate map $X_f\rightarrow \Rb /\Zb$
which evaluates at $e$, we will think of $\rho$ as a continuous pseudometric on $X_f$.

For each $x\in \Fix_{G_n}(X_f)$, we have  a unital homomorphism $\varphi_x: C(X_f)\rightarrow \Cb^{G/G_n}$ determined by
$(\varphi_x(g))(tG_n)=g(tx)$ for all $g\in C(X_f)$ and $t\in G$.
For any
$g\in C(X_f)$ and  $s, t\in G$, we have
\begin{gather*}
\varphi_x(\alpha_{f, s}(g))(tG_n)=\alpha_{f, s}(g)(tx)=g(s^{-1}tx)=(\varphi_x(g))(s^{-1}tG_n)=(\sigma_{n, s}\circ \varphi_x(g))(tG_n).
\end{gather*}
Thus $\varphi_x\circ \alpha_{f, s}=\sigma_{n, s}\circ \varphi_x$ for all $x\in \Fix_{G_n}(X_f)$ and $s\in G$,
and hence $\varphi_x\in \Hom(\cP, F, \delta, \sigma_n)$
for every nonempty finite subset $F$ of $G$ and every $\delta>0$.

Let $\vartheta$ be the compatible metric on $\Rb/\Zb$ defined by
\[
\vartheta(t_1 \mmod \Zb, t_2 \mmod \Zb)=\min_{m\in \Zb}|t_1-t_2-m|
\]
for all $t_1, t_2\in \Rb$.

\begin{lemma} \label{L-algebraic lower bound}
Let $f\in \Zb G$ be invertible in $C^*(G)$. Then
\begin{align*}
h_{\Sigma}(\cP)\ge \limsup_{n\to \infty} \frac{1}{|G/G_n|}\log |\Fix_{G_n}(X_f)|.
\end{align*}
\end{lemma}

\begin{proof}
Since both $\rho$ and $\vartheta$ are compatible metrics on $\Rb/\Zb$, there exists an $\varepsilon>0$ such that
$\vartheta(t_1, t_2)<1/\|f\|_1$ for all $t_1, t_2\in \Rb/\Zb$ with $\rho(t_1, t_2)\le \varepsilon$.
Let $F$ be a nonempty finite subset of $G$ and $\delta>0$. Let $n\in\Nb$. We will show that
$N_{\varepsilon}(\Hom(\cP, F, \delta, \sigma_n), \rho_{\cP ,\infty})\ge |\Fix_{G_n}(X_f)|$,
which by Proposition~\ref{P-infinity norm} will imply the result.

By Lemma~\ref{L-fixed point} the map $\Phi_n: X_{\pi_n(f)}\rightarrow \Fix_{G_n}(X_f)$ 
given by $(\Phi_n(h))_s=h_{sG_n}$ is an isomorphism.
Let $x, y\in \Fix_{G_n}(X_f)$. Set $\tilde{x}=\Phi_n^{-1}(x)$ and $\tilde{y}=\Phi_n^{-1}(y)$. Then
\begin{align*}
\rho_{\cP ,\infty}(\varphi_x, \varphi_y)&=\max_{p\in \cP}\|\varphi_x(p)-\varphi_y(p)\|_{\infty}
=\max_{p\in \cP}\max_{s\in G}|p(sx)-p(sy)|\\
&=\max_{s\in G}\rho(sx, sy)=\max_{s\in G}\rho(x_s, y_s)=\max_{s\in G}\rho(\tilde{x}_{sG_n}, \tilde{y}_{sG_n}).
\end{align*}
Suppose that $\rho_{\cP ,\infty}(\varphi_x, \varphi_y)\le \varepsilon$. Then
\begin{align*}
\max_{s\in G}\vartheta (\tilde{x}_{sG_n}-\tilde{y}_{sG_n}, 0 \mmod \Zb)=\max_{s\in G}\vartheta(\tilde{x}_{sG_n}, \tilde{y}_{sG_n})<1/\|f\|_1.
\end{align*}
Take $z\in [-1, 1]^{G/G_n}$ with $\tilde{x}_{sG_n}-\tilde{y}_{sG_n}=z_{sG_n} \mmod \Zb$ and $|z_{sG_n}|=\vartheta (\tilde{x}_{s_G}-\tilde{y}_{sG_n}, 0 \mmod \Zb)$ for all $s\in G$. Then $\pi_n(f)z\in \Zb^{G/G_n}$ and
\[
\|\pi_n(f)z\|_{\infty}\le \|\pi_n(f)\|_1\|z\|_{\infty}\le \|f\|_1\|z\|_{\infty}<1 .
\]
Thus $\pi_n(f)z=0$. Since $\pi_n(f)$ is invertible in $\cL(G/G_n)$, we get $z=0$ and hence
$\tilde{x}=\tilde{y}$. Consequently, $x=y$. Therefore the set $\{\varphi_x:x\in \Fix_{G_n}(X_f)\}$ is an
$(\rho_{\cP ,\infty}, \varepsilon)$-separated
subset of $\Hom(\cP, F, \delta, \sigma_n)$, and hence
$N_{\varepsilon}(\Hom(\cP, F, \delta, \sigma_n), \rho_{\cP ,\infty})\ge |\Fix_{G_n}(X_f)|$.
\end{proof}

\begin{lemma} \label{L-algebraic upper bound}
Let $f\in \Zb G$ be invertible in $C^*(G)$ and let $\mu$ be a $G$-invariant Borel probability measure on $X_f$.
Then
\begin{gather*}
h_{\Sigma, \mu}(\cP)\le \liminf_{n\to \infty} \frac{1}{|G/G_n|}\log |\Fix_{G_n}(X_f)|.
\end{gather*}
\end{lemma}

\begin{proof}
Using Definition~\ref{D-Hom finite} and the observation following it, it suffices to show that
$\bar{h}^{2\varepsilon}_{\Sigma, \mu}(\cP)\le \liminf_{n\to \infty} |G/G_n|^{-1} \log |\Fix_{G_n}(X_f)|$
for every $\varepsilon>0$.

So let $\varepsilon > 0$.
Since both $\rho$ and $\vartheta$ are compatible metrics on $\Rb/\Zb$, there exists an $\eta'>0$ such that
$\eta'<\varepsilon^2/2$ and $\rho(t_1, t_2)<\varepsilon/2$ for all $t_1, t_2\in \Rb/\Zb$ satisfying
$\vartheta(t_1, t_2)\le \sqrt{\eta'}$.

Denote by $F$ the union of $\{e\}$ and the support of $f$ in $G$.
Denote by $\omega$ the coordinate map $X_f\mapsto \Rb/\Zb$ sending $x$ to $x_e$. 
Then $\omega_*(\mu)$ is a Borel probability measure on $\Rb/\Zb$.
Thus there exists a $\xi\in (0, 1)$ with $\omega_*(\mu)(\{\xi \mmod \Zb\})=0$.
Take an $\eta>0$ such that $48|F|\eta\|f\|_1^2<(\eta'/(2\|f^{-1}\|))^2$.
Take a $\kappa>0$ with $\kappa<\eta'/(2\|f\|_1\|f^{-1}\|)$ such that the closed $(\vartheta, \kappa)$-neighborhood Y of $\xi \mmod \Zb$ in $\Rb/\Zb$ has $\omega_*(\mu)$-measure at most $\eta/2$. Take a $g\in C(\Rb/\Zb)$ with $0\le g\le 1$ on $\Rb/\Zb$, $g=1$ on $Y$, and $\omega_*(\mu)(g)<\eta$.
Via $\omega$ we will also think of $g$ as a function on $X_f$.

Since $\cP$ generates $C(\Rb/\Zb)$ as a unital $C^*$-algebra, there exist an $m\in \Nb$ and a
$\tilde{g}$ in the linear span of the set $\cP_{\{e\}, m}$
of products of the form $p_1 \cdots p_j$ where $1\leq j\leq m$ and $p_1 ,\dots , p_j \in\cP$
such that $\|\tilde{g}-g\|_{\infty}<\eta$. Denote by $M$ the sum of the
absolute values of the coefficients of $\tilde{g}$ as a linear combination of elements in $\cP_{\{e\}, m}$.

Take a $\delta>0$ with $16|F|(|\cP|+M)\delta  \|f\|_1^2<(\eta'/(2\|f^{-1}\|))^2$
such that $\vartheta(t_1, t_2)<\kappa$ for all $t_1, t_2\in \Rb/\Zb$ satisfying
$\rho(t_1, t_2)\le \sqrt{\delta}$. Let $n\in\Nb$.
It suffices to show, in the notation of Definition~\ref{D-Hom finite}, that
\[
N_{2\varepsilon}(\Hom^{X_f}_{\mu}(\cP, F, m, \delta,\sigma_n ), \rho_{\cP})\le |\Fix_{G_n}(X_f)| .
\]
In turn it suffices to show that
for every $\psi\in \Hom^{X_f}_{\mu}(\cP, F, m, \delta, \sigma_n )$ there exists
an $x\in \Fix_{G_n}(X_f)$ such that $\rho_{\cP}(\psi, \varphi_x)<\varepsilon$.

Let $\psi \in \Hom^{X_f}_{\mu}(\cP, F, m, \delta, \sigma_n )$.
Let $a\in G/G_n$. The unital homomorphism $f\mapsto \psi(f)(a)$ on $C(X_f)$ is given by evaluation at some point
$y_a$ of $X_f$. Take $\tilde{y}_a\in [\xi, 1+\xi)^G$ such that $(y_a)_s=(\tilde{y}_a)_s \mmod \Zb$ for all $s\in G$.
Then $f\tilde{y}_a\in \Zb^G$ with $\|f\tilde{y}_a\|_{\infty}\le \|f\|_1\|\tilde{y}_a\|_{\infty}\le 2\|f\|_1$.
Write $z$ for the element of $\Zb^{G/G_n}$ given by $z_a=(f\tilde{y}_a)_e$ for all $a\in G/G_n$.
Define $z'\in (\Rb/\Zb)^{G/G_n}$ by $z'_a=(\pi_n(f)^{-1}z)_a \mod \Zb$ for all $a\in G/G_n$. Then
$z'\in X_{\pi_n(f)}$.
Set $x=\Phi_n(z')$ where $\Phi_n$ is the isomorphism
$X_{\pi_n(f)}\rightarrow \Fix_{G_n}(X_f)$ from Lemma~\ref{L-fixed point}.
We claim that $\rho_{\cP}(\psi, \varphi_x)<\varepsilon$.

Define $u\in (\Rb/\Zb)^{G/G_n}$ and
$\tilde{u}\in [\xi, \xi+1)^{G/G_n}$ by $u_a=(y_a)_e$
and $\tilde{u}_a=(\tilde{y}_a)_e$ for all $a\in G/G_n$. Also, set $v=\pi_n(f)\tilde{u}\in [-2\|f\|_1, 2\|f\|_1]^{G/G_n}$.

Let $p\in \cP$ and $s\in F$. Then
\begin{align*}
\|\sigma_s\circ \psi(p)-\psi\circ \alpha_{f, s}(p)\|_2
&=\bigg(\frac{1}{|G/G_n|}\sum_{a\in G/G_n}|\psi(p)(s^{-1}a)-\alpha_{f, s}(p)(y_a)|^2 \bigg)^{1/2}\\
&=\bigg(\frac{1}{|G/G_n|}\sum_{a\in G/G_n}|p(y_{s^{-1}a})-p(s^{-1}y_a)|^2 \bigg)^{1/2}\\
&=\bigg(\frac{1}{|G/G_n|}\sum_{a\in G/G_n}|p(u_{s^{-1}a})-p((y_a)_{s^{-1}})|^2 \bigg)^{1/2}.
\end{align*}
Since $\|\sigma_s\circ \psi(p)-\psi\circ \alpha_{f, s}(p)\|_2<\delta$, the set of all $a\in G/G_n$ satisfying $|p(u_{s^{-1}a})-p((y_a)_{s^{-1}})|\ge \sqrt{\delta}$ has cardinality at most $\delta|G/G_n|$.
Thus the set $W$ of all $a\in G/G_n$ satisfying $|p(u_{s^{-1}a})-p((y_a)_{s^{-1}})|<\sqrt{\delta}$ for all $p\in \cP$
and $s\in F$ has cardinality at least $|G/G_n|-\delta|\cP||F||G/G_n|$.

We have
\begin{align*}
\zeta \circ \psi(g)-\mu(g)&=\frac{1}{|G/G_n|}\sum_{a\in G/G_n}g(y_a)-\mu(g) \\
&=\frac{1}{|G/G_n|}\sum_{a\in G/G_n}g(u_a)-\omega_*(\mu)(g)\\
&\ge \frac{1}{|G/G_n|}\sum_{a\in G/G_n}g(u_a)-\eta,
\end{align*}
and
\begin{align*}
|\zeta \circ \psi(g)-\mu(g)|&\le |\zeta \circ \psi(g)-\zeta \circ \psi(\tilde{g})|+|\zeta \circ \psi(\tilde{g})-\mu(\tilde{g})|+|\mu(\tilde{g})-\mu(g)|\\
&\le \|g-\tilde{g}\|_{\infty}+M\delta+\|g-\tilde{g}\|_{\infty}<2\eta+M\delta.
\end{align*}
Then the set of all $a\in G/G_n$ satisfying $u_a\in Y$ has cardinality at most $(3\eta+M\delta)|G/G_n|$.
Thus the set $V$ of all $a\in G/G_n$ satisfying $u_{s^{-1}a}\not \in Y$ for all $s\in F$ has cardinality at least
$|G/G_n|-|F|(3\eta+M\delta)|G/G_n|$.

Let $a\in W\cap V$ and $s\in F$. Since $a\in V$, one has $\vartheta(u_{s^{-1}a}, \xi \mod \Zb)>\kappa$,
and hence $\tilde{u}_{s^{-1}a} \in (\xi+\kappa, 1+\xi-\kappa)$.
As $a\in W$, one has
$\rho(u_{s^{-1}a}, (y_a)_{s^{-1}})=\max_{p\in \cP}|p(u_{s^{-1}a})-p((y_a)_{s^{-1}})|<\sqrt{\delta}$, and hence
$\vartheta( u_{s^{-1}a}, (y_a)_{s^{-1}})<\kappa$.
It follows that
$|\tilde{u}_{s^{-1}a}-(\tilde{y}_a)_{s^{-1}}|=\vartheta( u_{s^{-1}a}, (y_a)_{s^{-1}})<\kappa$.
Then
\begin{align*}
|v_a-z_a|&=\bigg|\sum_{s\in F}f_s\tilde{u}_{s^{-1}a}-\sum_{s\in F}f_s(\tilde{y}_a)_{s^{-1}}\bigg|
\le \sum_{s\in F}|f_s|\cdot |\tilde{u}_{s^{-1}a}-(\tilde{y}_a)_{s^{-1}}| \\
&<|f|_1\kappa<\frac{\eta'}{2\|f^{-1}\|}.
\end{align*}

Now we have
\begin{align*}
\|v-z\|_2&=\bigg(\frac{1}{|G/G_n|}\sum_{a\in W\cap V}|v_a-z_a|^2+\frac{1}{|G/G_n|}
\sum_{a\in (G/G_n)\setminus (W\cap V)}|v_a-z_a|^2\bigg)^{1/2}\\
&\le \bigg(\bigg(\frac{\eta'}{2\|f^{-1}\|}\bigg)^2
+\frac{|(G/G_n)\setminus (W\cap V)|}{|G/G_n|} \cdot 16\|f\|^2_1\bigg)^{1/2}\\
&\le \bigg(\bigg(\frac{\eta'}{2\|f^{-1}\|}\bigg)^2
+16|F|(3\eta+(|\cP|+M)\delta)\|f\|^2_1\bigg)^{1/2} \\
&<\frac{\eta'}{\|f^{-1}\|},
\end{align*}
and hence
\begin{align*}
\|\tilde{u}-\pi_n(f)^{-1}z\|_2\le \|\pi_n(f)^{-1}\|\cdot \|v-z\|_2\le \|f^{-1}\|\cdot \frac{\eta'}{\|f^{-1}\|}
=\eta'.
\end{align*}
Then the set $W'$ of all $a\in G/G_n$ satisfying $|\tilde{u}_a-(\pi_n(f)^{-1}z)_a|<\sqrt{\eta'}$ has cardinality
at least $|G/G_n|(1-\eta')\ge |G/G_n|(1-\varepsilon^2/2)$. For every $a\in W'$, one has $\vartheta(u_a, z'_a)\le |\tilde{u}_a-(\pi_n(f)^{-1}z)_a|<\sqrt{\eta'}$,
and hence $\rho(u_a, z'_a)<\varepsilon/2$.

For each $a\in G/G_n$ take an $s_a\in G$ such that $a=s_aG_n$.
For every $p\in \cP$ we have
\begin{align*}
\|\psi(p)-\varphi_x(p)\|_2&=\bigg(\frac{1}{|G/G_n|}\sum_{a\in G/G_n}|p(y_a)-p(s_ax)|^2\bigg)^{1/2}\\
&=\bigg(\frac{1}{|G/G_n|}\sum_{a\in G/G_n}|p(u_a)-p(x_{s_a})|^2\bigg)^{1/2}\\
&=\bigg(\frac{1}{|G/G_n|}\sum_{a\in G/G_n}|p(u_a)-p(z'_{s_aG_n})|^2\bigg)^{1/2}\\
&=\bigg(\frac{1}{|G/G_n|}\sum_{a\in G/G_n}|p(u_a)-p(z'_a)|^2\bigg)^{1/2}\\
&=\bigg(\frac{1}{|G/G_n|}\sum_{a\in W'}|p(u_a)-p(z'_a)|^2 \\
&\hspace*{30mm}\ +\frac{1}{|G/G_n|}\sum_{a\in (G/G_n)\setminus W'}|p(u_a)-p(z'_a)|^2\bigg)^{1/2}\\
&\le \bigg(\frac{1}{|G/G_n|}\sum_{a\in W'}\rho(u_a, z'_a)^2
+\frac{1}{|G/G_n|}\sum_{a\in (G/G_n)\setminus W'}1\bigg)^{1/2}\\
&\le \bigg(\frac{\varepsilon^2}{4}+\frac{\varepsilon^2}{2}\bigg)^{1/2}<\varepsilon.
\end{align*}
Therefore
$\rho_{\cP}(\psi, \varphi_x)=\max_{p\in \cP}\|\psi(p)-\varphi_x(p)\|_2<\varepsilon$.
\end{proof}

We are now ready to prove Theorem~\ref{T-algebraic}.

\begin{proof}[Proof of Theorem~\ref{T-algebraic}]
By \cite[Thm.\ 3.2]{Li10}, for each $n\in \Nb$ we have
\[\log \ddet_{\cL (G/G_n)}\pi_n(f)=\frac{1}{|G/G_n|}\log |X_{\pi_n(f)}|.\]
It follows by Theorem~\ref{T-approximate det} and Lemma~\ref{L-fixed point} that
\begin{align*}
\log \ddet_{\cL G}f&=\lim_{n\to \infty}\log \ddet_{\cL (G/G_n)}\pi_n(f)\\
&=\lim_{n\to \infty}\frac{1}{|G/G_n|}\log |X_{\pi_n(f)}|\\
&=\lim_{n\to \infty}\frac{1}{|G/G_n|}\log |\Fix_{G_n}(X_f)|.
\end{align*}
The theorem now follows from Lemmas~\ref{L-algebraic lower bound} and \ref{L-algebraic upper bound} and Theorem~\ref{T-variational}.
\end{proof}

Note that if we take $f$ to be $k$ times the unit in $\Zb G$ for some $k\in\Nb$, then the
action of $G$ on $X_f$ is the Bernoulli shift on $k$ symbols, whose
entropy was computed more generally in Example~\ref{E-Bernoulli} to be $\log k$ for any countable sofic $G$
and sofic approximation sequence $\Sigma$.

In the case of a countable discrete amenable group $G$ acting by automorphisms on a compact metrizable group $K$,
one can show directly that the topological entropy is equal to the measure entropy with
respect to the normalized Haar measure (this is done in \cite{Berg69} for $G=\Zb$ by an
argument that works more generally, and in \cite{Den06}). In our present context,
it follows from Theorem~\ref{T-algebraic} and \cite{Bowen10b} that when $f$ is invertible in $\ell^1 (G)$ we also
have $h_{\Sigma ,\mu} (X_f ,G) = h_\Sigma (X_f ,G)$ where $\mu$ is the normalized Haar measure on $X_f$.
However we do not see how to prove this in a more direct and general way. We thus ask the following.

\begin{problem}
Let $G$ be a countable sofic group acting by automorphisms on a compact metrizable group $K$.
Let $\Sigma$ be a sofic approximation sequence for $G$. Is it true in general that
$h_{\Sigma ,\mu} (K,G) = h_\Sigma (K,G)$ where $\mu$ is the normalized Haar measure on $K$?
What if $G$ is residually finite and $\Sigma$ is assumed to arise from a sequence of finite quotients?
Does equality hold for the type of actions studied in this section without the
assumption that $G$ is residually finite or that $\Sigma$ arises from a sequence of finite quotients?
\end{problem}

\end{document}